\documentclass[11pt]{amsart}
\usepackage{amsmath}
\usepackage{amsthm}
\usepackage{amsfonts}
\usepackage{amssymb}
\usepackage{verbatim}
\usepackage{graphicx}
\usepackage[margin=1.0in]{geometry}
\usepackage{url}
\usepackage{enumerate}
\usepackage[pdftex,bookmarks=true]{hyperref} 	
\newcommand{\norm}[1]{\left\lVert #1 \right\rVert}
\newcommand{\ab}[1]{\left| #1 \right|}

\usepackage{color}

\newcommand\R{{\mathbb{R}}}
\newcommand\C{{\mathbb{C}}}
\newcommand\N{{\mathbb{N}}}
\newcommand\Z{{\mathbb{Z}}}

\renewcommand\P{{\mathbf{P}}}
\newcommand\E{{\mathbf{E}}}

\newcommand\Var{\mathbf{Var}}
\newcommand\Cov{\mathbf{Cov}}
\renewcommand\Im{{\operatorname{Im}}}
\renewcommand\Re{{\operatorname{Re}}}

\newcommand\CE{{\mathcal E}}

\newcommand\CP{{\mathcal P}}
\newcommand\CQ{{\mathcal Q}}
\newcommand\CS{{\mathcal S}}

\newcommand\ep{{\varepsilon}}

\theoremstyle{plain}
 \newtheorem{theorem}{Theorem}[section]

 \newtheorem{proposition}[theorem]{Proposition}
 
 \newtheorem{lemma}[theorem]{Lemma}
 \newtheorem{cor}[theorem]{Corollary}

\theoremstyle{definition}
 \newtheorem{definition}[theorem]{Definition}
 
 \newtheorem{remark}[theorem]{Remark}
\def\Xint#1{\mathchoice
{\XXint\displaystyle\textstyle{#1}}%
{\XXint\textstyle\scriptstyle{#1}}%
{\XXint\scriptstyle\scriptscriptstyle{#1}}%
{\XXint\scriptscriptstyle\scriptscriptstyle{#1}}%
\!\int}
\def\XXint#1#2#3{{\setbox0=\hbox{$#1{#2#3}{\int}$ }
\vcenter{\hbox{$#2#3$ }}\kern-.6\wd0}}

\def\dashint{\Xint-}

\begin{document}

\title { Roots of random functions:\\ A framework for local universality}

\author{Oanh Nguyen} 
\address{Department of Mathematics, University of Illinois at Urbana-Champaign, IL 61801, USA}
\email{onguyen@illinois.edu}

\author{Van Vu}
\address{Department of Mathematics, Yale University, New Haven, CT 06520, USA}
\email{van.vu@yale.edu}

\thanks{Part of this work was done at VIASM (Hanoi), and the authors would like to thank the institute for its support and hospitality; V.~Vu is supported by NSF grant DMS-1307797 and AFORS grant FA9550-12-1-0083. Nguyen is supported by NSF grant DMS-2125031.}
\begin{abstract} 
We investigate  the local distribution of roots of random functions of the form $F_n(z)= \sum_{i=1}^n \xi_i \phi_i(z) $, where $\xi_i$ are independent random variables and $\phi_i (z) $ are arbitrary analytic functions. Starting with the fundamental works of Kac and Littlewood-Offord in the 1940s, random functions of this type have been studied extensively in many fields of mathematics.

\vskip2mm 

We develop a robust framework to solve the problem by reducing, via universality theorems, 
 the calculation of the distribution of the roots
 and the interaction between them to the case where $\xi_i$ are Gaussian. In
 this special case, one can use the Kac-Rice formula and various other tools to obtain precise answers.

 \vskip2mm 
 
 Our framework has a wide range of applications, which include the most  popular  models of random functions, such 
 as random trigonometric polynomials and all basic classes of random algebraic polynomials (Kac, Weyl, and elliptic). 
 Each of these ensembles has been studied heavily by deep and diverse methods. Our method, for the first time, provides a unified treatment for all of them.

 \vskip2mm 
 
 Among the applications, we derive the first 
 local universality result for  random trigonometric polynomials with arbitrary coefficients. When restricted to 
 the study of real roots, this result extends several  recent results, proved for less general 
 ensembles.   For random algebraic polynomials, we strengthen 
 several recent results of Tao and the second author, with significantly simpler proofs. As a corollary, we sharpen a classical result of Erd\H{o}s and Offord on real roots of Kac polynomials, providing 
 an optimal error estimate. Another application is a refinement of a recent result of Flasche and Kabluchko on 
 the roots of random Taylor series.

\end{abstract}

\maketitle

\section{Introduction} Let $n$ be a positive integer or $\infty$. Let $\phi_1, \dots, \phi_n $ be deterministic functions and $\xi_1, \dots, \xi_n $ be independent random variables. Consider the random function/series
\begin{equation}\label{F}
	F_n = \sum_{i = 1}^{n} \xi_i\phi_i.
\end{equation}
A fundamental task is to understand the distribution of and the interaction between 
 the roots (both real and complex) of $F_n$. For several decades, this task has been carried out in many different areas of mathematics such as analysis, numerical analysis, probability, mathematical physics; see \cite{BS, EK, Far, forrester1999exact, HKPV, kahane1985, sodin2005zeroes, zelditch2001random}, for example.

The most studied subcases are when $\phi_i = c_i x^i$ (in which case $F_n$ is a random algebraic polynomial) and $\phi_i =c_i \cos ix $ (in which case $F_n$ is a random trigonometric polynomial); here and later, the $c_i$ are deterministic coefficients that may depend on $i$ and $n$. In fact, these classes split further, according to the values of $c_i$. For instance, three important classes of random algebraic polynomials are: Kac polynomials $\left (c_i=1\right )$, Weyl polynomials $\left (c_i= \frac{1}{ \sqrt {i!}}\right )$ and elliptic polynomials 
$\left (c_i= \sqrt { {n \choose i }}\right )$. For random trigonometric polynomials, most papers seem to focus on the case $c_i=1$. 
A very significant part of the literature on random functions focuses on these special classes.

Even for these classical cases, the problem is already hard; see \cite{angstpoly, azais2015local, DHV, flasche, iksanov2016local, kabluchko2014asymptotic, pritsker1, soze1, soze2, TVpoly} for a partial list of recent developments. It requires a full book to 
discuss the results and methods concerning random polynomials, but one feature stands out. 
The distributions of the roots in different classes are quite different, and the methods to study them are often specialized.

In this paper, we aim to develop a robust framework to solve the general problem. The leading idea is to utilize universality theorems to reduce the problem of calculating the distribution of the roots
and the interaction between them to the case where the $\xi_i$ are Gaussian. 
In the Gaussian case, 
the answers can be (or, for most ensembles, have already been) computed in a precise form, using the Kac-Rice formula and various other tools which make use of special properties of Gaussian random variables and Gaussian processes; see, for instance \cite{EK, GKZ, HKPV, prosen1996, MP, sodin2005zeroes, TVpoly}. 
In particular, when the $\xi_i$ are complex Gaussian variables, $F_n$ is called a Gaussian analytic function, and 
we refer to Sodin's paper \cite{sodin2005zeroes} for an in-depth survey.

 Universality theorems of this type have 
recently been proved in \cite{DOV, TVpoly} by the authors, Do and Tao for many classes of random algebraic polynomials of various types, using complex machinery (see also \cite{KZ3, ledoan2012universality, pritsker1, pritsker2, starr2011universality} for related works concerning 
global universality).  The method built in these papers is sensitive. It does not apply to random trigonometric polynomials and many  other ensembles. 

In this paper, we are going to establish a new and general condition which guarantees universality for a  wide class of random functions. 
This class  contains all popular random functions. Among others,  it covers all classical random algebraic polynomials (such as those considered in \cite{DOV, TVpoly}  and many others). 
Quite remarkably, it also covers  random trigonometric polynomials with 
general coefficients, whose behavior is totally different. (For readers not familiar with the theory of random functions, let us point out that 
random trigonometric polynomials typically have 
$\Theta (n)$ real roots while Kac polynomials have only $\Theta (\log n) $.)

We would like to emphasize the simplicity and robustness of our approach.  
Proofs of local universality results have been, so far, considerably complex and long. Furthermore,  different ensembles require proofs which are different in at least a few key technical aspects.  Our proofs, based on new observations, are quite simple and robust. The proof for the general theorem is only a few pages long. 
 Next, and more importantly,   we can deduce universality results for completely different ensembles of random functions from this general theorem in an identical way using (essentially) one simply stated lemma. 
In each ensemble considered, we either  obtain completely  new results or a short,  new proof of the most current result, many times with a quantitative improvement.   The length of the paper is due 
to the number of applications.  The reader is invited to read Section \ref{mainideas} for a 
discussion of our method and a comparison with the previous ones.

Let us now briefly discuss the applications. 
 Consider two random functions 
$F_n = \sum_{i = 1}^{n} \xi_i\phi_i$ and $\tilde F_n = \sum_{i = 1}^{n} \tilde \xi_i\phi_i$, where $\xi_i$ and $\tilde \xi_i$ can have different distributions. We show (under some mild assumptions) that 
the local statistics of the roots of the two functions are asymptotically the same. In practice, we can set $\tilde \xi_i$ to be Gaussian, and thus reduce the study to this case. The local information can be used to derive 
certain global properties; for instance, the number of roots in a large region (which has been partitioned into many local cells) is simply the sum of the numbers of roots in each cell.

\begin{itemize}

\item We study random trigonometric polynomials in Section \ref{app1}. We derive (to the best of our knowledge) the first local universality of correlation 
for this class. Our setting is more flexible than most 
previous works on this topic, as we allow a large degree of freedom in choosing the deterministic coefficients $c_i$.

While we do not find comparable previous local universality results for random trigonometric polynomials, we can still make some comparisons to previous works by restricting to the popular sub-problem of estimating the density of the real roots. For this problem, our universality result yields new estimates which extend several existing results, some of which are quite recent and have been proved by totally different methods; see 
 Section \ref{app1} for details.

\item In Section \ref{app2}, we discuss Kac polynomials. We derive a short proof for a 
strengthening of a recent result of Tao and the second author \cite{TVpoly}. By almost the same argument, one could also recover the main result of Do and the authors \cite{DOV} which applies for generalized Kac polynomials. 
As a corollary, we obtain a more precise version of the classical result of Erd\H{o}s and Offord \cite{EO} on the number of real roots.

\item In Section \ref{app3}, we study Weyl series. Our universality result here provides 
an exact estimate for the expectation of the number of roots in any fixed domain $B$. Previous to our result, such 
an estimate was only known for sets of the form 
$rB$, where $r$ is a parameter tending to infinity, thanks to a very recent work of 
Kabluchko and Zaporozhets \cite{kabluchko2014asymptotic}. 

\item In Section \ref{app4}, we apply our results to random elliptic polynomials. 
We give a short proof of a recent result from \cite{TVpoly}, which generalizes an earlier result of Bleher and Di \cite{BD}.

\item The above applications already cover all traditional   classes of random functions in the literature. To illustrate the generality of our result, in Section \ref{app5}, we 
present one  more application, concerning random series with regularly varying coefficients, a class 
defined and studied by Flasche and Kabluchko very recently \cite{flasche2020expected}.

\item While revising this paper, we became aware of a recent work \cite{flasche2018real} which has some overlaps with ours. We made a brief comparison at the end of  Section \ref{app5}. 

\item  Additionally, after this work had been announced, the framework that we develop here has been applied to the following papers.
\begin{itemize}
	\item In \cite{CNNV}, Mei-Chu Chang, Hoi Nguyen and the authors study  the number of intersections between random
	eigenfunctions of general eigenvalues and a given smooth curve in flat tori.
	\item In \cite{do2020random}, Yen Do and the authors study random orthonormal polynomials.
\end{itemize} 

\end{itemize}

In most applications,  we will  work out corollaries concerning the problem of  counting real roots. While our results yield much more than just the density 
function of real roots, we focus on this subproblem since it  is,  traditionally, one of the most natural and appealing problems in the field. (Technically speaking, the study of zeros of random analytic functions started with papers of Littlewood-Offord and Kac in the 1940s, studying the number of real roots of Kac polynomials.)  Our corollaries provide  many new contributions to the existing vast literature on this subject. As a matter of fact, our results allow us to study
 any level set $L_a:= \{z\in \C: F_n (z)=a \} $ for any fixed $a$ (the roots form the level set $L_0$) at no extra cost. 

The rest of the paper is organized as follows. In the next section, we first describe our goal, namely, what we mean by universality. 
We then establish the general condition that guarantees 
universality, and comment on its strength. We next state the general universality theorems along with a discussion of the main ideas 
in the proof. 

The next 5 sections (Sections \ref{app1} - \ref{app5})  are devoted to the applications mentioned above. We state universality theorems for various classes of random functions, and derive corollaries 
concerning the density of both real and complex roots. In Section \ref{app1_proof_1}, we prove the general universality theorems stated in Section \ref{framework}. The rest of the paper is devoted to the verification of the 
applications in Sections \ref{app1} - \ref{app5}. We also include a short appendix at the end of the paper, which contains the proofs of a few  lemmas (some of which were proved elsewhere), for the sake of completeness.

\section {Universality theorems}\label{framework}

In the first subsection, we describe the traditional way to compare local statistics of the roots. 
Next, we provide the assumptions under which our theorems hold, and comment on their strength. The precise statements 
come in the final subsection.

The notation ${\bf 1}_E$ denotes the indicator of an event $E$; it takes value 1 if $E$ holds and $0$ otherwise.

\subsection{Comparing local statistics} 

For simplicity, let us first focus on the complex roots of $F_n$.
 These roots form a random point set on the plane.

The first interesting local statistics is the density. In order to understand the density 
around a point $z$, we consider the unit disk $ B(z, 1)$ centered at $z$. In practice, the radius of the disk is chosen so that 
the number of roots in it is typically of order $\Theta (1)$. 
The expected number of roots in the disk can be written as 
$$\sum_{i } \E f(\zeta_i ) $$ where $\zeta_1, \zeta_2, \dots$ are the roots of $F_n$, and $f$ is the indicator 
function of $B(z,1)$; in other words, $f(x)=1 $ if $x \in B(z,1)$ and zero otherwise. 

If one is interested in the pairwise correlation between the roots near $z$, then it is natural to look at 
$$\sum_{i, j} \E f(\zeta_i, \zeta_j ) $$ where $f(x,y)$ is the indicator 
function of $B(z,1)^2 :=B(z,1) \times B(z,1)$; in other words, $f(x,y)=1 $ if both $x, y \in B(z,1)$ and zero otherwise.

In general, the $k$-wise correlation can be computed from 
$$\sum_{ i_1, \dots, i_k } \E f(\zeta_{i_1}, \dots, \zeta_{i_k} ) $$ where $f(x_1, \dots, x_k)$ is the indicator 
function of $B(z,1)^k$. A good estimate for these quantities tells us how the nearby roots repel or attract each other. 

Even more generally, one can study the interaction of roots near different centers by looking at 
$$\sum_{ i_1, \dots, i_k } \E f(\zeta_{i_1}, \dots, \zeta_{i_k} ) $$ where $f(x_1, \dots, x_k)$ is the indicator function of $B(z_1,1) \times B(z_2,1) \dots \times B(z_k,1)$ with $B(z_i, 1)$ being the unit disk centered at $z_i$. 

Now, consider another random function
$$\tilde F_n = \sum_{i = 1}^{n} \tilde \xi_i\phi_i$$
where the $\tilde \xi_i$ are independent random variables distributed differently from the $\xi_i$. We end up with two sets of quantities $$\sum_{ i_1, \dots, i_k} \E f\left (\zeta_{i_1}, \dots, \zeta_{i_k} \right ) $$ and $$ \sum_{ i_1, \dots, i_k } \E f(\tilde\zeta_{i_1}, \dots, \tilde \zeta_{i_k} )$$
where the $\tilde \zeta_i$ are the roots of $\tilde F_n$. 

We would like to show (under certain assumptions) that these two quantities are asymptotically the same, namely
 \begin{equation} \label{asymp} \left | \sum_{ i_1, \dots, i_k} \E f(\zeta_{i_1}, \dots, \zeta_{i_k} ) - \sum_{ i_1, \dots, i_k } \E f\left (\tilde\zeta_{i_1}, \dots, \tilde \zeta_{i_k} \right ) \right | \le \delta_n \end{equation} for some $\delta_n$ tending to zero as $n$ goes to infinity.

 For technical convenience, we will replace the indicator function $f$ by a smoothed approximation. This makes no difference in applications. On the other hand, our results hold for any smoothed test function $f$, which may have nothing to do with the indicator function. 
 
 If one cares about the real roots, one  replaces the  disk $B(z,1)$ by the interval of length 1 centered at a real number $z$. In general, instead of the product $B(z_1,1) \times B(z_2,1) \dots \times B(z_k,1)$, one can consider a mixed product of disks and intervals. This enables one to understand the interaction between nearby roots of both types (complex and real).
 
 One, of course, could have made the previous discussion using the notion of correlation functions. However, we find the current format direct and intuitive. We refer to \cite{HKPV} or \cite{TVpoly} for more detailed discussions concerning local statistics using 
 correlation functions.

 \subsection{Assumptions} \label{condi} 
 
 Before stating the result, let us discuss the assumptions. There are two types of assumptions. The first is for the random variables $\xi_i$ and $\tilde \xi_i$. The second concerns the deterministic functions $\phi_i$.

For the random variables, our assumption is close to minimal. In the case that both $\xi_i$ and $\tilde \xi_i$ are real, our simplest assumption is

{\bf Condition C0.} The random variables $\xi_1, \dots, \xi_n, \tilde \xi_1, \dots, \tilde \xi_n$ are independent real random variables 
with the same mean $\E \xi_i = \E \tilde \xi_i$ for each $i$, variance one, and (uniformly) bounded $(2+\ep)$ central moments, for some constant $0<\ep <1$.

In fact, we can relax the assumption of matching means and variances, allowing a finite number of exceptions. If the $\xi_i$ and $\tilde \xi_i$ are complex, the matching mean and variance need to be adjusted to address both real and imaginary parts.

{\bf Condition C1.} Two sequences of random variables $(\xi_1, \dots, \xi_n)$ and $(\tilde \xi_1, \dots, \tilde \xi_n)$ are said to satisfy this condition if the following hold, for some constants $N_0, \tau >0$  and $0<\ep<1$. 

\begin{enumerate} [(i)]
	
	\item {\it Uniformly bounded $(2+\ep)$ central moments:} \label{cond-moment} The random variables $\xi_i$ (and similarly $\tilde \xi_i$), $1\le i \le n$,  are independent (real or complex, not necessarily identically distributed) random variables with unit 
	variance (namely, $\E|\xi_i - \E \xi_i|^{2} = 1$), and bounded $(2+\ep)$ central moments, namely $\E\left |\xi_i - \E\xi_i\right |^{2+\ep} \le \tau$.

	\item {\it Matching moments to second order with finitely many exceptions:}\label{cond-matching} For any $i\ge N_0$, for all $a, b\in \{0, 1, 2\}$ with $a+b\le 2$, $$\E \Re\left (\xi_i\right )^{a}\Im \left (\xi_i\right )^{b} = \E \Re\left (\tilde \xi_i\right )^{a}\Im \left (\tilde \xi_i\right )^{b},$$ and for $0\le i< N_0$, $\left | \E \xi_i -\E \tilde \xi_i\right |\le \tau$. 	
	
	\end{enumerate}

It is trivial that Condition {\bf C1} contains Condition {\bf C0} as a special case. 
We find it rewarding to go with the more general, but slightly technical, assumption \eqref{cond-matching}, which allows non-matching means, as it 
leads to an interesting phenomenon that changing a finite number of terms in $F_n(z)$ does not influence the asymptotic distribution of the roots. 
Among other benefits, this allows us to generalize all results to level sets $\{ z\in \C: F_n (z) =a \} $ for any fixed $a$; see Remark \ref{rmk1} for more details.

\vskip2mm 

We now turn to the assumption on the deterministic functions $\phi_i$. 
The statement of our theorems will involve two parameters, an error term $0<\delta_n<1$ (see \eqref{asymp}) 
and a region $D_n\subset \C$, from which the base points $z_1, \dots, z_k$ are chosen. 
As their subscripts indicate, both $\delta_n$ and $D_n$ can depend on $n$. In most of our applications,
$\delta_n$ tends to zero with $n$ but it is not required. When $n = \infty$ for example, $\delta_\infty$ can be any parameter in $(0, 1)$.
The assumptions below are tailored to these two parameters, $\delta_n$ and $D_n$. 

For two sets $\mathcal A, \mathcal B\subset \C$, define  $\mathcal A+\mathcal B: = \{a+b: a\in \mathcal A, b\in \mathcal B\}$.
Let $k, C_1, \alpha_1, A, c_1, C$ be positive constants. We say that $F_n$ satisfies Condition {\bf C2} with parameters 
$(k, C_1, \alpha_1, A,c_1, C) $ if the following holds.

{\bf Condition C2.}
\begin{enumerate}
	
	\item \label{cond-poly} For any $z\in D_n$, $F_n$ is analytic on the disk $B(z, 2)$ with probability 1 and 
	$$\E N^{k+2}\textbf{1}_{N\ge \delta_n^{-C_1}} \le C,$$
	where $N$ is the number of zeros of $F_n$ in the disk $B(z, 1)$. We note that throughout this paper, if $F_n$ is identically 0, we adopt the (admittedly artificial) convention that $F_n$ has no roots in $\C$.
	
	\item {\it Anti-concentration:}\label{cond-smallball} For every $z\in D_n$, with probability at least $1 - C\delta_n^{A}$, there exists $z'\in B (z, 1/100)$ for which $|F_n(z')|\ge \exp(-\delta_n^{-c_1})$.
	
	\item {\it Boundedness:} \label{cond-bddn} For any $z\in D_n$, with probability at least $1 - C\delta_n ^{A}$, $|F_n(w)|\le \exp(\delta_n^{-c_1})$ for all $w\in B (z, 2)$.
	
	\item {\it Delocalization:}\label{cond-delocal} For every $z\in D_n+B (0, 1)$, it holds that $\sum _{j = 1}^{n}|\phi_j(z)|^{2}\neq 0$ and for every $i = 1, \dots, n$, 
	$$\frac{|\phi_i(z)|}{\sqrt{\sum _{j = 1}^{n}|\phi_j(z)|^{2}}}\le C\delta_n^{\alpha_1}.$$ 
	
	\item {\it Derivative growth:}\label{cond-repulsion} For any real number $x\in D_n + B(0, 1)$,
	\begin{equation}
		\sum_{j=1}^{n} |\phi_j'(x)|^{2}\le C\delta_n ^{-c_1}\sum_{j=1}^{n} |\phi_j(x)|^{2},\nonumber
	\end{equation}
	\begin{equation}
		\sum_{j=1}^{n} \sup_{z\in B(x, 1)}|\phi_j''(z)|^{2}\le C\delta_n ^{-c_1}\sum_{j=1}^{n} |\phi_j(x)|^{2},\nonumber
	\end{equation}
	and
	\begin{equation}
		\sum_{j=1}^{n} |\E \xi_j|\sup_{z\in B(x, 1)}|\phi_j''(z)|\le C\delta_n ^{-c_1}\sqrt{\sum_{j=1}^{n} |\phi_j(x)|^{2}}.\nonumber
	\end{equation}
	
\end{enumerate}

 \begin{remark}
	While Condition {\bf C2} still involves the random variables $\xi_i$, in the verification of these conditions, we only need to use basic information about the mean of these variables. On the other hand, the type of arguments one needs to use in the verification depends strongly on the functions $\phi_i$. 
	\end{remark}

\begin{remark}

	The last Condition {\bf C2} \eqref{cond-repulsion} is important only in the study of real roots; in particular, it is used to prove the repulsion of the real roots (Lemma \ref{lmrepulsion}). It can be ignored in the study of complex roots. 
\end{remark}

Let us now comment on the verification of Condition {\bf C2} in practice. 

\begin{remark} \label{anticond} 
Typically, we assume $\delta_n$ tends to zero with $n$. We transform the functions so that the expectation of $N$ is of order $1$ where $N$ is the number of roots of $F_n$ in a disk $B(z, 1)$, $z\in D_n$. With this in mind, the first condition is a large deviation estimate on $N$ and can be proved using standard 
large deviation tools combined with classical complex analytic estimates such as Jensen's inequality. The third condition (boundedness) is also a large deviation statement and can be dealt with using standard tools, since for any fixed $w$, $F_n (w) $ is a sum of independent random variables. 

The two Conditions  {\bf C2} \eqref{cond-delocal} and  {\bf C2} \eqref{cond-repulsion} are deterministic properties of the functions $\phi_i$ and hold for many natural classes of functions. The forth condition (delocalization) simply says that 
in the vector $(\phi_i (z))_1^n $, no coordinate dominates. The fifth condition asserts that the first and second derivatives of $\phi_i$ do not exceed the value of the function itself by a large multiplicative factor, in an average sense. Checking these conditions is usually a routine task. 
Furthermore, the proof allows us to easily modify these conditions, if necessary. 

The second (anti-concentration) condition is the one that may require some work. However, this condition is trivial if (some of) the random variables $\xi_i$ have continuous distributions with bounded density. For instance, if $\phi_1=1$ (constant function) and $\xi_1$ has a continuous distribution with bounded density, then the required anti-concentration property  holds trivially by conditioning on the rest of the random variables (which can have arbitrary distributions). There is a sizable literature 
focusing on continuous ensembles, and  our results allow us to recover, in a straightforward  manner, a number of
existing  results, whose original proofs were quite technical; see Sections \ref{app2} and \ref{app4} for examples. 
	\end{remark}

\subsection{Results} 

Given the assumptions discussed in the previous section, we are now ready to state our universality theorems.

\begin{definition} \label{defnorm}  For any function $G:\R^{k}\to \R$ and any natural number $a$, we define $\norm{\triangledown^aG}_{\infty}$ to be the supremum over $x\in \R^{k}$ of the absolute value of all partial derivatives of total order $a$ of $G$ at $x$. For a function $G:\R^{k}\times \C^{l}\to \C$, we define $\norm{\triangledown^aG}_{\infty}$ to be the maximum of $\norm{\triangledown^aG_1}_{\infty}$ and $\norm{\triangledown^aG_2}_{\infty}$, where $G_1, G_2:\R^{k+2l}\to \R$ are the real and imaginary parts of $G$:
$$G_1(x_1, \dots, x_k, u_1, \dots, u_l, v_1, \dots, v_l)=\Re (G(x_1, \dots, x_k, u_1+iv_1, \dots, u_l+iv_l)),$$
$$G_2(x_1, \dots, x_k, u_1, \dots, u_l, v_1, \dots, v_l)=\Im (G(x_1, \dots, x_k, u_1+iv_1, \dots, u_l+iv_l)).$$
\end{definition}

\begin{theorem}[General Complex universality]\label{gcomplex} 
	
	Assume that the coefficients $\xi_i$ and $\tilde \xi_i$ satisfy Condition {\bf C1} for some constants $N_0, \tau, \ep$. Let $\alpha_1, C_1$ be positive constants and $k$ be a positive integer. Set  $A := 2kC_1 + \frac{\alpha_1\ep }{60}$ and 
	$c_1:=  \frac{\alpha_1 \ep }{10^{5} k^{2}}$. Assume that there exists a constant $C>0$ such that the random functions $F_n$ and $\tilde F_n$ satisfy Conditions {\bf C2} \eqref{cond-poly}-{\bf C2} \eqref{cond-delocal} with parameters $(k, C_1, \alpha_1, A,c_1, C)$.  Then there exist positive constants $C', c$ depending only on the constants in Conditions {\bf C1} and {\bf C2} (but not on $\delta_n$, $D_n$ and $n$) such that the following holds.
	
	For any complex numbers $z_1, \dots, z_k$ in $D_n$ and any function $G: \mathbb{C}^{k}\to \mathbb{C}$ supported on \newline$\prod_{i=1}^{k} B (z_i, 1/100) $ with continuous derivatives up to order $2k+4$ and $\norm{\triangledown^aG}_{\infty}\le 1$ for all $0\le a\le 2k+4$, we have 
	\begin{eqnarray}
		\left |\E\sum G\left (\zeta_{i_1}, \dots, \zeta_{i_k}\right) -\E\sum G\left (\tilde \zeta_{i_1}, \dots, \tilde \zeta_{i_k}\right) \right |\le C'\delta_n^{c},\label{gcomplexb}
	\end{eqnarray}
	where the first sum runs over all $k$-tuples $(\zeta_{i_1}, \dots, \zeta_{i_k})$ of the roots $\zeta_1, \zeta_2, \dots$ of $F_n$, and 
		the second  sum runs over all $k$-tuples $(\tilde \zeta_{i_1}, \dots, \tilde \zeta_{i_k})$ of the roots $\tilde \zeta_1, \tilde  \zeta_2, \dots$ of $ \tilde F_n$. 
	\end{theorem}

As an example for the summation in \eqref{gcomplexb}, if $k=2$ and $F_n$ only has two roots $\zeta_1$ and $\zeta_2$, then the first sum is $G(\zeta_1, \zeta_1) + G(\zeta_1, \zeta_2)+G(\zeta_2, \zeta_1)+G(\zeta_2, \zeta_2)$.

 \begin{theorem}[General Real universality]\label{greal} Assume that $\phi_i(\R)\subset \R$ and $\xi_i$ and $\tilde \xi_i$ are real random variables that satisfy Condition {\bf C1} for some constants $N_0, \tau, \ep$. Let $\alpha_1, C_1$ be positive constants and $k, l$ be nonnegative integers with $k+l\ge 1$. Set   $A = 2(k+l+2)(C_1+2) + \frac{\alpha_1\ep }{60}$ and $c_1 = \frac{\alpha_1\ep }{10^9(k+l)^{4}}$.
	Assume that there exists a constant $C>0$ such that the random functions $F_n$ and $\tilde F_n$ satisfy Condition {\bf C2} with parameters $(k+l, C_1, \alpha_1, A,c_1, C)$. Then  there exist positive constants $C', c$ depending only on $k, l$ and the constants in Conditions {\bf C1} and {\bf C2} (but not on $\delta_n$, $D_n$ and $n$) such that the following holds.
	
	For any real numbers $x_1,\dots, x_k$, complex numbers $z_1, \dots, z_l$, all of which are in $D_n$, and any function $G: \mathbb{R}^{k}\times\mathbb{C}^{l}\to \mathbb{C}$ supported on $\prod_{i=1}^{k}[x_i-1/100, x_i+1/100] \times \prod_{j=1}^{l}B (z_j, 1/100)$ with continuous derivatives up to order $2(k+l)+4$ and $\norm{\triangledown^aG}_{\infty}\le 1$ for all $0\le a\le 2(k+l)+4$, we have
	\begin{eqnarray} 
		\left |\E\sum G\left (\zeta_{i_1}, \dots, \zeta_{i_k}, \zeta_{j_1} , \dots, \zeta_{j_l}\right) 
		-\E\sum G\left (\tilde \zeta_{i_1}, \dots, \tilde \zeta_{i_k}, \tilde \zeta_{j_1}, \dots, \tilde \zeta_{j_l}\right) \right |\le C' \delta_n^{c},\nonumber
	\end{eqnarray}
	where the first sum runs over all $(k+l)$-tuples $(\zeta_{i_1}, \dots, \zeta_{i_k}, \zeta_{j_1}, \dots, \zeta_{j_l}) \in \R^{k}\times \C_{+}^{l}$ of the roots $\zeta_1, \zeta_2, \dots$ of $F_n$, and the second  sum runs over all $(k+l)$-tuples $(\tilde \zeta_{i_1}, \dots, \tilde \zeta_{i_k}, \tilde \zeta_{j_1}, \dots, \tilde \zeta_{j_l}) \in \R^{k}\times \C_{+}^{l}$ of the roots $\tilde \zeta_1, \tilde  \zeta_2, \dots$ of $ \tilde F_n$. 
\end{theorem}

\begin{remark}\label{rmkconstants} The specific values  of  $A$ and $c_1$ in both theorems are chosen for the sake of explicitness. The theorems hold for any bigger $A$ and any smaller $c_1$. 
The constant $c$ in both theorems can be chosen to be $c_1$, namely $\frac{\alpha_1 \ep }{10^{5} k^{2}}$ and $\frac{\alpha_1\ep }{10^9(k+l)^{4}}$, respectively. We make no attempt to optimize these constants.
\end{remark}

\subsection{Main ideas and technical novelties} \label{mainideas} 
\subsubsection{Main ideas}

Let us consider the simplest setting where $k=1, l=0$ and we need to show 
\begin{equation} \nonumber
\sum_{i=1}^n \E  G ( \zeta_i )=  \sum_{i=1}^n \E  G (\tilde \zeta_i ) +O\left (\delta_n^{c}\right ) ,\end{equation} 
 where the $\zeta_i$ (and the $\tilde \zeta_i $) are the roots of $F_{n}$ (and $\tilde F_{n}$, respectively) and $G$ is a (smooth) test function supported on a disk $B(z_0, 1/100)$.

\noindent Our starting point is the Green's formula, which asserts that 

$$ G(0) = \frac{1}{2\pi}\int_{\C} \log |z| \Delta G(z) dz. $$

\noindent By change of variables, this implies that for all $i$,  

$$G(\zeta_i) = \frac{1}{2\pi}\int_{\C} \log |z -\zeta_i | \Delta G(z) dz, $$ which, in turn, yields

\begin{equation} \nonumber
\sum_i \E  G ( \zeta_i ) = \frac{1}{2\pi}\E  \int_{\C} \log | \prod_{i=1}^n (z -\zeta_i ) | \Delta G(z) dz = \frac{1}{2\pi}\E  \int_{B(z_0, 1/100)} \log | F_{n} (z)  | \Delta G(z) dz. \end{equation}

An obvious, and major, technical difficulty here is that the logarithmic function has a singularity at 0. This, naturally, leads to the anti-concentration issue that we discussed earlier, namely we need to bound the probability that 
$|F_n(z)|$ is close to zero.  Condition {\bf C2} (2)  has been introduced to address this issue.

Let us assume, for a moment, that the singularity problem has been handled properly (we will discuss the anti-concentration property shortly).  Then, by using Conditions {\bf C2} \eqref{cond-poly}-\eqref{cond-bddn}, we can show that  the function $F_n$ is nice enough that we can replace  $\log|F_n|$ by $K(F_n)$ where  $K$ is a bounded smooth function. The key argument of this part is to bound the error term, which turns out to be relatively simple. 

The task is now reduced to showing that
$$ \E \int_{B(z_0, 1/100)} K\left (F_n(z)\right )\Delta G(z)dz - \E \int_{B(z_0, 1/100)} K\left (\tilde F(z)\right )\triangle G(z)dz= O(\delta^{c}).$$

Because of the boundedness of $G$, for each $z\in B(z_0, 1/100)$, it suffices to show that
 $$ \E K\left (F_n(z)\right )  - \E K\left (\tilde F(z)\right )= O(\delta^{c}).$$
 
Since for each fixed $z$, $F_n(z)$ is a sum of independent random variables, the desired bound can  be viewed, in some sense, as a quantitative version of  the Central Limit Theorem. We will actually   prove it by the Lindeberg swapping method, which, by now, is a standard tool for proving local universality. 

Generalizing the whole scheme to the general case of $k$ and $l$  requires several additional technical steps, but the spirit of the method remains the same.

\subsubsection{Comparison with earlier papers \cite{TVpoly} and \cite{DOV}}
Our method differs from that of  \cite{TVpoly}   at  essential steps.  The first key idea in \cite{TVpoly} is to handle the integral 

$$ \frac{1}{2\pi}\E  \int_{B(z_0, 1/100)} \log | F_{n} (z)  | \Delta G(z) dz $$ by 
a random Riemann sum. One tries to  approximate  this integration by $\frac{c}{m} (f(z_1) + \dots f(z_m)) $, where $z_i$ are iid random points sampled from the disk,  $m$ is a properly chosen parameter which tends to infinity with $n$, $c$ is a normalizing constant, and $f:= \log |F_n| \Delta G $.

 With that approach, one faces two major technical tasks. The first (and harder one) is to control the error term in the approximation. This leads to the 
 problem of estimating the variance in the sampling process.  The other task  is to prove a comparison estimate for the 
random vector $(f(z_1), \dots, f(z_m))$, where we now view 
the points  $z_1, \dots, z_m$ as fixed, with  the randomness coming from $F_n (z)$. This, again, can be done using a Lindeberg type argument (applying to high dimensional setting).

Our new proof avoids this  sampling argument completely, making  the argument much shorter and more direct. For instance, the proof of Theorem \ref{gcomplex}, barring some lemmas in the appendix, is now only 3 pages. 

Let us now discuss the critical anti-concentration property. In practice, it has been  a major issue to prove  that a random function satisfies the anti-concentration phenomenon in some way. (As pointed out earlier, this is needed in order to address the singularity problem concerning the logarithmic function.)

In earlier papers  \cite{TVpoly} and \cite{DOV},  every class of random (algebraic) polynomials required a different proof. 
In \cite{TVpoly}, for Weyl and elliptic polynomials, the authors used 
Littlewood-Offord arguments for lacunary sequences. In the same paper, the  proof for Kac polynomials required a much more sophisticated argument, based
on the  Inverse Littlewood-Offord theory (see Nguyen-Vu \cite{nguyen2011optimal}) and a weak version of
the quantitative (Gromov) rigidity theorem (see Shalom-Tao \cite{shalom2010finitary}). However, this proof does not hold for 
the derivatives of Kac polynomials and random polynomials with 
slowly growing coefficients.  In order to handle these classes, in \cite{DOV}, 
the authors needed to use  a beautiful result on log-integrability by Nazarov-Nishry-Sodin \cite{NNS}, a very recent development. 
However, none of these tools works for  random trigonometric polynomials, whose roots behave quite  differently.

An important new  point in our   proof  is that we require a much weaker 
 anti-concentration property than in previous papers.
We only require that $F_n(z)$, as a random variable, satisfies the anti-concentration for only one  point $z$ in the whole neighborhood, while in \cite{TVpoly} one requires anti-concentration to hold for most points in the same neighborhood. (Notice that since we are taking an integration with respect to $z$, this earlier requirement from \cite{TVpoly} looks natural.)  The key to this observation  is our  Lemma \ref{2norm}, which asserts  that under favorable conditions, a lower bound 
on $|F_n(w)|$ guarantees a weaker, but still useful, lower bound for  $|F_n (z)|$  for any $z$ in a neighborhood of 
$w$.

Building upon this new  observation,  we have developed  a novel method (based on  old results of Tur\'an and Hal\'asz) 
to verify the anti-concentration property  in a simple and robust manner. This effort leads to  Lemma \ref{lmanti_concentration}, which we can use, in a  rather straightforward way, to prove the desired anti-concentration property for all 
ensembles of random functions discussed in this paper (including all the algebraic polynomials discussed above, random trigonometric polynomials with general coefficients, and a very recent ensemble studied by Flasche-Kabluchko \cite{flasche2020expected}).

\section{Application: Universality for random trigonometric polynomials} \label{app1}

 In this section, we apply our theorems to study 
 {\it random trigonometric polynomials} of the following form
 \begin{equation}
P_n(x) = \sum_{j=0}^{n} c_j\xi_j\cos(jx) + \sum_{j=1}^{n} d_j\eta_j\sin(jx)\nonumber
 \end{equation}
 where $c_j$ and $d_j$ are deterministic coefficients, 
 and $\xi_0, \xi_1, \dots, \xi_n$ and $\eta_1, \dots, \eta_n$ are independent random variables with unit variance. All of the $c_j, d_j, \xi_j$ and $\eta_j$ may depend on $n$.

 Most of the existing literature deals with the special case $c_i =d_i=1$ or $c_i= 1, d_i=0 $ for every $i$. 
 The generality of our study enables us to consider more general coefficients. All we need to assume about the coefficients $c_i, d_i$ is the following 
 
{\bf  Condition C3}. 
There exist positive  constants $\tau_1, c$ and an interval   $\mathcal I_{0}  \subset \{1, \dots, n\}$ of size at least $cn$ such that 
 \begin{equation} 
 |c_i|\ge \tau_1 \max_{0\le j\le n} \{|c_j|, |d_j|\} \qquad\text{for all } i\in \mathcal I_{0}. 
 \end{equation}

With regard to the random variables, we assume that they have mean 0, except for finitely many of them whose mean can be as large as $n^{1/2+o(1)}$. Specifically, we assume

{\bf Condition C4}. There is a constant $N_0 \ge 0$ such that for  $i\ge N_0$, $\E\xi_i = \E\eta_i = 0$ and for  $0\le i< N_0$, $|\E \xi_i|\le n^{\tau_0}$, and $|\E \eta_i|\le n^{\tau_0}$, where  $\tau_0 := 1/2+ 10^{-11}\ep$.

The $\ep$ in this condition is the $\ep$ in Condition
{\bf C1}. The constant $\tau_0$  is not optimal but we make no attempt to improve it. 
 We use the same notation $N_0$ in both  Condition {\bf C4} and Condition {\bf C1}, as we can always replace two different $N_0$ by their maximum.  The assumption that $\mathcal I_{0}$ is an interval is only used in the following simple lemma. 

\begin{lemma}\label{J} 
Let $\mathcal I_{0}$ be an interval in $\{1, \dots, n \} $ of length $\beta n$, for some constant $\beta >0$. Then there is a
constant $\beta' >0$ such that  for any real number $a$, 
	the set $\mathcal I_{0}$ contains a subset $J_a$  of size at least $\beta'n$, where 
$\min_{k\in \Z} \{|2aj - (2k+1)\pi| \} \ge \beta'$ for all $j\in J_{a}$.
	
\end{lemma} 

Let 
\begin{equation}
\tilde P_n(x) = \sum_{j=0}^{n} c_j\tilde \xi_j\cos(jx) + \sum_{j=1}^{n} d_j\tilde \eta_j\sin(jx)\nonumber
 \end{equation}
where $\tilde \xi_0, \tilde \xi_1, \dots, \tilde \xi_n$ and $\tilde \eta_1, \dots, \tilde \eta_n$ are some other independent random variables.

%

\begin{theorem}[Universality for trigonometric polynomials] \label{real} Let $k, l$ be nonnegative integers.
	Assume that the real coefficients $c_i$ and $d_i$ satisfy Condition {\bf C3} and the two sequences of real random variables $(\xi_0, \dots, \xi_n, \eta_1, \dots, \eta_n)$ and $(\tilde \xi_0, \dots, \tilde \xi_n, \tilde \eta_1, \dots, \tilde \eta_n)$ satisfy Conditions {\bf C1} and {\bf C4}. Then for any positive constant C, there exist positive constants $C', c$ depending only on $C, k, l$ and the constants in Conditions {\bf C1, C3, C4} such that the following holds. 

For any real numbers $x_1,\dots, x_k$, and complex numbers $z_1, \dots, z_l$ such that $|\Im(z_j)|\le C/n$ for all $1\le j\le l$, and for any function $G: \mathbb{R}^{k}\times\mathbb{C}^{l}\to \mathbb{C}$ supported on $\prod_{i=1}^{k}[x_i-1/n, x_i+1/n] \times \prod_{j=1}^{l}B (z_j, 1/n)$ with  continuous derivatives up to order $2(k+l)+4$ and $\norm{\triangledown^aG}_\infty\le n^{a}$ for all $0\le a\le 2(k+l)+4$, we have
\begin{eqnarray}\nonumber
\left |\E\sum G\left (\zeta_{i_1}, \dots, \zeta_{i_k}, \zeta_{j_1} , \dots, \zeta_{j_l}\right) 
 -\E\sum G\left (\tilde \zeta_{i_1}, \dots, \tilde \zeta_{i_k}, \tilde \zeta_{j_1}, \dots, \tilde \zeta_{j_l}\right) \right |\le C'n^{-c},
\end{eqnarray}
where the first sum runs over all $(k+l)$-tuples $(\zeta_{i_1}, \dots, \zeta_{i_k}, \zeta_{j_1}, \dots, \zeta_{j_l}) \in \R^{k}\times \C_{+}^{l}$ of the roots $\zeta_1, \zeta_2, \dots$ of $P_n$, and the second  sum runs over all $(k+l)$-tuples $(\tilde \zeta_{i_1}, \dots, \tilde \zeta_{i_k}, \tilde \zeta_{j_1}, \dots, \tilde \zeta_{j_l}) \in \R^{k}\times \C_{+}^{l}$ of the roots $\tilde \zeta_1, \tilde  \zeta_2, \dots$ of $ \tilde P_n$. 
\end{theorem}

To the best of our knowledge, the above theorems seem to be the first universality results concerning local statistics of the roots of 
random trigonometric polynomials. To make a comparison to existing literature, let us focus on
the distribution of real roots, which is the case $k=1, l=0$ in Theorem \ref{real}). 

The number of real roots has been a main focus of the study of random trigonometric polynomials. 
The Gaussian setting has been investigated by a number of reseachers, including Dunnage \cite{Dunnage1966number}, Sanbandham \cite{sambandham1978number}, Das \cite{das1968trig}, Wilkins \cite{wilkins1991trig}, Edelman and Kostlan
\cite{EK} and many others. One can compute an exact 
answer for the expectation using either Kac-Rice formula or Edelman-Kostlan formula \cite{EK}.

For the non-Gaussian case, little has been known until very recently. Angst and Poly \cite{angstpoly}, in a recent preprint, proved the asymptotics of the mean number of roots of $P_n$ in a fixed interval $[a, b]$ under the assumptions of finite fifth moment and a Cramer-type condition. Their approach introduced a novel way to work with the Kac-Rice formula which had been considered to be difficult in discrete settings. Using an approach originated by Erd\H{o}s-Offord \cite{EO} and later developed by Ibragimov-Maslova \cite{Ibragimov1968average} \cite{Ibragimov1971expected1}, Flasche \cite{flasche} extended the result in \cite{angstpoly} with assumptions on the first two moments only. Let $N_{P_n}(a, b)$ denote the number of real roots of $P_n$ in an interval $[a, b]$.

\begin{theorem} [Flasche \cite{flasche}]\label{flasche}
	Let $u\in \R$ and $0\le a< b\le 2\pi$ be fixed numbers. Let $P_n(x) = u\sqrt n + \sum_{j=0}^{n} \xi_j\cos(jx) + \sum_{j=1}^{n} \eta_j\sin(jx)$ where $\xi_j$ and $\eta_j$, $j\in \N$, are iid random variables with mean 0 and variance 1. Then
	$$\lim_{n\to \infty} \dfrac{\E N_{P_n}(a, b)}{n} = \frac{b-a}{\pi\sqrt 3} \exp\left (-\frac{u^{2}}{2}\right ).$$
\end{theorem}

Notice that in this theorem, the interval $[a,b]$ contains a linear number of roots. 
For smaller intervals, a few years ago, Aza{\"\i}s and coauthors \cite{azais2015local} showed that if $\xi_i$ and $\eta_i$ are iid with a smooth density function, then in an interval of size $\Theta(1/n)$, the number of real zeros converges in distribution to that of a suitable Gaussian process (and is thus universal). In an even more recent paper \cite{iksanov2016local}, Iksanov-Kabluchko-Marynych removed the assumption of smooth density, using a different method.

\begin{theorem} [Iksanov-Kabluchko-Marynych \cite{iksanov2016local}]\label{IKK}
	Let $P_n(x) = \sum_{j=0}^{n} \xi_j\cos(jx) + \sum_{j=1}^{n} \eta_j\sin(jx)$ where $(\xi_j,\eta_j)$, $j\in \N$, are iid real random vectors with mean 0 and unit covariance matrix. Let $(s_n)$ be any sequence of real numbers and $[a, b]\subset \R$ a fixed interval. Then
	$$N_{P_n} \left (s_n + \frac{a}{n}, s_n + \frac{b}{n}\right ) \underset{n\to \infty}{\overset{d}{\longrightarrow}} N_{Z}(a, b)$$
	where $(Z(t))_{t\in R}$ is the stationary Gaussian process with mean 0 and covariance matrix 
	$$\Cov(Z(t), Z(s))=\begin{cases}
	\frac{\sin(t-s)}{t-s} \quad\mbox{if } t\neq s\\
	1 \quad\mbox{if } t = s.
	\end{cases}$$ 
	\end{theorem}

In all of these previous works, the coefficients $c_i, d_i$ are: $c_i=d_i=1$ or $c_i=1, d_i=0$. 
Our setting is more general, as we only require a linear fraction of the  $c_i$ to be sufficiently large and 
allow the rest of the (smaller) 
coefficients to be arbitrary.

Our result implies the following corollary concerning the number of real roots. 

\begin{theorem}\label{comparison}
	Under the assumptions of Theorem \ref{real}, there exist positive constants $C$ and $c$ such that for any $n$ and for any numbers $a_n< b_n$, we have
	$$\frac{|\E N_{P_n}(a_n, b_n) - \E N_{\tilde P_n}(a_n, b_n)|}{(b_n-a_n)n}\le Cn^{-c}\left (1 + \frac{1}{(b_n-a_n)n}\right ).$$
\end{theorem}

By using the Kac-Rice formula (Proposition \ref{KacRice}) for the Gaussian case, we obtain the following
precise estimate.

\begin{cor}\label{maincor} Let $C, \ep$ and $\tau_1$ be positive constants. Let 
	$-C\le u_n\le C$ be a deterministic number. Let
	$$P_n(x) = u_n\sqrt{\sum_{i=0}^{n}c_i ^{2}} + \sum_{j=0}^{n}c_j \xi_j\cos(jx) + \sum_{j=1}^{n} c_j\eta_j\sin(jx)$$ 
	where $\xi_j$ and $\eta_j$, $j\le n$, are independent (not necessarily identically distributed) real random variables with mean 0, variance 1 and $(2+\ep)$-moments bounded by $C$, and the real coefficients $c_j$ satisfy condition {\bf C3}. 
 Then for any numbers $a_n< b_n$, we have
	$$\E N_{P_n}(a_n, b_n) = \frac{b_n-a_n}{\pi} \sqrt{\frac{\sum_{j=0}^{n} c_j^{2}j^{2}}{\sum_{j=0}^{n} c_j^{2}}}\exp\left (-\frac{u_n^{2}}{2}\right ) + O\left (n^{-c} \right ) ((b_n-a_n)n + 1) $$
where the positive constant $c$ and the implicit constant depend only on $C,  \ep$ and $\tau_1$.
\end{cor} 
 
This corollary extends both Theorems \ref{flasche} and \ref{IKK} in the sense that it holds for general coefficients $c_i, d_i$ and intervals of all scales. It does not seem that the methods used in these papers can cover 
the same range. On the other hand, our random coefficients are required to have bounded $(2+ \ep)$-moments. It is an interesting open problem to see 
to what extent this assumption is necessary.

\begin{remark} \label{rmk1}
In the proof, we will show that Corollary \ref{maincor} holds for a more general case in which 
\begin{eqnarray}
P_n(x) &=& \sqrt{\sum_{i=0}^{n}c_i ^{2}} \left (u_n+ \sum _{j=0}^{N_0} u_j n^{-\alpha}   \cos(jx)+ \sum_{j=1}^{N_0} v_j n^{ -\alpha} \sin(jx) \right )\nonumber\\ 
&&+ \sum_{j=0}^{n}c_j \xi_j\cos(jx) + \sum_{j=1}^{n} c_j\eta_j\sin(jx)\label{newP}
\end{eqnarray}
where $N_0, \alpha>0$ are any constants and $-C\le u_j, v_j\le C$ are deterministic numbers that can depend on $n$. This means that the result is applicable to not only the number of zeros of $P_n$ but also the number of intersections between $P_n$ and a deterministic trigonometric polynomial 
$$Q(x) := \sqrt{\sum_{i=0}^{n}c_i ^{2}}  \left (u_n'+ \sum _{j=0}^{N_0} u_j n^{-\alpha}  \cos(jx)+ \sum_{j=1}^{N_0} v_j n^{ -\alpha} \sin(jx) \right )$$ where $u_n'$, $u_j$ and $v_j$ are bounded deterministic numbers. To see this, one only needs to apply the result to the random polynomial $P_n - Q$. 
\end{remark}

Now let us go back to the special case with $c_i=d_i =1$
\begin{equation}
P_n(x) = \sum_{i=0}^{n}\xi_i \cos(ix) + \sum_{i=1}^{n}\eta_i \sin(ix). \nonumber
\end{equation}

By applying Corollary \ref{maincor} directly to the derivatives of $P_n$, we get the following result.
\begin{cor}
Let $k$ be a nonnegative integer and $C$ be a positive constant. Assume that the random variables $\xi_i$ and $\eta_i$, $i\le n$, are independent (not necessarily identically distributed) real random variables with mean 0, variance 1 and $(2+\ep)$-moments bounded by $C$. For any numbers $a_n< b_n$, the expected number of real zeros of the $k$-th derivative of $P_n$ in an interval $[a_n, b_n]$ is
$$\E N_{P_n^{(k)}}(a_n, b_n) = \sqrt{\frac{2k+1}{2k+3}} \frac{(b_n-a_n)n}{\pi}+ O\left (n^{-c}  \right ) ((b_n-a_n)n + 1)$$
where the positive constant $c$ and the implicit constant depend only on $k, C$ and $\ep$.
\end{cor}

The key to our proof is the new technique to verify anti-concentration, which we discussed at the end of the Introduction (see also Remark \ref{anticond}) and  at the end of the previous section. For details, see Section \ref{proof-main}.

\section{Application: Universality for Kac polynomials}\label{app2}

In this section, we apply our result to Kac polynomials, 
$$P_n(x) = \sum_{i=0}^{n} \xi_i x^{i}$$
where $\xi_0, \xi_1, \dots, \xi_n$ are iid copies of a real random variable $\xi$ with mean zero and unit variance.
This is perhaps the most studied model of random polynomials. Indeed, the starting point of the theory of random functions was a series of papers in the early 1900s examining the number of real roots of the Kac polynomials.

The first rigorous work on random polynomials was due to Bloch and Polya in 1932 \cite{BP}, 
who considered the Kac polynomial with $\xi$ being 
Rademacher, namely $\P(\xi=1)=\P(\xi=-1)=1/2$. In what follows, we denote by $N_{n, \xi}$ the number of real roots of $P_n (x)$.
Next came the ground-breaking series of papers by Littlewood and Offord \cite{LO2, LO3, LO1} in the early 1940s, 
which, to the surprise of many mathematicians at the time, showed that $N_{n, \xi}$ is typically poly-logarithmic in $n$.

\begin{theorem} [Littlewood-Offord]
	For $\xi$ being Rademacher, Gaussian, or uniform on $[-1,1]$, 
	$$ \frac{\log n} {\log \log n} \le N_{n, \xi} \le \log^2 n$$ with probability $1-o(1)$. 
\end{theorem}

During more or less the same time, Kac \cite{Kac1943average} discovered his famous formula for the density function $\rho(t)$ of $N_{n, \xi}$
\begin{equation} \nonumber \rho(t) = \int_{- \infty} ^{\infty} |y| p(t,0,y) dy, \end{equation} where 
$p(t,x,y)$ is the joint probability density of $P_{n} (t) =x$ and the derivative $P'_{n} (t) =y$.

Consequently, 
\begin{equation} \label{Kacformula} \E N_{n ,\xi} = \int_{-\infty}^{\infty} dt \int_{- \infty} ^{\infty} |y| p(t,0,y) dy.
\end{equation}

In the Gaussian case ($\xi$ is Gaussian), one can compute the joint distribution of $P_{n} (t)$ and $P'_{n}(t)$ rather easily. 
Kac showed in \cite{Kac1943average} that
\begin{equation} \nonumber \E N_{n, Gauss} = \frac{1}{\pi} \int_{-\infty} ^{\infty} \sqrt { \frac{1}{(t^2-1) ^2} + \frac{(n+1)^2 t^{2n}}{ (t^{2n+2} -1)^2} } dt = \left (\frac{2}{\pi} +o(1)\right ) \log n. \end{equation}

In his original paper \cite{Kac1943average}, Kac thought that his formula would lead to the same estimate for $\E N_{n, \xi}$ for all other random variables $\xi$. It has turned out not to be the case, as the right-hand side of 
\eqref{Kacformula} is often hard to compute, especially when $\xi$ is discrete (Rademacher for instance).
Technically, the computation of the joint distribution of $P_{n} (t)$ and $P'_{n}(t)$ is easy in the Gaussian case, thanks to special properties of the Gaussian distribution, 
but can pose a great challenge 
in general. Kac admitted this in a later paper \cite{Kac2} in which he managed to push his method to 
treat the case $\xi$ being uniform in $[-1,1]$, using analytic tools. A further extension was made by Stevens \cite{Stev}, who evaluated Kac's formula for a large class of $\xi$ having continuous and smooth distributions
with certain regularity properties (see \cite[page 457]{Stev} for details). Since the distributions are smooth, the two later results 
follow rather easily from our universality results; see the discussion at the end of the last section and Remark
\ref{anticond}; we leave the routine verification as an exercise for the interested reader.

The computation of $\E N_{n, \xi}$ for discrete random variables $\xi$ required a considerable effort. It took more than 10 years until Erd\H{o}s and Offord \cite{EO} found a completely new approach to handle the Rademacher case, proving the following.
\begin{theorem} \cite{EO} \label{e.erdos-offord} Let $\xi_i$ be iid Rademacher random variables. Then
	\begin{equation}
	N_{n, \xi} = \frac{2}{\pi} \log n + o\left ((\log n)^{2/3} \log \log n\right )\nonumber
	\end{equation} 
	with probability at least $1 - o\left (\frac{1}{\sqrt{\log \log n}}\right )$.
\end{theorem}

The argument of Erd\H{o}s and Offord is combinatorial and very delicate, even by today's standards. Their main idea is to approximate the number of roots by the number of sign changes 
in $P_{n} (x_1) , \dots, P_{n}(x_k)$ where $(x_1, \dots, x_k)$ is a carefully chosen deterministic sequence of points of length $k = (\frac{2}{\pi} +o(1)) \log n$. The authors showed that with high probability, almost every interval 
$(x_i, x_{i+1} )$ contains exactly one root, and used this fact to prove Theorem \ref{e.erdos-offord}.

Our main result in this section is the following universality statement.

\begin{theorem}[Universality for Kac polynomials]\label{kacreal}
Let $k, l$ be nonnegative integers with $k+l\ge 1$.
	Assume that $\xi_0, \dots, \xi_n$ and $\tilde \xi_0, \dots, \tilde \xi_n$ are real random variables with mean 0, satisfying Condition {\bf C1} and the polynomials $P_n$, $\tilde P_n$ 
	are Kac polynomials with respect to these variables. Then there exist positive constants $C', c$ depending only on $k, l$ and the constants in Condition {\bf C1} such that the following holds. 

For every $0 < \theta_n < 1$, for any real numbers $x_1,\dots, x_k$, and complex numbers $z_1, \dots, z_l$ with $1-2\theta_n \le |x_i|, |z_j|\le 1-\theta_n +1/n$ for all $i, j$, and for any function $G: \mathbb{R}^{k}\times\mathbb{C}^{l}\to \mathbb{C}$ supported on $\prod_{i=1}^{k}[x_i-10^{-3}\theta_n, x_i+10^{-3}\theta_n] \times \prod_{j=1}^{l}B (z_j, 10^{-3}\theta_n)$ with  continuous derivatives up to order $2(k+l)+4$ and $\norm{\triangledown^aG}_\infty\le (\theta_n+1/n)^{-a}$ for all $0\le a\le 2(k+l)+4$, we have
	\begin{eqnarray}\nonumber
	\left |\E\sum G\left (\zeta_{i_1}, \dots, \zeta_{i_k}, \zeta_{j_1} , \dots, \zeta_{j_l}\right) 
	-\E\sum G\left (\tilde \zeta_{i_1}, \dots, \tilde \zeta_{i_k}, \tilde \zeta_{j_1}, \dots, \tilde \zeta_{j_l}\right) \right |\le C'\theta_n^{c} + C'n^{-c},
	\end{eqnarray}
	where the first sum runs over all $(k+l)$-tuples $(\zeta_{i_1}, \dots, \zeta_{i_k}, \zeta_{j_1}, \dots, \zeta_{j_l}) \in \R^{k}\times \C_{+}^{l}$ of the roots $\zeta_1, \zeta_2, \dots$ of $P_n$, and the second  sum runs over all $(k+l)$-tuples $(\tilde \zeta_{i_1}, \dots, \tilde \zeta_{i_k}, \tilde \zeta_{j_1}, \dots, \tilde \zeta_{j_l}) \in \R^{k}\times \C_{+}^{l}$ of the roots $\tilde \zeta_1, \tilde  \zeta_2, \dots$ of $ \tilde P_n$. 
\end{theorem}

\begin{remark} \label{Qrmk}
Theorem \ref{kacreal} provides universality result for the polynomial $P_n$ on the disk $B(0, 1+1/n)$. For the complement of this disk, consider $Q_n(z): = z^{n} P_n(z^{-1})$ which is another Kac polynomial. Since the roots of $Q_n$ are just the reciprocal of the roots of $P_n$, the universality of $Q_n$ in $B(0, 1)$ implies the universality of $P_n$ outside the disk $B(0, 1)$.
\end{remark}

As a corollary, we get the following result on the number of real roots of these polynomials which recovers the main result of Do and the authors in \cite{DOV}.

\begin{cor}\label{kacmean} 
	Let $C$ be a positive constant. Assume that the random variables $\xi_i$ are independent (not necessarily identically distributed) real random variables with mean 0, variance 1 and $(2+\ep)$-moments bounded by $C$. Then
	$$\E N_{P_n}(\R) =\frac{2}{\pi} \log n +O(1)$$
	where the implicit constant depends only on $C$ and $\ep$.
\end{cor}

Theorem \ref{kacreal} 
 strengthens an earlier result of Tao and the second author \cite{TVpoly}. The result in \cite{TVpoly} only covers 
the bulk of the spectrum, namely the region $1-n^{-\ep}\le |x|\le 1+n^{-\ep}$. Restricting to the number of real roots, it yields 	
$$\E N_{P_n}(\R) =  O\left (\log n\right )$$ instead of the more precise (and optimal) estimate in Corollary \ref{kacmean}. 
Another new feature is that our result also yields sharp estimates for the size of level sets $\{ z\in \C: P_n (z) =a \}$, for any fixed $a$, since we can
allow that in Theorem \ref{kacreal} and Corollary \ref{kacmean}, $\xi_0 $ (and in fact any finite number of $\xi_i$) has non-zero, bounded mean. Our proofs work automatically under this extension. A version of Corollary \ref{kacmean} is obtained earlier in \cite{NNV} using a different approach that combines the local universality in the bulk and a comparison of the number of real roots of $P_n$ with that of $P_{n'}$ for $n'$ much larger than $n$.

The proof in \cite{TVpoly} made use of a deep anti-concentration lemma \cite[Lemma 14.1]{TVpoly} whose proof 
relies on the Inverse Littlewood-Offord theory and a weak quantitative version of Gromov's theorem. The proof we will provide here
is simple and almost identical to the one used to treat random trigonometric polynomials in the last section. For random variables having continuous distributions (such as the cases treated by Kac and Stevens mentioned above), the anti-concentration 
property (see Remark \ref{anticond}) is immediate.

\begin{remark}
One can routinely modify the proofs of Theorem \ref{kacreal} and Corollary \ref{kacmean} to show that these results hold for more general settings. For example, the proofs can be used to show that these results apply for 
$$P_n(x) = \sum_{i=0}^{n} c_i \xi_i x^{i}$$
where $\xi_i$ are independent (not necessarily identically distributed) random variables satisfying Condition {\bf{C1}} with zero mean and the deterministic coefficients $c_i$ grow polynomially. Specifically, these results hold for derivatives of the Kac polynomials of any given order. We leave the details to the interested reader. The aforementioned results for this general version were proven in the previous work \cite{DOV} using much more involved tools and arguments. 
\end{remark}

We defer the proofs of Theorem \ref{kacreal} and Corollary \ref{kacmean} to Section \ref{kacproof}.

\section{Application: Universality for Weyl series}\label{app3}
In this section, we discuss an application of our main theorems to Weyl series
$$P(z) = \sum_{j=0}^{\infty} \frac{\xi_j z^{j}}{\sqrt{j!}}$$
where $\xi_j$ are independent complex random variables satisfying the matching condition {\bf C1} with the $\tilde \xi_j$ being standard complex Gaussian random variables with density $\frac{1}{\pi} e^{-|z|^{2}}$. In the literature, Weyl series are also referred to as 
flat series. 

The flat series $\tilde P(z) = \sum_{j=0}^{\infty} \frac{\tilde \xi_j z^{j}}{\sqrt{j!}}$ is also known as the flat Gaussian analytic function and has been studied intensively over the past few decades. See, for example, \cite{HKPV}, \cite{sodin2005zeroes}, \cite{sodin2004random}, and the references therein. Using the Edelman-Kostlan formula \cite{EK}, one can show  that for any Borel set $B\subset \C$, the expected number of roots of $\tilde P$ in $B$ is
\begin{equation}\label{flat1} \E N_{\tilde P}(B) = \frac{1}{\pi} m(B) \end{equation}
where $m(B)$ is the Lebesgue measure of $B$. 

For general random variables, to compare the distribution of the roots of $P$ with that of $\tilde P$, Kabluchko and Zaporozhets (2014) \cite{kabluchko2014asymptotic} showed that with probability $1$, the rescaled empirical measure $\mu_{r}$ defined by 
$$\mu_r(A) = \frac{1}{r}\sum_{\zeta: P(\zeta) = 0} \textbf{1}_{\zeta\in \sqrt r A}$$ 
converges vaguely as $r\to\infty$ to the measure $\frac{1}{\pi}m(\cdot)$, which is, as mentioned above, the corresponding measure for $\tilde P$. We recall that a sequence of measures $(\mu_r)$ is said to converge vaguely to a measure $\mu$ if $\lim _{r\to\infty}\int fd\mu_r = \int fd\mu$ for every continuous, compactly supported function $f$.

The aforementioned result of \cite{kabluchko2014asymptotic} is about the rescaled measures $\mu_r$. Thus, it provides an asymptotically sharp   estimate on  the number of roots of $P$ in large domains of the form $\sqrt{r} B$ where $r\to \infty$ and $B$ is a fixed ``nice" measurable domain, but does not give estimates for the number of roots in 
 domains with fixed area, as in \eqref{flat1}.

Using our framework, we obtain the following result at the local scale.

\begin{theorem}\label{uni_flat}(Universality for random flat series)
Assume that the complex random variables $\xi_j$ satisfy the matching condition {\bf C1} with the $\tilde \xi_j$ being standard complex Gaussian random variables and the random variables $\Re(\xi_0), \Im(\xi_0), \Re(\xi_1), \Im(\xi_1), \dots$ are independent. Then there exist positive constants $C, c$ depending only on the constants in Condition {\bf C1} such that the following holds.

For any complex number $z_0$ and for any function $G: \mathbb{C} \to \mathbb{C}$ supported on $ B (z_0, 1)$ with continuous derivatives up to order $6$ and $\norm{\triangledown^aG}_\infty\le 1$ for all $0\le a\le 6$, we have
\begin{eqnarray} 
\left |\E\sum G\left (\zeta \right) -\E\sum G\left (\tilde \zeta\right) \right |\le C |z_0|^{-c},\nonumber
\end{eqnarray}
where the first sum runs over all the roots $\zeta_1, \zeta_2, \dots$ of $P$, and the second  sum runs over all the roots $\tilde \zeta_1, \tilde  \zeta_2, \dots$ of $ \tilde P$. 
\end{theorem}

As a corollary, we obtain a sharp estimate on the number of roots in regions with a fixed area.

\begin{cor}\label{mean_flat} For any constant $C>0$, let $B$ be an angular square $B = \{Re^{i\theta}: R\in [r, r+1], \theta\in [\theta_0, \theta_0 + C/r]$ for some numbers $r>0$ and $\theta_0$. Under the assumption of Theorem \ref{uni_flat}, we have
$$\E N_{P}(B) = \frac{1}{\pi} m (B)+ O(r^{-c})\quad\text{as $r\to \infty$,}$$
 where $c$ and the implicit constant only depend on $C$ and the constants in Condition {\bf C1}.
\end{cor}

The angular square $B$ can be replaced by a disk, square, or any other nice domains whose indicator functions can be well approximated by smooth functions, with only a nominal modification of the proof. Thus, we have a generalization of \eqref{flat1} for flat series with general 
random coefficients. 

To the best of our knowledge, Theorem \ref{uni_flat} and Corollary \ref{mean_flat} are new. We present a short proof of these results in Section \ref{proof_flat}.

\section{Application: Universality for elliptic polynomials}\label{app4}
In this section, we briefly illustrate how to apply our framework to elliptic polynomials
$$P_n(z) = \sum_{i=0}^{n}\sqrt{n\choose i} \xi_i z^{i}.$$
where $\xi_j$ are independent real random variables satisfying the matching condition {\bf C1} with the $\tilde \xi_j$ being standard real Gaussian random variables.

For the Gaussian case, the polynomial $\tilde P_n(z) = \sum_{i=0}^{n}\sqrt{n\choose i} \tilde \xi_i z^{i}$ has exactly $\sqrt n$ real roots in expectation (see, for example, \cite{bleher1997correlations}, \cite{EK}). 
In their paper \cite{BD}, among other results, Bleher and Di extended this result to the non-Gaussian setting.
\begin{theorem}\cite[Theorem 5.3]{BD}\label{BDthm}
Let $\xi_j$ be iid random variables with mean 0 and variance 1. Assume furthermore that they are continuously distributed with sufficiently smooth density. Then
\begin{equation}
\lim _{n\to \infty }\frac{\E N_{P_n}(\R)}{\sqrt{n}}=1\nonumber.
\end{equation}
\end{theorem}

We refer the reader to the original paper \cite{BD} for the  precise description of ``sufficiently smooth". The same result with this assumption being removed is obtained in a recent work of Flasche-Kabluchko \cite{flasche2018real}.

Later, Tao and the second author in \cite[Theorem 5.6]{TVpoly} showed that the same result holds when the random variables $\xi_j$ are only required to be independent with mean 0, variance 1, and finite $(2+\ep)$-moments. Here we apply our framework to recover these results assuming the more flexible  Condition {\bf C1}, which allows a constant number of $\xi_j$ to have non-zero  means. Let us first start with a local universality result.

\begin{theorem}\label{uni_elliptic}(Universality for random elliptic polynomials)
Assume that the real random variables $\xi_j$ are independent and satisfy the matching condition {\bf C1} with the $\tilde \xi_j$ being standard real Gaussian random variables. Then there exist positive constants $C, c$ depending only on the constants in Condition {\bf C1} such that the following holds.

For any real number $x_0$ with $n^{-1/2+\ep} \le |x_0| \le 1$ and for any function $G: \mathbb{C} \to \mathbb{C}$ supported on $[x_0- 1/\sqrt{n}, x_0+1/\sqrt{n}]$ with  continuous derivatives up to order $6$ and $\norm{\triangledown^aG}_\infty\le n^{a/2}$ for all $0\le a\le 6$, we have
\begin{eqnarray} 
\left |\E\sum G\left (\zeta \right) -\E\sum G\left (\tilde \zeta\right) \right |\le C n^{-c},\nonumber
\end{eqnarray}
where the first sum runs over all roots $\zeta_1, \zeta_2, \dots$ of $P_n$, and the second  sum runs over all the roots $\tilde \zeta_1, \tilde  \zeta_2, \dots$ of $ \tilde P_n$. 
\end{theorem} 

\begin{remark}\label{remark_elliptic}
If $P_n$ satisfies the assumptions of Theorem \ref{uni_elliptic}, so does the polynomial $Q_n(z) = z^{n}P_n\left (\frac{1}{z}\right ) = \sum_{i=0}^{n} \sqrt{n\choose i} \xi_{n-i} z^{i}$. And since the roots of $Q_n$ are just the reciprocals of the roots of $P_n$, from the conclusion of Theorem \ref{uni_elliptic} for $Q_n$, one can obtain the corresponding universality result of $P_n$ on the domain $1\le |x_0| \le n^{1/2-\ep}$. 
\end{remark} 

Thanks to this remark, our result proves  universality  on the domain $ n^{-1/2+\ep}\le |x_0| \le n^{1/2-\ep}$. By showing that the contribution outside of this domain is negligible, we obtain the following more quantitative version of Theorem \ref{BDthm}.
\begin{cor}\label{mean_elliptic} Under the assumption of Theorem \ref{uni_elliptic}, we have
$$\E N_{P_n}(\R)=\sqrt{n} + O(n^{1/2-c})$$
where $c$ and the implicit constant only depend on the constants in Condition {\bf C1}.
\end{cor}

We give a short proof of these results in Section \ref{proof_elliptic}. We note that a corresponding statement can be made concerning the expected number of real roots on a fixed interval $[a, b]\subset \R$, using the same proof. 

\section{Application: Universality for Random Taylor series} \label{app5} 
Let $\Gamma$ denote the Gamma function. In a recent paper \cite{flasche2020expected}, Flasche and Kabluchko considered the following random series 
$$P(x) = \sum_{k=0}^{\infty} \xi_k c_k x^{k}$$
where the $c_k$ are real deterministic coefficients such that 
$$c_k^{2}=\frac{k^{\gamma-1}}{\Gamma(\gamma)}L(k)$$
for some constant $\gamma>0$ and some function $L: (0, \infty)\to \R$ satisfying $L(t)>0$ for sufficiently large $t$ and $\lim _{t\to\infty} \frac{L(\lambda t)}{L(t)}=1$ for all $\lambda>0$. For example, $L(x)$ is some power of $\log x$.

We follow the terminology in \cite{flasche2020expected} and call such a  function $L$ a {\it slowly varying} function and the function $P$ a  {\it random series with regularly varying coefficients}. The following is the main result of \cite{flasche2020expected}.

\begin{theorem}\cite[Theorem 1.1]{flasche2020expected}\label{kacseries_thm}
Assume that the random variables $\xi_k$ are iid real random variables with zero mean and unit variance. Then
$$\lim _{r\uparrow 1} \frac{\E N_{P}[0, r]}{-\log(1-r)}=\frac{\sqrt\gamma}{2\pi}.$$
\end{theorem}

We reprove  Theorem \ref{kacseries_thm} under the (slightly different) assumption that the random variables $\xi_k$ are independent (not necessarily identically distributed) real random variables with zero mean, unit variance, and uniformly bounded $(2+\ep)$-moments. As usual, we  allow that a few random variables have nonzero bounded mean, and so our result also applies to level sets. Our method also yields a polynomial rate  of convergence.

As before, we obtain this as a corollary of a stronger theorem establishing the 
 local universality of the roots. Let 
$$\tilde P(x) = \sum_{k=0}^{\infty} \tilde \xi_k c_k x^{k}$$
where the $\tilde \xi_k$ are independent standard Gaussian.

\begin{theorem}[Universality for random series with regularly varying coefficients]\label{kacseries_uni}
Let $k, l$ be nonnegative integers with $k+l\ge 1$. Assume that the real random variables $\xi_j$ are independent and satisfy the matching condition {\bf C1} with the $\tilde \xi_j$ being standard real Gaussian random variables. There exist positive constants $C', c$  depending only on the constants in Condition {\bf C1} such that the following holds.

Let $0 < \delta < 1$, and  let $x_1,\dots, x_k$ be real numbers  and $z_1, \dots, z_l$ be complex numbers satisfying  $1-2\delta \le |x_i|, |z_j|\le 1-\delta$ for all relevant $i,j$.  Let $G: \mathbb{R}^{k}\times\mathbb{C}^{l}\to \mathbb{C}$ by a function  supported on $\prod_{i=1}^{k}[x_i-10^{-3}\delta, x_i+10^{-3}\delta] \times \prod_{j=1}^{l}B (z_j, 10^{-3}\delta)$ with  continuous derivatives up to order $2(k+l)+4$ and $\norm{\triangledown^aG}_\infty\le \delta^{-a}$ for all $0\le a\le 2(k+l)+4$. Then 
	\begin{eqnarray} 
	\left |\E\sum G\left (\zeta_{i_1}, \dots, \zeta_{i_k}, \zeta_{j_1} , \dots, \zeta_{j_l}\right) 
	-\E\sum G\left (\tilde \zeta_{i_1}, \dots, \tilde \zeta_{i_k}, \tilde \zeta_{j_1}, \dots, \tilde \zeta_{j_l}\right) \right |\le C'\delta^{c},\nonumber
	\end{eqnarray}
	where the first sum runs over all $(k+l)$-tuples $(\zeta_{i_1}, \dots, \zeta_{i_k}, \zeta_{j_1}, \dots, \zeta_{j_l}) \in \R^{k}\times \C_{+}^{l}$ of the roots $\zeta_1, \zeta_2, \dots$ of $P$, and the second  sum runs over all $(k+l)$-tuples $(\tilde \zeta_{i_1}, \dots, \tilde \zeta_{i_k}, \tilde \zeta_{j_1}, \dots, \tilde \zeta_{j_l}) \in \R^{k}\times \C_{+}^{l}$ of the roots $\tilde \zeta_1, \tilde  \zeta_2, \dots$ of $ \tilde P$. 
\end{theorem}

\begin{cor}\label{kacseries_cor}
Under the assumption of Theorem \ref{kacseries_uni}, there exist positive constants $C'$ and $c$ such that the following hold.
\begin{enumerate}
\item For any $r\in (0, 1)$, 
$$\left|\E N_{P}[0, r] - \E N_{\tilde P}[0, r] \right |\le C$$
where $N_{P}[0, r]$ and $N_{\tilde P}[0, r]$ are the number of real roots of $P$ and $\tilde P$ in $[0, r]$, respectively. 
\item We have 
$$\lim _{r\uparrow 1} \frac{\E N_{P}[0, r]}{-\log(1-r)}=\frac{\sqrt\gamma}{2\pi}.$$
\end{enumerate}
\end{cor}
We prove Theorem \ref{kacseries_uni} and Corollary \ref{kacseries_cor} in Section \ref{kacseries_proof}.

After this paper has been finished,  the authors become aware of a very recent and interesting result of Flasche-Kabluchko \cite{flasche2018real} in which a completely different method is developed to study systematically the elliptic polynomial, Weyl polynomial, flat random analytic function, and hyperbolic random analytic function. As Flasche and Kabluchko mentioned in their paper, a similar approach has been applied to random trigonometric polynomials \cite{flasche} and random Taylor series \cite{flasche2020expected}.  Here we draw a quick comparison of the results.

 \begin{itemize} 
 	
 \item The results in \cite{flasche2018real}  prove the universality of 
the  density functions, while our results prove universality of all correlation functions. 
The authors of \cite{flasche2018real} do not seem to be aware of our paper  (which was put on the arxiv several months earlier)  and made a comparison with 
\cite{TVpoly}. However, the main result of \cite{TVpoly} is also about  universality of all correlation functions, 
but this critical point has been ignored. 
 
\item   \cite{flasche2018real} and related papers require that the random variables are identically distributed 
with finite second moment; our  method requires $(2+\ep)$-moment, but the variables do not need to be iid. 

\item The results in \cite{flasche}, \cite{flasche2020expected}, \cite{flasche2018real}  provide the limits as $n\to \infty$. Our  results  prove the limits with quantitative error terms. 
 
 \item Our  method  allows the coefficients to fluctuate. Specifically, in most of the applications in the above sections, a result stated for a random function
 $$F(x) = \sum_{k} \xi_k \phi_k(z)$$
 can readily be generalized 
 (with no significant changes in the proofs) to a random function
 $$G(x) = \sum_{k} c_k\xi_k \phi_k(z),$$
 where $c_k$ are deterministic coefficients that can take any values in the interval $[1/2, 2]$ (say). In this respect, the method in \cite{flasche2018real} which relies on assumptions such as \cite[Equation (6)]{flasche2018real} may be more susceptible to coefficients' fluctuations.

\end{itemize}

\section{Proof of Theorems \ref{gcomplex} and \ref{greal}}\label{app1_proof_1}

Before starting the proofs, let us mention two Jensen's inequalities that we use several times in this manuscript. It will be clear in the context which Jensen's inequality is used.
The first, and perhaps more popular, Jensen's inequality relates the value of a convex function of an integral to the integral of that convex function. In particular, for any convex function $\phi$ on the real line and any real integrable random variable $X$, we have
$$\phi\left (\E(X)\right )\le \E \phi(X).$$

The second Jensen's inequality provides an upper bound on the number of roots of an analytic function. Assume that $f$ is an analytic function on an open domain that contains the closed disk $\bar B(z, R)$. Then for any $r<R$, we have
\begin{equation}
N(B(z, r))\le \frac{\log \frac{M}{m}}{\log\frac{R^{2}+r^{2}}{2Rr}}\label{jensenbound}
\end{equation}
where $N(B(z, r))$ is the number of roots (including multiplicities) of $f$ in the open disk $B(z, r)$ and $M = \max_{w\in \bar B(z, R)} |f(w)|$, $m = \max_{w\in \bar B(z, r)} |f(w)|$. For completeness, we include a short proof of this inequality in Appendix \ref{proof_jensen}.

\subsection{Proof of Theorem \ref{gcomplex}}\label{pgcomplex}
We first state a few lemmas.  The first lemma reduces the theorem  to the case when the function $G$ {\it splits}, namely $G$ 
 is a product of functions of a single variable. In many applications, $G$ automatically takes this form. 
 This lemma was proved in \cite{TVpoly}. We include a short proof in Appendix \ref{fourier_proof}.
 
\begin{lemma}\label{fourier}
If Theorem \ref{gcomplex} holds for every function $G$ of the form 
\begin{equation}\label{h2}
 G(w_1,\dots, w_m) = G_1(w_1)\dots G_k(w_k)
 \end{equation} where for each $1\le i\le k$, $G_i:\mathbb{C}\to \mathbb{C}$ is a function supported in $B(z_i, 1/50)$ with continuous derivatives up to order $3$ and $\norm{\triangledown^aG_{i}}_\infty\le 1$ for all $0\le a\le 3$, then it holds for any function $G$ satisfying the hypothesis of Theorem \ref{gcomplex}. Similarly for Theorem \ref{greal}.
\end{lemma}

The next lemma plays a critical role in our approach, as it shows that the singularity problem at 0 (see the discussion in the last subsection of Section \ref{framework}) can be dealt with assuming anti-concentration at a single point.

\begin{lemma}\label{2norm}
Let $0<\delta_n, c_2<1$ and let $F_n$ be an entire function with $|F_n(w)|\ge \exp(-\delta_n^{-c_2})$ for some complex number $w$ and $|F_n(z)|\le \exp(\delta_n^{-c_2})$ for all $z\in B(w, 3/2)$. Then 
\begin{equation}
\int_{B(w, 1/2)} \left |\log\left |F_n(z)\right |\right |^{2} dz \le 720^2 \times\delta_n^{-6c_2}.\nonumber
\end{equation}

\end{lemma}

The constant $720^2= 518400$ is for explicitness and plays no specific role. Both this and the constant $6$ in the exponent can be reduced but we make no attempt to optimize these constants. The proof follows from ideas in \cite{DOV} and is included in Appendix \ref{2norm_proof}.

The following lemma shows that the logarithm function satisfies a universality property. It is a variant 
of a lemma in \cite{TVpoly} and we include the proof in Appendix \ref{logcomp_proof}.

\begin{lemma}\textbf{(Log-comparability)}\label{logcomp}
Assume that the coefficients $\xi_i$ and $\tilde \xi_i$ satisfy Condition {\bf C1} for some constants $N_0, \ep, \tau$. Let $\alpha_1$ be a positive constant and $k$ be a positive integer. Assume that there exists a constant $C>0$ such that the random functions $F_n$ and $\tilde F_n$ satisfy Condition {\bf C2} \eqref{cond-delocal} with parameters $\alpha_1$ and $C$. There exist positive constants $\alpha_0$ and $C'$ such that for any ${z_1}, \dots, {z_k}\in D_n + B(0, 1/10)$, and function $K:\mathbb{C}^k\to \mathbb{C}$ with continuous derivatives up to order $3$ and $\norm{\triangledown^a K}_\infty\le \delta_n^{-\alpha_0}$ for all $0\le a\le 3$, we have
\[\big|\E K\big(\log|F_n(z_1)|, \dots, \log|F_n(z_{k})|\big)-\E K\big(\log|\tilde {F_n}(z_1)|, \dots, \log|\tilde {F_n}(z_{k})|\big) \big|\le {C'}\delta_n^{\alpha_0}.
 \]
\end{lemma}
\begin{remark} \label{alpha0} 
Following the proof, one can set $\alpha_0 = \frac{3\alpha_1\ep }{10^{3}}$.
\end{remark}

\begin{proof}[Proof of Theorem \ref{gcomplex}] 
By Lemma \ref{fourier}, we can assume that the function $G$ has the form \eqref{h2}. We need to show that 
\begin{eqnarray}
\ab{\E \prod_{j=1}^{k} \left (\sum_{i} G_j (\zeta_i)\right )-\E \prod_{j=1}^{k} \left (\sum_{i} G_j (\tilde \zeta_i)\right )}\le C'\delta_n^{c},\label{du5}
\end{eqnarray}
for some constant $c >0$. By Green's formula, we have
\begin{equation}
\sum_{i} G_j({\zeta}_i)= \int_{\mathbb C}\log |F_n(z)|H_j(z)dz = \int_{B( z_j, 1/10)}\log |F_n(u_j)|H_j(u_j)du_j,\label{sat1}
\end{equation}
where $H_j(z) = \frac{1}{2\pi}\triangle G_j(z)$. Note that $supp(H_j)\subset B( z_j, 1/10)$ and $\norm{H_j}_\infty\le 1$ for all $z\in \C$, thanks to the assumption on $G$ in Theorem \ref{gcomplex}. (As usual, $\| f\|_\infty = \sup_{z \in \C } | f(z)| $.) When $F_n$ is identically 0, we assume by convention that the left-hand side and the right-hand side are 0.

%
%
%

Let $A$ be a sufficiently large constant and $c_1$ be a sufficiently small positive constant. For this proof,
it suffices to set  $c_1 := \frac{\alpha_0}{300 k^{2}}$ and $A := 2kC_1 + \frac{\alpha_1\ep}{60}$. This choice, together with the value of  $\alpha_0$ in Remark \ref{alpha0},  yields  the explicit values of $A$ and $c_1$ in the theorem. 
 
Let $\bar c_1 := 100k c_1$. The power $c$ in \eqref{du5}  can be chosen  (quite generously) to be $c_1$.
 
Let $K:\R \to \R$ be a smooth function with the following properties 

\begin{itemize}
	\item  $K$ is  supported on the interval $[-2\delta_n^{-\bar c_1}, 2\delta_n^{-\bar c_1}]$
	
	\item  $K(x) = x$ for all $x\in [-\delta_n^{-\bar c_1}, \delta_n^{-\bar c_1}]$ 
	
	\item  $||K^{(a)}||_\infty = O\left (\delta_n^{-\bar c_1}\right )$ for all $0\le a\le 3$ (where $K^{(l)}$ is the $l$-th derivative of $K$). 
	
	\item $|K(x)|\le |x|$ for all $x\in \R$.  \end{itemize} 
	

		Let $\Gamma := \prod_{j=1}^{k} B (z_j, 1/10)$ and $H(u) := \prod_{j=1}^{k} H_j(u_j)$ for $u :=(u_1, \dots, u_k)$.

By \eqref{sat1}, we have 

\begin{eqnarray}
\E \prod_{j=1}^{k} \left (\sum_{i} G_j (\zeta_i)\right ) = \E \int_{\Gamma} H(u)\prod_{j=1}^{k}\log |F_n(u_j)| du &=& A_1+A_2\nonumber
\end{eqnarray}
where 
$$A_1 := \E\int_{\Gamma} H(u)\prod_{j=1}^{k}K(\log|F_n(u_j)|) du,$$
$$A_2 := \E\int_{\Gamma} H(u)\left [\prod_{j=1}^{k}\log |F_n(u_j)| -\prod_{j=1}^{k} K(\log|F_n(u_j)|) \right ] du.$$

Let $\tilde A_1$ and $\tilde A_2$ be the corresponding terms  for $\tilde F_n$. Our goal is to show that 
\begin{equation} \label{goal1} A_1 + A_2 - \tilde A_1 - \tilde A_2 = O\left (\delta_n^{c}\right ).\end{equation}

By Lemma \ref{logcomp}, we have $A_1 - \tilde A_1 = O\left (\delta_n^{\bar c_1}\right )$.
 We next show that both $A_2$ and $\tilde A_2$ are of order $O\left (\delta_n^{ c_1}\right )$. It suffices to consider $A_2$, as the treatment of $\tilde A_2$ is similar.

 Let $\mathcal A_0$ be the event on which the following two properties hold
 
 \begin{itemize}
 	\item For all $1\le j\le k$, $|F_n(z'_j)|\ge \exp(-\delta_n^{-c_1})$ for some $z_j'\in B(z_j, 1/100)$
 	
 	\item  $|F_n(z)|\le \exp(\delta_n^{-c_1})$ for all $z\in B(z_j, 2)$. \end{itemize}

 	 By Conditions {\bf C2} \eqref{cond-smallball} and {\bf C2} \eqref{cond-bddn}, $\P(\mathcal A_0^{c}) \le C\delta_n^{A}$, where $\mathcal A_0^{c}$ is the complement of $\mathcal A_0$. 
We next  break up $A_2$ as follows
\begin{eqnarray}
 A_2 &=& \E\int_{\Gamma} H(u)\left [\prod_{j=1}^{k}\log |F_n(u_j)| -\prod_{j=1}^{k} K(\log|F_n(u_j)|) \right ] du \textbf{1}_{\mathcal A_0}+\E\int_{\Gamma} H(u) \prod_{j=1}^{k}\log |F_n(u_j)| du\textbf{1}_{\mathcal A_0^{c}} \nonumber\\
&& - \E\int_{\Gamma} H(u) \prod_{j=1}^{k} K(\log|F_n(u_j)|) du\textbf{1}_{\mathcal A_0^{c}} =: A_3 + A_4 - A_5\nonumber.
 \end{eqnarray}
For $A_5$, since $\| K\|_\infty  \le 2 \delta_n^{-\bar c_1}$ by construction and $A\ge 2k\bar c_1$, we have
$$|A_5|\le 2\delta_n^{-k\bar c_1}\P(\mathcal A_0^{c})\le 2 C\delta_n^{A -k\bar c_1} =O(\delta_n^{\bar c_1})
= O(\delta_n^{c_1}).$$ 

To bound $A_4$, from \eqref{sat1} and the boundedness of $H_j$, we have
$$\left |\int_{B(z_j, 1/100)} \log|F_n(u_j)|H_j(u_j)du_j\right |\le N_{F_n}(B(z_j, 1/100))=: N_j.$$ 
By H{\"o}lder's inequality for products, 
$$|A_4|\le \prod_{j=1}^{k} \left (\E N_j^{k}\textbf{1}_{\mathcal A_0^{c}}\right )^{1/k}.$$

We bound each term on the  right  using  H{\"o}lder's inequality as follows 
$$\E N_j^{k}\textbf{1}_{\mathcal A_0^{c}} \le \delta_n^{- kC_1}\P(\mathcal A_0^{c}) + \left (\E N_j^{k+1}\textbf{1}_{N_j\ge \delta_n^{-C_1}}\right )^{k/(k+1)}\left (\P\left (\mathcal A_0^{c}\right )\right )^{1/(k+1)} . $$

In our setting,  $A\ge kC_1 + (k+1)\bar c_1$,  the first term on the right-hand side is $O(\delta_n^{c_1})$. 
Moreover, Condition {\bf C2} \eqref{cond-poly} implies that the second term is $O(\P\left (\mathcal A_0^{c}\right )^{1/(k+1)} ) = O( \delta_n ^{c_1}) $. Thus, $A_4 =O (\delta_n^{c_1})$.

Finally, to bound $A_3$, we let $B$ be the (random) set of all $u\in \Gamma$ on which $\left |\log|F_n(u_j)|\right |\ge \delta_n^{-\bar c_1}$ for some $j$. Notice that if $u = (u_1, \dots, u_k ) \notin B$, then 
$K (\log |F_n (u_j)| ) = \log |F_n (u_j)|$ by the properties of $K$ and the definition of $B$. 
Moreover, for $u \in B$, $| K (\log |F_n (u_j)| ) | \le| \log |F_n (u_j)||$ as $|K(x)| \le |x| $ for all $x$.
It follows that 
\begin{eqnarray}
|A_3|&\le& 2 \E\int_{\Gamma} \left |\prod _{j=1}^{k}\log |F_n(u_j)| \right | \textbf{1}_{B}(u)du \textbf{1}_{\mathcal A_0}.
\end{eqnarray}

By H\"older's inequality, the right-hand side is at most 
$$ 2 \left [\E\int_{\Gamma} \left |\prod _{j=1}^{k}\log |F_n(u_j)| \right | ^{2} du \textbf{1}_{\mathcal A_0} \right ]^{1/2} \left [\E\int_{\Gamma} \textbf{1}_{B}(u) du \textbf{1}_{\mathcal A_0} \right ]^{1/2}. $$ 

By Lemma \ref{2norm}, on the event $\mathcal A_0$, we have 
\begin{equation} \label{onA0} \int_{B(z_j, 1/100)} \left |\log |F_n(u_j)| \right | ^{2} du_j =O(\delta_n^{-6c_1}). \end{equation}  It follows that 
$$\int_{\Gamma} \left |\prod_{j=1}^{k}\log |F_n(u_j)| \right | ^{2} du =O(\delta_n^{-6kc_1}).$$

On the other hand, by the definition of $B$,
 $$\int_{\Gamma} \textbf{1}_{B}(u) du \textbf{1}_{\mathcal A_0}= O \left (\textbf{1}_{\mathcal A_0}\sum_{j=1}^{k}\int_{B(z_j, 1/100)} \textbf{1}_{|\log|F_n(u_j)||\ge \delta_n^{-\bar c_1} } du_j \right ). $$

Furthermore, 
$$ \int_{B(z_j, 1/100)} \textbf{1}_{|\log|F_n(u_j)||\ge \delta_n^{-\bar c_1} } du_j \le \delta_n^{2\bar c_1}\int_{B(z_j, 1/100)}\left |\log |F_n(z)| \right |^{2} dz. $$

Using \eqref{onA0}, we obtain 
$$\E \int_{\Gamma} \textbf{1}_{B}(u) du \textbf{1}_{\mathcal A_0} =O( \delta_n^{2\bar c_1 } \delta_n^{-6k c_1}). $$
 
It follows that 
$$|A_3| = O\left ( \left ( \delta_n^{-6k c_1} \times  \delta_n^{2\bar c_1 } \delta_n^{-6k c_1} \right )^{1/2}\right  )
=O(\delta_n^{\bar c_1 - 6kc_1}) = O(\delta_n^{c_1}) $$ as we set $\bar c_1 > 7k c_1 $. The bounds on 
$|A_3|, |A_4|$ and $ |A_5| $ together imply $|A_2|= O(\delta_n ^{c_1}) $, concluding the proof. 
\end{proof}

\subsection{Proof of Theorem \ref{greal}} \label{pgreal}
By Lemma \ref{fourier}, it suffices to assume that $G$ can be decomposed into functions of single variables, namely 
$$G(x_1, \dots, x_k, z_1, \dots, z_l) = H_1(x_1)\dots H_k(x_k) G_1(z_1)\dots G_l(z_l)$$
where the  $H_i:\mathbb{R}\to\mathbb{C}$ and $G_j:\mathbb{C}\to \mathbb{C}$ are smooth functions supported on $[x_i-1/50, x_i+1/50]$ and $B(z_j, 1/50)$ (respectively) 
and satisfying 
\[|{\triangledown^{a}H_i}(x)|, |{\triangledown^{a}G_j}(z)|\le 1
\]
for any $x\in \R$, $z\in \C$ and $0\le a\le 3$.

In other words, one needs to show that
\begin{eqnarray}
\ab{\E \left(\prod_{i=1}^{k}X_{i}\right)\left(\prod_{j=1}^{l}Y_{j}\right)-\E \left(\prod_{i=1}^{k}\tilde X_{i}\right)\left(\prod_{j=1}^{l}\tilde Y_{j}\right)}\le C'\delta_n^{\bar c},\label{du6}
\end{eqnarray} for some constants $C', \bar c >0$, where 
$$X_{i} = \sum_{\zeta_s\in\mathbb{R}}H_i(\zeta_s), \quad \tilde X_{i} = \sum_{\tilde \zeta_s\in\mathbb{R}}H_i(\tilde \zeta_s), \quad Y_{j}= \sum_{\zeta_s\in\mathbb{C}_+}G_j(\zeta_s), \quad \tilde Y_{j}= \sum_{\tilde \zeta_s\in\mathbb{C}_+}G_j(\tilde \zeta_s).$$ (We use  $\bar c$ instead of $c$  to denote the exponent 
on the right hand side, since we  reserve $c$ for the exponent in Theorem \ref{gcomplex}, which we will use in the proof.)

The proof follows the ideas in \cite{TVpoly}. 
The first step is to show that the number of complex zeros near the real axis is small with high probability. Let  $c$ be the constant exponent  in Theorem \ref{gcomplex} corresponding to $k+l$. Following Remark \ref{rmkconstants}, we can set $c= \frac{\alpha_1\ep }{10^{5}(k+l)^{2}}$. 
 
With this choice of $c$, we set  $c_2 := \frac{c}{100}= \frac{\alpha_1\ep }{10^{7}(k+l)^{2}}$ and  $\gamma := \delta_n^{c_2}$. Let us also recall that in the statement of this theorem (Theorem \ref{greal}), $c_1= \frac{\alpha_1 \ep}{10^9 (k+l)^4}$, which is much smaller than $c_2$: $c_1= \frac{c_2}{100(k+l)^{2}}$.

\begin{lemma}\label{lmrepulsion} Under the assumptions of Theorem \ref{greal}, we have 
$$\P \left(N_{ F_n}{B( x,\gamma)}\ge 2\right) = O(\gamma^{3/2}), 
\qquad\text{for all } x\in \R \cap \left (D_n + B(0, 1/50)\right )$$
where the implicit constant depends only on the constants in Conditions {\bf C1} and {\bf C2} (but not on $n, \delta_n, D_n$ and $x$.
\end{lemma}
The power $3/2$ in the above lemma is not critical, we only need something strictly greater than 1. 

Assuming this lemma, the rest of the proof is relatively simple. For every $1\le i\le k$, consider the strip $S_i := [ x_i - 1/50, x_i + 1/50]\times [-\gamma/4, \gamma/4]$. We can cover $S_i$ by $O(\gamma^{-1})$ disks of the form $B( x, \gamma)$, where $x \in [ x_i-1/50, x_i+1/50]$.  Since $F_n$ has real coefficients, if $z$ is a root of $F_n$ in $S_i\backslash \R$, so is it conjugate $\bar z$. Using Lemma \ref{lmrepulsion} and the union bound,  we obtain
\begin{eqnarray}
\P (\text{there is at least 1 (or equivalently 2) root(s) in } S_i\backslash \mathbb R )
&=& O(\gamma^{-1}\gamma^{3/2}) = O(\gamma^{1/2})\label{du4}.
\end{eqnarray}

Define
$\mathfrak H_i(z) := H_i(\Re (z))\phi \left (\frac{4\Im (z)}{\gamma}\right )$, 
where $\phi: \R \to [0, 1]$ is a smooth function that is supported on $[-1,1]$, with $\phi(0)=1$ and $\norm{\phi^{(a)}}_{\infty}=O(1)$ for all $0\le a\le 3$. It is easy to see that $\mathfrak H_i$ is a smooth function supported on $S_i$ with  $\| \mathfrak H_i \| _{\infty}\le 1$, and $\big\| \triangledown ^a \mathfrak H_i\big\| _{\infty} = O(\gamma^{-a})$ for $0\le a\le 3$.

 Set 
 $\mathfrak X_i := \sum_{s} \mathfrak H_i(\zeta_s)$ and $D_i := \mathfrak X_{i} - X_{i}$. By the definitions of $\mathfrak X_{i}$ and $X_{i}$, $D_i = \sum_{\zeta_s\notin \R} \mathfrak H_i(\zeta_s)$. Our general strategy is to use $\mathfrak X_i$ to approximate $X_i$, then apply Theorem \ref{gcomplex} to $\mathfrak X_i$ and  finish the proof using a triangle inequality.

From \eqref{du4}, $D_{i} = 0$ with probability at least $1 - O(\gamma^{1/2})$. Notice that by the definition of $D_i$ and the fact that  $\| \mathfrak H_i \| _{\infty}\le 1$, 
 \begin{equation} \label{Di} |D_{i}| \le N_{ F_n}{B( x_i, 1/5)}. \end{equation}  
 
 By  \eqref{Di} and  Jensen's inequality \eqref{jensenbound}, 
 $$|D_{i}|\le N_{ F_n}{B( x_i, 1/5)} =O\left (\log \max_{w\in B(x_i, 2)} |F_n(w)|-\log\max_{z\in B(x_i, 1/5)} |F_n(z)|\right ) .$$

 By Conditions {\bf C2} \eqref{cond-smallball} and {\bf C2} \eqref{cond-bddn}, with probability at least $1-O(\delta_n^{A})$, there exists $z\in B(x_i, 1/100)$ such that both terms on the right-hand side are of order  $O\left (\delta_n^{-c_1}\right )$. Therefore, with probability at least $1-O(\delta_n^{A})$, we have $|D_i|\le N_{ F_n}{B( x_i, 1/5)}\le C'\delta_n^{-c_1}$ for some constant $C'$. For the rest of this proof, we denote $N_i:= N_{ F_n}{B( x_i, 1/5)}$.

Our next step is to bound  $ \E \ab{D_{i}}^{k+l}$. To start, we have 
\begin{equation} \label{real1} 
\E \ab{D_{i}}^{k+l}\le \E \left (|D_{i}|^{k+l}\textbf{1}_{N_i\le C'\delta_n^{-c_1}}\right ) + \E \left (N_i^{k+l}\textbf{1}_{ N_i > C'\delta_n^{-c_1}}\right ). \end{equation}

Since  $D_{i} = 0$ with probability at least $1 - O(\gamma^{1/2})$, 
\begin{equation}  \E \left (|D_{i}|^{k+l}\textbf{1}_{N_i\le C'\delta_n^{-c_1}}\right ) =O\left (\delta_n^{-c_1(k+l)}\gamma^{1/2}\right )=O\left (\delta_n^{-c_1(k+l)+c_2/2}\right ) =O\left (\delta_n^{c_1(k+l)^{2}}\right ) \nonumber
\end{equation} 
because $c_2\ge 4c_1(k+l)^{2}$.

For the second term in \eqref{real1}, we further break up the event $ N_i > C'\delta_n^{-c_1}$ into two events 
$$\Omega_1:= \delta_n^{-C_1}\ge N_i > C'\delta_n^{-c_1} \,\, {\rm and}\,\, \Omega_2:= \delta_n^{-C_1}\le N_i$$ where $C_1$ is the constant in the statement of Theorem \ref{greal}. We have
$$\E N_i^{k+l}\textbf{1}_{\Omega_1}\le \delta_n^{-C_1(k+l)}\P(\Omega_1) = O\left (\delta_n^{A-C_1(k+l)}\right )=O\left (\delta_n^{c_1(k+l)^{2}}\right ).$$
Moreover, by  H{\"o}lder's inequality,
$$\E N_i^{k+l}\textbf{1}_{\Omega_2}\le \P\left(\Omega_2\right )^{\frac{2}{k+l+2}} \left (\E N_i^{k+l+2}\textbf{1}_{\Omega_2}\right )^{\frac{k+l}{k+l+2}} =O\left (\delta_n^{A/(k+l+2)}\right )\left (\E N_i^{k+l+2}\textbf{1}_{\Omega_2}\right )^{\frac{k+l}{k+l+2}}.$$
Under the assumption of Theorem \ref{greal}, Condition {\bf C2} \eqref{cond-poly} holds for the parameter $k+l$, which provides  $\E N_i^{k+l+2}\textbf{1}_{\Omega_2} = O(1)$. As we set  $A$ to be much larger than $c_1$, it is easy to check that 
$$\E N_i^{k+l}\textbf{1}_{\Omega_2} =O\left (\delta_n^{A/(k+l+2)}\right ) = O\left (\delta_n^{c_1(k+l)^{2}}\right ).$$

Thus, 
\begin{equation}
 \E \left (N_{ F_n}{B( x_i, 1/5)}\right )^{k+l}\textbf{1}_{N_{ F_n}{B( x_i, 1/5)}\ge C'\delta_n^{-c_1}}=O\left (\delta_n^{A/(k+l+2)}\right )=O\left (\delta_n^{c_1(k+l)^{2}}\right ). \label{boundN}
\end{equation}
Combining all these bounds with \eqref{real1}, we obtain 
$$\E |D_i|^{k+l}= O\left (\delta_n^{c_1(k+l)^{2}}\right ).$$
 
Moreover, from the above bounds, we get
\begin{equation}
\E |{\mathfrak X}_{i}|^{k+l} \le \E N_i^{k+l}= \E N_i^{k+l}\textbf{1}_{N_i\le C'\delta_n^{-c_1}}+\E N_i^{k+l}\textbf{1}_{\Omega_1} + \E N_i^{k+l}\textbf{1}_{\Omega_2} =O\left (\delta_n^{-c_1(k+l)}\right ),\nonumber
\end{equation} where the main contribution comes from the first term. Similarly, $\E |{X}_{i}|^{k+l}=O\left (\delta_n^{-c_1(k+l)}\right )$.

Next, for each $1\le j\le l$, let $\mathfrak G_j(z) := G_j(z)\varphi(\text{Im}(z)/\gamma)$ where $\varphi$ is a smooth function on $\R$ supported on $[1/2, \infty)$ with $\varphi=1$ on $[1, \infty)$ and $\norm{\varphi^{(a)}_{\infty}}=O(1)$ for all $0\le a\le 3$.

Set  $\mathfrak Y_j := \sum_{s}\mathfrak G_j(\zeta_s)$. By similar reasoning, we have  $\E|\mathfrak Y_j - Y_j|^{k+l}=O\left (\delta_n^{c_1(k+l)^{2}}\right )$ and 
$$\max\left \{\E |\mathfrak Y_j|^{k+l}, \E |Y_j|^{k+l}\right \} = O\left (\delta_n^{-c_1(k+l)}\right ).$$

Now, we show  that the difference $\E\left |(\prod_{i=1}^{k}X_{i})(\prod_{j=1}^{l}Y_{j}) - (\prod_{i=1}^{k}\mathfrak X_{i})(\prod_{j=1}^{l}\mathfrak Y_{j} )\right |$ is small. Using the ``telescopic sum" argument,  we decompose the difference inside the abolute value sign into the sum of $k+l$ differences, in each of which exactly one of the $X_1, \dots, X_k, Y_1, \dots, Y_j$ is replaced by its counterpart, and then use the triangle inequality to finish. 
Let us bound the first difference; the argument for the  rest is the same. By H{\"o}lder's inequality and the previous bounds on $D_i, X_i, Y_i$ etc, we have 
\begin{eqnarray}
\E\left |X_1(\prod_{i=2}^{k}X_{i})(\prod_{j=1}^{l}Y_{j}) - \mathfrak X_{1}(\prod_{i=2}^{k} X_{i})(\prod_{j=1}^{l} Y_{j} )\right |&\le& \left (\E|D_1|^{k+l}\right )^{\frac{1}{k+l}}\prod_{i=2}^{k} \left (\E|X_i|^{k+l}\right )^{\frac{1}{k+l}}\prod_{j=1}^{l}\left (\E|Y_{j}|^{k+l}\right )^{\frac{1}{k+l}}\nonumber\\
&=&O\left (\delta_n^{c_1(k+l)}\prod_{k+l-1 \mbox{ terms}} \delta_n^{-c_1}\right )= O\left (\delta_n^{c_1}\right ).\nonumber
\end{eqnarray}
Thus,
\begin{equation}
\E\left |(\prod_{i=1}^{k}X_{i})(\prod_{j=1}^{l}Y_{j}) - (\prod_{i=1}^{k}\mathfrak X_{i})(\prod_{j=1}^{l}\mathfrak Y_{j} )\right |=O\left (\delta_n^{c_1}\right ).\nonumber 
\end{equation}

We can obtain the same bound for the corresponding terms  of $\tilde F_n$. 
Finally, from Theorem \ref{gcomplex}, we have
\begin{equation}
\left |\E(\prod_{i=1}^{k}\mathfrak X_{i})(\prod_{j=1}^{l}\mathfrak Y_{j} )-\E(\prod_{i=1}^{k}\tilde {\mathfrak X_{i}})(\prod_{j=1}^{l}\tilde{ \mathfrak Y_{j}} )\right |= O\left (\delta_n^{c_1}\right )\nonumber. 
\end{equation}

The desired estimate now follows from the triangle inequality. 

\begin{proof}[Proof of Lemma \ref{lmrepulsion}]
	The first step is to use 
	 Theorem \ref{gcomplex} to reduce to the Gaussian case.  Borrowing ideas from \cite[Chapter 2]{HKPV}, we handle the Gaussian case using Rouch\'e's theorem and various probabilistic 
	 estimates based on some properties of the Gaussian distribution. 
	 
	 For this proof, we let $\tilde \xi_1, \dots, \tilde \xi_n$ be Gaussian random variables with unit variance and satisfying $\E \tilde \xi_i = \E \xi_i$ for each $1\le i\le n$. 
	 
	 Let $H:\C\to [0, 1]$ be a non-negative smooth function supported on $B(x, 2\gamma)$, such that $H=1$  on $B(x, \gamma)$ and $|\triangledown ^a H|\le C\gamma^{-a}$ for all $0\le a\le 8$.

	 Applying Theorem \ref{gcomplex} to $H$, we obtain 
	\begin{equation} \label{roots1} 
	\P (N_{ F_n}{B( x, \gamma)}\ge 2)\le
	\E \sum_{i\neq j} H( \zeta_i)H( \zeta_j) \le\E \sum_{i\neq j} H( {\tilde \zeta_i})H( \tilde {\zeta}_j)+ O(\delta_n^c\gamma^{-8}). \end{equation}

	The definition of $\gamma $ guarantees (via a trivial calculation)  that $O( \delta_n^c\gamma^{-8})
	=O(\gamma^{3/2})$, with room to spare. Thus, it remains to show 
	\begin{equation} \label{roots2}  \E \sum_{i\neq j} H( {\tilde \zeta_i})H( \tilde {\zeta}_j) =O(\gamma^{3/2}). \end{equation}

	Set  $N := N_{ {\tilde F_n}}{B( x, 2\gamma)}$; 
	 we bound the LHS of \eqref{roots2} from above by
	\begin{equation} \label{roots3}  \E N^{2}\textbf{1}_{ N\ge C'\delta_n^{-c_1}} + \E N(N-1) \textbf{1}_{ N < C'\delta_n^{-c_1}}. \end{equation}

	Using the same  argument as in the proof of  \eqref{boundN}, we can show that  $$\E N^{2}\textbf{1}_{ N\ge C'\delta_n^{-c_1}}=O\left (\delta_n^{A/(k+l+2)}\right)=O(\gamma^{3/2}).$$

	Thus, it remains to show that $\E N(N-1) \textbf{1}_{ N < C'\delta_n^{-c_1}} = O(\gamma^{3/2}) $. Since 
	 $$ \E N(N-1) \textbf{1}_{ N < C'\delta_n^{-c_1}}  \le C'^{2}\delta_n^{-2 c_1} \P( N \ge 2), $$ it suffices to 
	 prove 
	\begin{equation} 
	\P (N \ge 2) = \P(N_{\tilde{F_n}}{B(x,2\gamma)}\ge 2)  =O(\delta_n^{2c_1}\gamma^{3/2}). \label{repulsion_gau}
	\end{equation}
	
 	Thus, we have reduced the problem to the Gaussian setting. 
	Let $g(z) := \tilde F_n(x) + \tilde F_n'(x)(z-x)$ and $p(z) := \tilde F_n(z) - g(z)$.
	By Condition {\bf C2} \eqref{cond-delocal}, for any fixed $x$, we have $\tilde F_n(x) \tilde F_n '(x) \neq  0$ with probability 1. So, $g(z)$ has exactly one 
	root. Thus, by  Rouch\'{e}'s theorem, 
	$$\P(N_{\tilde{F_n}}{ B(x,2\gamma)}\ge 2) \le \P\left (\min_{z\in \partial B(x, 2\gamma)}|g(z)|\le \max_{z\in\partial B(x, 2\gamma)}|p(z)|\right ).$$
	
	In the rest of the proof, we bound the right-hand side. We are going to show that with (appropriately) high probability,  $\min_{z\in \partial B(x, 2\gamma)}|g(z)|$ is not too small and $ \max_{z\in\partial B(x, 2\gamma)}|p(z)|$ is not too large.

 For every $z\in B(x, 4\gamma)$, we have $p(z) = \sum_{j=1}^{n} \tilde \xi_j v_j(z)$ where $v_j(z) = \phi_j(z) - \phi_j(x) + (z-x) \phi_j'(z)$. Thus
	\begin{eqnarray}
	|v_j(z)| &\le& |z-x|^{2}\sup _{w\in B(x, 2\gamma)} |\phi_j''(w)| = O\left (\gamma^{2} \sup _{w\in B(x, 2\gamma)} |\phi_j''(w)|\right ).\nonumber 
	\end{eqnarray}
	
	By Condition {\bf C2} \eqref{cond-repulsion}, 
	\begin{equation} \label{exp1} |\E p(z)|=O\left (\gamma^{2} \sum_{j=1}^{n} |\E \tilde \xi_j|\sup_{w\in B(x, 1)}|\phi_j''(z)|\right )=O\left (\delta_n ^{2c_2-c_1}\sqrt{\sum_{j=1}^{n} |\phi_j(x)|^{2}}\right ), \end{equation} and 
	\begin{equation}
	\Var (p(z))=O\left ( \gamma^{4} \sum_{i=1}^{n}\sup _{w\in B(x, 2\gamma)} |\phi_j''(w)|^{2}\right) = O\left (\delta_n^{4c_2-c_1} \sum_{j=1}^{n} |\phi_j(x)|^{2}\right) = O\left (\delta_n^{4c_2-c_1} \Var (\tilde F_n (x))\right)\label{varp}.
	\end{equation}

	Set $t  := \delta_n^{2c_2-c_1} \sqrt{\Var (\tilde F_n (x))}$.  The previous estimates show that  $|\E p(z)| = O(t)$ and $\Var(p(z)) = O(t^{2}\delta_n^{c_1})$ for all $z\in B(x, 4\gamma)$. We will show the following concentration inequality
	\begin{equation}
	\P\left (\max_{z\in\partial B(x, 2\gamma)} |p(z)-\E p(z)|\ge \frac{1}{2}t\right )=O(1)\exp\left (-\frac{t^{2}}{100\max_{z\in B(x, 4\gamma)}\Var (p(z))}\right )=O\left (\gamma^{16/10}\delta_n^{2c_1}\right).\label{concenp1}
	\end{equation}

	Set $\bar p(z) := p(z)-\E p(z)$. For  any $z\in \partial B(x, 2\gamma)$, by Cauchy's integral formula, 
	\begin{eqnarray}
	|\bar p(z)|&\le& \int_0^{2\pi}\frac{|\bar p(x + 4\gamma e^{i\theta})|}{|z - x - 4\gamma e^{i\theta}|}4\gamma\frac{d\theta}{2\pi} \le 2\int_0^{2\pi}|\bar p(x + 4\gamma e^{i\theta})|\frac{d\theta}{2\pi}\nonumber\\
	&\le&\max_{w\in B(x, 4\gamma)}\sqrt{\Var (p(w))}\int_0^{2\pi}\frac{|\bar p(x + 4\gamma e^{i\theta})|}{\sqrt{\Var (\bar p(x + 4\gamma e^{i\theta}))}}\frac{d\theta}{2\pi}\nonumber.
	\end{eqnarray}
	
	Hence, by Markov's inequality,
	\begin{equation}
	\P(\max_{z\in \partial B(x, 2\gamma)} |\bar p(z)|\ge t)\le \E \exp\left( \left (\int_0^{2\pi}\frac{|\bar p(x + 4\gamma e^{i\theta})|}{10\sqrt{\Var (\bar p(x + 4\gamma e^{i\theta}))}}\frac{d\theta}{2\pi}\right )^{2}\right ) e^{-t^{2}/100\max_{z\in B(x, 4\gamma)}\Var (p(z))}.\nonumber
	\end{equation}
	
	Using  Jensen's inequality for convex functions and Fubini's theorem, we obtain 
	\begin{eqnarray}
	&&\E\exp\left (\left (\int_0^{2\pi}\frac{|\bar p(x + 4\gamma e^{i\theta})|}{10\sqrt{\Var (\bar p(x + 4\gamma e^{i\theta}))}}\frac{d\theta}{2\pi}\right )^{2}\right )\le \int_0^{2\pi}\E\exp\left (\frac{|\bar p(x + 4\gamma e^{i\theta})|^{2}}{100 {\Var (\bar p(x + 4\gamma e^{i\theta}))}}\right )\frac{d\theta}{2\pi}.\nonumber
	\end{eqnarray}
	The right-hand side is $O(1)$  by basic properties of the Gaussian distribution. (Notice that 
	$p(z)$, for any fixed $z$ is a Gaussian random variable.)  This proves \eqref{concenp1}. Using the bound $|\E p(z)| = O(t)$ for all $z\in B(x, 2\gamma)$, one concludes that with probability at least $1 - O\left (\gamma^{16/10}\delta_n^{2c_1}\right)$, 
	\begin{equation}  \label{maxx1} \max_{z\in \partial B(x, 2\gamma)} |p(z)| \le Kt, \end{equation}
	for some constant $K>0$.

	Now, we address $g(z)$; since  $g$ is a linear function with real coefficients, we have 
	$$\min_{z\in \partial B(x, 2\gamma)} |g(z)| = \min \{|g(x -  2\gamma)|, |g(x +  2\gamma)| \}, $$ which reduces the task to obtaining lower bounds for  the two end points only.
	
	Note that $g(x+  2\gamma)$ is normally distributed with standard deviation
	\begin{equation}
	\sqrt{\Var(g(x + 2\gamma))} = \sqrt{\sum_{j=1}^{n} (\phi_j (x) + 2\gamma \phi_j'(x))^{2}} \ge \sqrt{\sum_{j=1}^{n} \phi_j ^{2}(x) } - 2\gamma \sqrt{\sum_{j=1}^{n} \phi_j'^{2}(x)}\ge 1/2 \sqrt{\sum_{j=1}^{n} \phi_j ^{2}(x) }\nonumber
	\end{equation}
	where in the last two inequalities, we used the triangle inequality and then Condition {\bf C2} \eqref{cond-repulsion}. Note that by the definition of $t$,	
	 $$ \sqrt{\sum_{j=1}^{n} \phi_j ^{2}(x)} = \sqrt{ \Var (\tilde F_n(x))} = t\delta_n^{-2c_2+c_1}.$$

Since $g(x+2\gamma)$, as a random variable,  is a real Gaussian with density bounded by $\frac{1}{2\sqrt{\Var g(x+2\gamma)}}\le\frac{\delta_n^{2c_2-c_1}}{t}$, we have for any constant $K>0$, 	 	 
	\begin{equation}
	\P(|g(x+  2\gamma)| \le K t)=O\left ( \delta_n^{2c_2-c_1}\right ) =O\left ( \delta_n^{2c_1}\gamma ^{3/2}\right )\nonumber. 
	\end{equation}
	
	In the last inequality we used the fact that $c_2$ is set to be much larger than $c_1$; see the paragraph following \eqref{du6}. 
	
	We can prove a similar statement for $g(x-2\gamma) $. Thus we can conclude that for any constant $K>0$, 
	\begin{equation} \label{minn1} 
	\P\left (\min_{z\in \partial B(x, 2\gamma)} |g(z)| \le Kt \right )= O\left (\delta_n^{2c_1}\gamma ^{3/2}\right ).\nonumber
	\end{equation}
	Combining \eqref{minn1}  and \eqref{maxx1}, we conclude   the proof of Lemma \ref{lmrepulsion}.
\end{proof}

\section{ Proof of Theorem \ref{real}}\label{proof-main}

In this section, we prove Theorem \ref{real} by applying Theorem \ref{greal}. By dividing  the coefficients $c_i$ and $d_i$ by their maximum modulus, it suffices to assume that $\max_{0\le j\le n}\{|c_j|, |d_j|\}= 1$. For the sake of simplicity, we assume all random variables have mean 0; the more general setting in Condition {\bf C4} can be dealt with via a routine modification. 

Our crucial new ingredient is the following lemma, which is a generalization of 
a classical result of Tur\'an \cite{turan1953}.

\begin{lemma}\cite[Chapter I]{Nazarov94}  \label{turan111}
	For $i = \sqrt{-1}$, let
	$$p(t) = \sum_{k=0}^{h}a_k e^{i\lambda_k t}, \quad a_k \in \C, \quad\lambda_0<\lambda_1<\dots< \lambda_h \in \R.$$
	Then for any interval $J\subset \R$ and any measurable subset $E\subset J$ of positive measure, we have
	$$\max_{t\in J} |p(t) |\le \left (\frac{C|J|}{|E|}\right )^{h}\sup _{t\in E} |p(t) |$$
	where $C$ is an absolute constant.
\end{lemma}

We shall apply Theorem \ref{greal} to the function $F_{2n+1}(z) := P_n(10^{4}Cz/n)$ and the number of summands is $2n+1$ in place of $n$ (so we only care about $F_{k}$ where $k$ is odd). The corresponding parameters are $\delta_{2n+1} := 1/n$, and $D_{2n+1} := \{z: |\Im(z)| \le 1/10^{4}\}$.  The functions $\phi_i$ in \eqref{F} are 
$$\phi_1(z) = c_0, \phi_2(z) = c_1 \cos(z), \dots, \phi_{n+1}(z) = c_n \cos(n z),$$
$$ \phi_{n+2}(z) = d_1\sin(z), \dots, \phi_{2n+1}(z) = d_n \sin(nz)$$
and the random variables $\xi_1, \dots, \xi_{{2n+1}}$ in \eqref{F} will be $\xi_0, \dots, \xi_n, \eta_1, \dots, \eta_n$, respectively. The constant $10^{4}$ is chosen rather arbitrarily, any  sufficiently large constant would work.

To deduce Theorem \ref{real}  from Theorem \ref{greal}, for this model, we set $\alpha_1=1/2$ and $C_1$ to be any constant larger than $1$. We only need to show that for any positive constants $A, c_1$, there exists a constant $C$ for which Condition \textbf{C2} holds with parameters $(k, C_1, \alpha_1, A, c_1, C)$. 

For Condition {\bf C2} \eqref{cond-poly}, notice that the periodic function $P_n$ has at most $2n$ complex zeros in the region $[a, a+2\pi)\times \R \subset \C$ for any $a\in \R$. Indeed, let $w = e^{iz}$ then 
$$w^{n}P(z) = \frac{1}{2} \left (\sum_{k = 0}^{n} \xi_k (w^{n+k}+w^{n-k}) -i \sum_{k = 1}^{n} \eta_k (w^{n+k}-w^{n-k})\right )$$
which is a polynomial of degree $2n$ in $w$ and has at most $2n$ zeros. For each $w$ there is only one $z$ in the above region that corresponds to $w$. Thus this condition holds trivially for any constant 
$C_1 >1$, as the left hand side of  Condition {\bf C2} \eqref{cond-poly} becomes zero.

Now we address (the critical) Condition {\bf C2} \eqref{cond-smallball}. 
We will prove the following   stronger statement that for every positive constants $c_1, A$, there exists a constant $C'$ such that the following holds. For every complex number $z_0$, there exists a real number $x$ such that $|x-z_0|\le |\Im (z_0)| + \frac{1}{n}$ and 
$$\P\left ( |P(x)|\le \exp(-n^{c_1})\right )\le C'n^{-A}.$$

Let $x_ 0 = \Re(z_0)$ and $I = [x_0-\frac{1}{n}, x_0+\frac{1}{n}]$. By conditioning on the random variables $\eta_i$  and replacing $A$ by $2A$, it suffices to show that there exists $x\in I$ for which 
\begin{equation}
\sup_{Z\in \R}\P\left ( \left |\sum_{j=0}^{n} c_j\xi_j \cos(jx)-Z \right |\le e^{-n^{c_1}}\right )\le C'n^{-A/2}.\label{anti_trig}
\end{equation}
Now let us recall the definition of $\mathcal I_{0}$ in Condition {\bf C3}. We would like to point out  that in this part of the proof, we only use the fact that 
the size of $\mathcal I_{0}$ is of order $\Theta (n)$. 

We shall prove a more general version which will be useful for all of the remaining models in this manuscript.

\begin{lemma}\label{lmanti_concentration}
	Let $\CE$ be an index set of size $N\in \N$, and let $(\xi_j)_{j\in \CE}$ be independent random variables satisfying the moment Condition {\bf C1} \eqref{cond-moment}. Let $(e_j)_{j\in \CE}$ be deterministic (real or complex) coefficients with $|e_j|\ge \bar e$ for all $j$ and for some number $\bar e\in \R_+$. Then for any $A\ge 1$, any interval $I\subset\R$ of length at least $N^{-A}$, there exists an $x\in I$ such that
	$$\sup_{Z\in \R}\P\left (\left| \sum_{j\in \CE} e_j \xi_j \cos(jx)-Z\right |\le \bar e N^{-16A^{2}}\right )= O_A\left (N^{-A/2}\right )$$  
	where the implicit constant depends only on $A$ and the constants in Condition {\bf C1} \eqref{cond-moment}. 
\end{lemma}

Assuming this Lemma, we condition on the random variables $(\xi_j)_{j\notin \mathcal I_{0}}$ and apply the Lemma with $\CE := \mathcal I_{0}, e_j :=c_j, N:=|\mathcal I_{0}|=\Theta(n)$ to obtain \eqref{anti_trig} directly with $\bar e = \Theta(1)$.

\begin{proof}[Proof of Lemma \ref{lmanti_concentration}]
	We will prove Lemma \ref{lmanti_concentration} in three steps. In the first (and most important) step, we handle the case where $\xi_i$ are iid 
	Rademacher. In the second step, we handle the case where the $\xi_i$ have symmetric distributions. In the final step, we 
	address the most general setting.

	{\it Step 1.}  $\xi_i$ are iid Rademacher (that is, $\P(\xi_i=1)=\P(\xi_i=-1)=1/2$). The key 
	ingredient in this step is  the following inequality, which is a variant of a result of Hal\'asz \cite{halasz1977estimates}; see also \cite[Cor 7.16]{taovubook}, \cite[Cor 6.3]{nguyenvusurvey}) for relevant estimates. Before stating the result, we recall a definition of multi-sets: a multi-set is a collection of unordered elements in which each element can appear more than once.

	\begin{lemma}\label{halasz-inequality} 
		Let $\ep_1, \dots, \ep_n$ be independent Rademacher random variables. Let $a_1, \dots, a_n$ be real numbers and $l$ be a fixed integer. Assume that there is a constant $a>0$ such that  for any 
		two different  multi-sets $\{i_1, \dots, i_{l'}\}$ and $\{j_1, \dots, j_{l''}\}$ where $l'+ l''\le 2l$,  $|a_{i_1}+\dots + a_{i_{l'}} - a_{j_1}-\dots - a_{j_{l''}}|\ge a$. Then  
		$$\sup_{Z\in \R} \P\left (\left |\sum_{j=1}^{n}a_j\ep_j- Z\right |\le a n^{-l} \right ) = O_{l}(n^{-l}).$$
	\end{lemma} 
	
	For the sake of completeness, we present a short proof of this lemma in Appendix \ref{proof_Halasz}.
	
	\vskip2mm  
	There exists a subset $\CE'\subset \CE$ of size at least half the size of $\CE$ such that either for all $i\in \CE'$, $|\Re(e_i)|\ge \bar e/2$ or for all $i\in \CE'$, $|\Im(e_i)|\ge \bar e/2$. Since 
	\begin{equation}
	\P\left ( \left |\sum_{j\in \CE} e_j\xi_j \cos(jx)-Z \right |\le \bar e N^{-16A^{2}}\right )\le \P\left ( \left |\sum_{j\in \CE} \Re(e_j) \xi_j\cos(jx)-\Re(Z) \right |\le \bar e N^{-16A^{2}}\right )\nonumber
	\end{equation}
	and 
	\begin{equation}
	\P\left ( \left |\sum_{j\in \CE} e_j\xi_j \cos(jx)-Z \right |\le \bar e N^{-16A^{2}}\right )\le \P\left ( \left |\sum_{j\in \CE} \Im(e_j) \xi_j\cos(jx)-\Im(Z) \right |\le \bar e N^{-16A^{2}}\right )\nonumber,
	\end{equation}
	we can, by conditioning on the $(\xi_j)_{j\notin \CE'}$ and replacing $\CE$ by $\CE'$,  assume that the $e_i$ are real and $Z$ is real. This allows us to apply Lemma \ref{halasz-inequality}. 
	
	In order to apply Lemma \ref{halasz-inequality}, 
	we  first show that there exists an $x\in I$ such that for every 2 distinct multi-sets $\{i_1, \dots, i_{A'}\}$ and $\{j_1, \dots, j_{A''}\}$ in $\CE$ with $A' + A''\le 2A$, we have 
	\begin{equation} \label{H1} \left |\sum_{t =1}^{A'} e_{i_t}\cos(i_t x) - \sum_{t =1}^{A''} e_{j_t}\cos(j_t x)\right |> \bar e N^{-16A^{2}}N^{A}. 
	\end{equation}

	Let us fix  such two multi-sets and let $$h(x) := \sum_{t =1}^{A'} e_{i_t}\cos(i_t x) - \sum_{t =1}^{A''} e_{j_t}\cos(j_t x). $$ Let  $E := \{x\in I: |h(x)|\le \bar e N^{-16A^{2}}N^{A}\}$. Since $h$ can be written in terms of exponential polynomials with $4A$ frequencies, we can apply Lemma \ref{turan111} to obtain
	\begin{equation}
	\max_{[0, 2\pi]} |h| \le \left (\frac{C'}{|E|}\right )^{4A}\sup _{E}|h|.\label{turan}
	\end{equation}
	By the definition of $E$, the right-hand side is bounded from above by $\left (\frac{C'}{|E|}\right )^{4A} \bar e N^{-16A^{2}}N^{A}$. To bound the left-hand side from below, observe from orthogonality of the functions $\cos kx$ that
	\begin{eqnarray}
	2\pi \max_{[0, 2\pi]}|h|^{2}&\ge& \int_{0}^{2\pi} |h|^{2}dx \ge \pi  \bar e^{2} ,\label{H3} 
	\end{eqnarray} as all $|e_i|$ with $i\in \CE$ is at least $\bar e$.

	Therefore, from \eqref{turan}, we get
	$|E|= O_A(N^{-4A+1/4})$. Since there are only $O(N^{2A})$ choices for the sets $A'$ and $A^{''}$,  we conclude that every $x$ in $I$, except for a set of Lebesgue measure at most $O_A(N^{-2A+1/4}) =o_A(|I|)$, satisfies \eqref{H1}.
	
	To conclude the proof, we use \eqref{H1} with Lemma \ref{halasz-inequality}. By setting $a :=\bar e N^{-16A^{2}}N^{A}$ and $l:=A$, Lemma \ref{halasz-inequality} gives
	\begin{equation} \label{H2} 
	\sup_{Z\in \C}\P\left ( \left |\sum_{j\in \CE} e_j\xi_j \cos(jx)-Z \right |\le \bar e N^{-16A^{2}}\right )
	=O_A (N^{-A}) .
	\end{equation}

	This proves Lemma \ref{lmanti_concentration} for the Rademacher case. 
	
	\vskip2mm
	
	{\it Step 2. } In this step, we consider the case where random variables $\xi_j$ have symmetric distributions. In this case, 
	$(\xi_j)_j$ and $(\xi_j\ep_j)_j$ have the same distribution where $\ep_j$ are independent Rademacher random variables that are independent of the $\xi_j$. Thus,  the claimed statement is equivalent to 
	\begin{equation}
	\sup _{Z\in \R} \P\left ( \left |\sum_{j\in \CE} e_j\xi_j \ep_j\cos(jx)-Z \right |\le \bar e N^{-16A^{2}}\right ) = O_A\left (N^{-A/2}\right )\label{reduce_rademacher1}
	\end{equation}
	for some $x\in I$.

	The natural way to  prove this is to use the (standard) conditioning argument, one fixes  all $\xi_j$ and uses 
	the Rademacher variables as the only random source, going  back to Step 1. 
	However, the situation here is more delicate, as 
	$x$ may not be the same in each evaluation of $\xi_j$. We handle this extra complication by 
	proving the stronger statement that 	
	\begin{equation} \label{H5} \dashint_{I}\sup _{Z\in \R} \P\left ( \left |\sum_{j\in \CE} e_j\xi_j \ep_j\cos(jx)-Z \right |\le \bar e N^{-16A^{2}}\right )dx= O_A(N^{-A/2})\end{equation} 
	where $\dashint_I fdx:= \frac{1}{|I|}\int_{I}fdx$.
	
	The left-hand side is at most $ \dashint_{I}\E_{(\xi_j)}\sup _{Z\in \R} \P_{(\ep_j)}\left ( \left |\sum_{j\in \CE} e_j\xi_j \ep_j\cos(jx)-Z \right |\le \bar e N^{-16A^{2}}\right )dx$.
	
	By  Fubini's theorem, it suffices to show that
	\begin{equation} \label{H6} 
	\E_{(\xi_j)}\dashint_{I}\sup _{Z\in \R} \P_{(\ep_j)}\left ( \left |\sum_{j\in \CE} e_j\xi_j \ep_j\cos(jx) -Z\right |\le \bar e N^{-16A^{2}}\right )dx= O _A(N^{-A/2}). 
	\end{equation}

	We first  show that 
	with high probability, there are $\Theta (N)$ indices $j \in \CE$ such that $|\xi_j | = \Theta (1) $, which is needed 
	to guarantee \eqref{H3}. 	Assume, for a moment, that $\P (| \xi_j| < d )\ge 1-d$ for some small positive constant $d$.
	Since the random variables $\xi_j$ are symmetric, they have mean 0. Using the boundedness of the $(2+\ep)$ central moment of $\xi_j$ (Condition {\bf C1}), 
	and the fact that $\xi_j$ has variance 1,  we have 
	\begin{equation} 
	\E |\xi_j|^2=  1 = \E|\xi_j|^{2} \textbf{1}_{|\xi_j|< d} + \E|\xi_j|^{2} \textbf{1}_{|\xi_j|\ge d}\le d^{2} + d^{\ep/(2+\ep)}(\E|\xi_j|^{2+\ep})^{2/(2+\ep)}\le d^{2} + d^{\ep/(2+\ep)}\tau^{2/(2+\ep)}.\nonumber
	\end{equation}
	
	Thus, if  $d$ is small enough (depending on $\tau$ and $\ep$), we have a contradiction. 
	Hence, there is a constant $d>0$ such that  $\P(|\xi_j|< d)\le 1-d$. Now, by Chernoff's inequality, with probability at least $1 - e^{-\Theta(N)}$, there are at least $\Theta(N)$ indices $j\in \CE$ for which $|\xi_j|\ge d$.
	On the event that this happens, we condition on the $\ep_j$ where $|\xi_j|<d$ and use Step 1 to conclude that outside a subset of $I$ of measure at most $O_A(N^{-2A+1/4})$, we have
	$$\sup_{Z\in \C}\P_{(\ep_j)}\left ( \left |\sum_{j\in \CE} e_j\xi_j \ep_j\cos(jx)-Z \right |\le \bar e N^{-16A^{2}}\right ) =O_A(N^{-A}).$$

	Therefore, the left-hand side of \eqref{H6} is at most 
	$$ e^{-\Theta(N)}\dashint_{I} 1dx +O_A(N^{-2A+1/4}/|I|)+ O_A\left (\dashint_{I} N^{-A}dx\right ) = O_A(N^{-A+1/4}) = O_A(N^{-A/2}),$$ completing the proof for this case.

	\vskip2mm {\it Step 3.} Finally, we address  the general case.  Let $\xi_j'$ be independent copies of $\xi_j$, $j\in \CE$. Then the variables  $\xi_j'' := \xi_j - \xi_j'$  are symmetric and have uniformly bounded $(2+\ep)$-moments. 
	By Step 2, we have 
	\begin{eqnarray}
	&&\left [\P\left ( \left |\sum_{j\in \CE} e_j\xi_j \cos(jx)-Z\right |\le \bar e N^{-16A^{2}}\right )\right ]^{2}
	\le \P\left ( \left |\sum_{j\in \CE} e_j\xi_j'' \cos(jx) \right |\le 2 \bar e N^{-16A^{2}}\right )\le O_A(n^{-A})\nonumber
	\end{eqnarray}
	where in the last inequality, we decompose the disk $B\left (0, 2 \bar e N^{-16A^{2}}\right )$ into $O(1)$ disks of radius $\bar e N^{-16A^{2}}$ (not necessarily centered at 0) before applying Step 2. Taking square root of both sides, we obtain Lemma \ref{lmanti_concentration}. 
\end{proof}

The remaining conditions are easy to check. 
Condition {\bf C2} \eqref{cond-bddn} follows from the following lemma.
\begin{lemma}
	For any positive constants $A$, $c_1$ and $C$,  we have,  with probability at least $1 - O(n^{-A})$,  $\log M \le n^{c_1}$, where $M := \max\{|P(z)|: |\Im(z)|\le C/n\}$.
\end{lemma}

\begin{proof}
	For every $1\le j\le n$, we have $|e^{ijz}| = e^{-j\Im (z)}\le e^{C}$. And so,
	\begin{equation}
	\max _{|\Im(z)|\le C/n, 1\le j\le n} \{ |\cos jz|, |\sin jz|\}\le e^{C}.\label{trigbound}
	\end{equation}	
	Let $B$ be the event on which  $|\xi_j|\le n^{A/2 +1 }$ for all $0\le j\le n$.
	Notice that on the complement $B^{c}$ of $B$,  $\log M= o\left (n^{c_1}\right ) $ for any constant $c_1>0$. By Chebyshev's inequality
	(exploiting  the fact that $\E |\xi_i|^2 =1$)  and the union bound, we have
	$$\P(B^{c})\le  \frac{n}{n^{A +2} }  = o(n^{-A}), $$ completing the proof. \end{proof}

Finally Condition {\bf C2} \eqref{cond-delocal} follows from the following lemma

\begin{lemma}
	\label{log-comp-11}
	For any constant $C$, there exists a constant $C'>0$ such that for every $z$ with $|\Im(z)|\le C/n$, 
	\begin{equation}
	\frac{|c_j||\cos(jz)|}{\sqrt{S}}\le C'n^{-1/2}, \qquad \text{for all } 0\le j\le n,\label{log-compa1-P}
	\end{equation}
	and 
	\begin{equation}
	\frac{|d_j||\sin(jz)|}{\sqrt{S}}\le C'n^{-1/2}, \qquad \text{for all }  0\le j\le n, \label{log-compa2-P}
	\end{equation} where $S:= \sum_{j=0}^n  |c_j|^2 |\cos  jz| ^2 + \sum_{j=1}^n  |d_j|^2 |\sin  jz| ^2$.
\end{lemma}

\begin{proof}
	
	Write $z =: a + ib$. Without loss of generality, assume that $b\ge 0$. By \eqref{trigbound}, $|\cos(jz)|\le C$ and $|\sin(jz)|\le C$ for all $0\le j\le n$, so it suffices to show  $S:= \Omega(n)$. To achieve this bound on $S$, it suffices to show that $\mathcal I_{0}$ contains a subset $ J$ of size $\Theta (n)$ 
	such that 
	\begin{equation}
	|\cos(jz)|\ge c^{\ast} \quad\mbox{for all $j\in J$, for some positive  constant $c^{\ast}$}.\label{large_c}
	\end{equation}

	Since $b\ge 0$ and $j\ge 0$, we have $$2|\cos(jz)| = e^{jb}|e^{-2jb+2ija}+1| \ge |w^{j} + 1|$$ where $w := e^{-2b + 2ia}$. By Condition {\bf C3}  and Lemma \ref{J}, we can find  a subset  $J$ of $\mathcal I_{0}$ of size $\Theta (n)$  such that $$\min_{k\in \Z} \{|2aj - (2k+1)\pi| \}\ge c$$ for some constant $c>0$ and all $j\in J$. We can assume, without loss of generality, that  $c \le 1/10$ and 
	this guarantees  $|\cos (2aj) +1 | \ge c^2/4$.

	Consider $j \in J$,  if $1 - e^{-2jb}\ge c^{2}/10$ then by the triangle inequality, 
	$$|w^{j} + 1| = |e^{-2jb}e^{2iaj}+1|\ge 1 -|e^{-2jb}e^{2iaj}| = |e^{-2jb}-1|\ge c^{2}/10.$$ 
	In the opposite case, $e^{-2jb}\ge 1-c^2/10 > .99$. 
	Keeping in mind that $c \le 1/10$, we have 	
	\begin{equation}
	|w^{j} + 1| \ge e^{-2jb}|e^{2iaj}+1| - |e^{-2jb}-1|\ge .99 c^2/4 -c^{2}/10 \ge c^{2}/10.
	\end{equation} Thus, we achieved  \eqref{large_c} with $c^{\ast} = c^2/10$. 
\end{proof} 

Finally, using Conditions {\bf C3, C4} and \eqref{large_c}, it is a routine to prove that the repulsion Condition {\bf C2} \eqref{cond-repulsion} holds. That completes the proof.

\section{Proof of Theorem \ref{comparison} and Corollary \ref{maincor}}\label{proof-comparison}
As before, by rescaling the coefficients, we can assume that $\max_{0\le j\le n} \{ |c_j|, |d_j|\} = 1$.
Before going to the proofs, let us state a version of the Kac-Rice formula for Gaussian processes. Note that a Gaussian process $P(t), t\in (a_0, b_0)$ is a random variable $P:\Omega\times (a_0, b_0) \to \R$ with $\Omega$ being a probability space such that for each $\omega\in \Omega$, $P(\omega, \cdot)$ is a continuous function on $(a_0, b_0)$ and for each $k\in \N$, $t_1, \dots, t_k\in (a_0, b_0)$, $(P(\cdot, t_1), \dots, P(\cdot, t_k))$ is a Gaussian random vector.
\begin{proposition}\cite[Theorem 2.5]{Far} \label{KacRice}
Let $P(t), t\in (a_0, b_0)$ be a real, differentiable Gaussian process. Let $\mathcal P (t)= \Var (P(t))$, $\mathcal Q(t) = \Var (P'(t))$, $\mathcal R(t) = \Cov (P(t), P'(t))$, $\rho(t)=\frac{\mathcal R(t)}{\sqrt{\mathcal P(t)\mathcal Q(t)}}$, $m(t) = \E P(t)$, and $\eta(t) = \frac{m'(t)-\rho(t)m(t)\sqrt{\CQ(t)/\CP(t)}}{\sqrt{\CQ(t)(1-\rho^{2}(t))}}$. Assume that $m'(t)$ is continuous and the joint normal distribution for $P(t)$ and $P'(t)$ has non-singular covariance matrix for each $t$, then for any interval $[a, b]\subset (a_0, b_0)$, we have
$$\E N_{P}(a, b) = \int_{a}^{b}\sqrt{\frac{\CQ(t)(1-\rho^{2}(t))}{\CP}}\phi\left (\frac{m(t)}{\sqrt{\CP(t)}}\right )\big (2\phi(\eta(t))+\eta(t)\left (2\Phi(\eta(t))-1\right )\big )dt$$
where $\phi(t)$ and $\Phi(t)$ are the standard normal density and distribution functions, respectively.
\end{proposition}

\begin{proof}[Proof of Theorem \ref{comparison}]
By triangle inequality, we can assume that $\tilde \xi_j$ and $\tilde \eta_j$ are Gaussian random variables. Let $c$ be the constant in Theorem \ref{real} with $\alpha_1 = 1/2, k=1, l=0$. As in Remark \ref{rmkconstants}, we can set  $c=\frac{\ep}{2\cdot 10^{9}}$. Let $\alpha = c/7$. It suffices to show that for every interval $(a_n, b_n)$ of size at most $1/n$, we have
\begin{equation}
\left |\E N_{P_n} (a_n, b_n) - \E N_{\tilde P_n} (a_n, b_n) \right | = O(n^{-\alpha/2}).\label{dd1}
\end{equation}

If $b_n-a_n\ge 1/n$, we simply divide the interval $(a,b)$ into $\lfloor (b_n-a_n)n \rfloor + 1$ intervals of size at most $1/n$ each and then apply \eqref{dd1} to each interval and then sum up the bounds. 

Let $\ell := (b_n-a_n)/2$. Let $G$ be a smooth function on $\R$ with support in \newline $\left [\frac{a_n+b_n}{2}-\ell-n^{-1-\alpha},\frac{a_n+b_n}{2}+ \ell+n^{-1-\alpha}\right] $ such that $0\le G\le 1$, $G = 1$ on $\left [\frac{a_n+b_n}{2}-\ell, \frac{a_n+b_n}{2}+\ell\right ]$, and $\norm{G^{(a)}} _\infty\le Cn^{6\alpha+a}$ for all $0\le a \le 6$. 

By the definition of $G$, we have  
\begin{equation}
\E N_{{P_n}}{(a_n, b_n)}\le \E \sum G(\zeta_i) \le \E N_{{P_n}}{(a_n-n^{-1-\alpha}, b_n+n^{-1-\alpha})}\nonumber
\end{equation}
where $\zeta_i$ are the real roots of $P_n$. 
Similarly, 
\begin{equation}
\E N_{{\tilde P_n}}{(a_n, b_n)}\le \E \sum G(\tilde \zeta_i) \le \E N_{{\tilde P_n}}{(a_n-n^{-1-\alpha}, b_n+n^{-1-\alpha})}\nonumber.
\end{equation}
Applying Theorem \ref{real} (with $k = 1, l=0$) to the function $G/n^{6\alpha}$, we get
\begin{eqnarray}
\E \sum G(\zeta_i)&=&  \E \sum G(\tilde \zeta_i)+  O\left (n^{-c+6\alpha}\right )= \E \sum G(\tilde \zeta_i)+  O\left (n^{-\alpha}\right ) \nonumber.
\end{eqnarray}

Since $\alpha =c/7$, we obtain 
\begin{eqnarray}
\E N_{{P_n}}{(a_n, b_n)}&\le& \E N_{{\tilde P_n}}{(a_n-n^{-1-\alpha}, b_n+n^{-1-\alpha})}+ O(n^{-\alpha} ) \le \E N_{\tilde P_n}{(a_n, b_n)}+ 2\mathcal I_{\tilde{P}_n} + O(n^{-\alpha} ) \nonumber,
\end{eqnarray}
where 
$\mathcal I_{\tilde{P}_n} := \sup_{x\in \R}\E N_{\tilde P_n} (x - n^{-1-\alpha}, x)$. We will show later that $\mathcal I_{\tilde P_n} = O(n^{-\alpha/2})$, which gives the upper bound $\E N_{{P_n}}{(a_n, b_n)}\le \E N_{\tilde P_n}{(a_n, b_n)}+ O (n^{-\alpha/2}$).

Let us quickly address the lower bound  $\E N_{{P_n}}{(a_n, b_n)}\ge \E N_{\tilde P_n}{(a_n, b_n)}-O(n^{-\alpha/2})$.
If $\ell > n^{-1-\alpha}$, we can argue as for the upper bound. In the case 
$\ell \le n^{-1-\alpha}$,  the desired bound follows from the observation that  $\E N_{{P_n}}{(a_n, b_n)}\ge 0\ge \mathcal I_{\tilde P_n} - O(n^{-\alpha/2})\ge \E N_{\tilde P_n}{(a_n, b_n)}- O(n^{-\alpha/2})$. The upper and lower bounds together give \eqref{dd1}.

To prove the stated bound on $\mathcal I_{\tilde{P}_n}$, we use Proposition \ref{KacRice}, which asserts that for  every $x\in \R$, 
\begin{eqnarray}\label{nh1}
\E N_{\tilde P_n}[x - n^{-\alpha-1},x] &\le& \int_{x - n^{-\alpha-1}}^{x}\sqrt{\frac{\mathcal S}{\mathcal P^{2}}}dt +\int_{x - n^{-\alpha-1}}^{x}\frac{|m'|\mathcal P + |m|\mathcal R}{\mathcal P^{3/2}} dt,
\end{eqnarray}
where

\begin{itemize}
	\item  $m(t) := \E\tilde{P}_n(t)$
	\item $ \mathcal P(t) : =\Var (\tilde P_n)=\sum_{k=0}^{n} c_k^{2}\cos^{2}(kt) + d_k^{2}\sin^{2}(kt)$
	\item $\mathcal Q(t) :=\Var(\tilde P_n')=\sum_{k=0}^{n} k^{2}c_k^{2}\sin^{2}(kt) + k^{2}d_k^{2}\cos^{2}(kt)$ 
	 \item $\mathcal R(t)  := \textbf{Cov}(\tilde P_n, \tilde P_n')=\sum_{k=0}^{n} k\cos(kt)\sin(kt) (-c_k^{2}+ d_k^{2})$
	 \item  $\mathcal S(t) = \mathcal P(t) \mathcal Q (t) - \mathcal R^{2}(t) $. 
	 \end{itemize}  

Observe that the covariance matrix of $(P_n(t), P'_n(t))$ is non-singular if and only if $\CS(t)\neq 0$. Since the deterministic function $\CS(t)$ only has finitely many zeroes in $ [x - n^{-\alpha-1},x] + (-\ep^*, \ep^*)$ (where we add $(-\ep^*, \ep^*)$ only to make the interval bigger to apply Proposition \ref{KacRice}, $\ep^*$ can be any positive number), we can decompose this  interval  into subintervals whose interiors do not contain any zero of $\CS$, and use linearity of expectation if necessary. This way, we can assume that  the joint distribution of $\tilde P_n$ and $\tilde P_n'$ is non-singular, as required  in Proposition \ref{KacRice}.

From \eqref{trigbound} and \eqref{large_c}, there is a constant $K >0$ such that  for every $t\in \R$,
\begin{equation}
\mathcal P\ge \frac{n}{K}, \quad \mathcal Q\le K n^{3},\mbox{ and }
\mathcal R \le Kn^{2}\le K n\mathcal P.\nonumber
\end{equation}

From here, we obtain (for all $t$) that  $\frac{\mathcal S}{\mathcal P^{2}}\le \frac{\mathcal Q}{\mathcal P}\le Kn^{2}$.

Moreover, from Condition {\bf C4}, we have $|m(t)|\le K n^{\tau_0}$ and $|m'(t)|\le K n^{1/2 + \tau_0}$ (notice that 
$m(t)=0$ if all atom random variables have zero mean; the upper bounds here come from 
the bound on the expections). It follows that 
\begin{eqnarray}
\frac{|m'|\mathcal P + |m|\mathcal R}{\mathcal P^{3/2}} \le K n^{1/2+\tau_0}.\nonumber
\end{eqnarray}

Using the above estimates, we conclude that  the integrand on the right-hand side of \eqref{nh1} is bounded (in absolute value) by $O(n^{1/2+\tau_0})$.  Since the length of the interval in the  integration is $n^{-\alpha-1}$, the integral is of order $O(n^{\tau_0-\alpha-1/2}) = O(n^{-\alpha/2})$, as  $\tau_0-1/2 = \frac{\ep}{10^{11}}\le \alpha/2$.
\end{proof}

\begin{proof}[Proof of Corollary \ref{maincor}] As promised in Remark \ref{rmk1}, we will prove the desired statement for $P_n$ as in \eqref{newP}. Applying Theorem \ref{comparison} with 
$$\tilde P_n(x) :=u_n\sqrt{\sum_{i=0}^{n}c_i ^{2}} + \sum _{j=0}^{N_0} u_j n^{1/2-\alpha} \cos(jx) + \sum_{j=1}^{N_0} v_j n^{1/2-\alpha} \sin(jx) + \sum_{j=0}^{n}c_j \tilde \xi_j\cos(jx) + \sum_{j=1}^{n} c_j\tilde \eta_j\sin(jx)$$
where $\tilde \xi_j$ and $\tilde \eta_j$ are iid standard Gaussian, it suffices to prove that the desired estimate holds for $\tilde P_n$. Applying Proposition \ref{KacRice} to $\tilde P_n$, we obtain
$$\E N_{\tilde P_n}(a_n, b_n) = \int_{a_n}^{b_n}\sqrt{\frac{\sum_{i=0}^{n}c_i ^{2}i^{2}}{\sum_{i=0}^{n}c_i ^{2}}} \phi\left (\frac{m(x)}{\sqrt{\sum_{i=0}^{n}c_i ^{2}}}\right )\big [2\phi(q(x)) + q(x)\left (2\Phi(q(x))-1\right )\big ]dx$$
where $$m(x) := u_n\sqrt{\sum_{i=0}^{n}c_i ^{2}} + \sum _{j=0}^{N_0} u_j n^{1/2-\alpha} \cos(jx) + \sum_{j=1}^{N_0} v_j n^{1/2-\alpha} \sin(jx)$$ and $q(x) := \frac{m'(x)}{\sqrt{\sum_{i=0}^{n}c_i ^{2}i^{2}}}$.

In our setting,  ${\sum_{i=0}^{n}c_i ^{2}} = \Theta(n)$, ${\sum_{i=0}^{n}c_i ^{2}i^{2}} = \Theta(n^{3})$, and so $\frac{m(x)}{\sqrt{\sum_{i=0}^{n}c_i ^{2}}} = u_n+O(n^{-\alpha})$ and $q(x) = O(n^{-1})$.
Therefore, by the boundedness of the functions $\Phi$, $\phi$ and $\phi'$, we get
$$\phi\left (\frac{m(x)}{\sqrt{\sum_{i=0}^{n}c_i ^{2}}}\right ) = \phi(u_n) + O(n^{-\alpha}), \mbox{ and } 2\phi(q(x))+q(x)\left (2\Phi(q(x))-1\right )=2\phi(0)+O(n^{-1}).$$
 
It follows that 
\begin{eqnarray}
\E N_{\tilde P_n}(a_n, b_n) &=& 2\sqrt{\frac{\sum_{i=0}^{n}c_i ^{2}i^{2}}{\sum_{i=0}^{n}c_i ^{2}}} (b_n-a_n) \phi\left (u_n\right )\phi(0) + O\left (n^{-\alpha}\sqrt{\frac{\sum_{i=0}^{n}c_i ^{2}i^{2}}{\sum_{i=0}^{n}c_i ^{2}}} (b_n-a_n)\right )\nonumber\\
 &=& 2\sqrt{\frac{\sum_{i=0}^{n}c_i ^{2}i^{2}}{\sum_{i=0}^{n}c_i ^{2}}} (b_n-a_n) \phi\left (u_n\right )\phi(0) + O\left (n^{-\alpha} (b_n-a_n)n\right )\nonumber.
\end{eqnarray} 
Plugging in $\phi(x)=\frac{1}{\sqrt{2\pi}} e^{-x^{2}/2}$, we obtain 
$$\E N_{P_n}(a_n, b_n) = \frac{b_n-a_n}{\pi} \sqrt{\frac{\sum_{j=0}^{n} c_j^{2}j^{2}}{\sum_{j=0}^{n} c_j^{2}}}\exp\left (-\frac{u_n^{2}}{2}\right ) + O\left (n^{-c}((b_n-a_n)n + 1) \right ) $$
where the positive constant $c$ and the implicit constant depend only on $\alpha, N_0, K, \tau_1$, $\ep$, 
completing the proof.
\end{proof}

\section{Proof of Theorem \ref{kacreal} and Corollary \ref{kacmean}}\label{kacproof}

\begin{proof}[Proofs of Theorem \ref{kacreal}]
Let us first consider the case $0<\theta_n<\frac{1}{K}$ for some sufficiently large constant $K>0$. Let $\delta_n =\theta_n+1/n$.

We apply Theorem \ref{greal} to the random function 
$F_n(z) := P_n(z\theta_n/10)$
and the domain $D_n := \{z: 1-2\theta_n\le |z\theta_n/10|\le 1-\theta_n+1/n\}$.

For this model, one can choose $\alpha_1=1/2$ and $C_1=1$. The main task is to show that for any positive constants $A, c_1$, there exists a constant $C$ for which Conditions {\bf C2} \eqref{cond-poly}-{\bf C2} \eqref{cond-repulsion} hold with parameters $(k+l, C_1, \alpha_1, A, c_1, C)$. Conditions {\bf C2} \eqref{cond-delocal} and {\bf C2} \eqref{cond-repulsion} can be checked by a  simple algebraic manipulation, which we leave as an exercise. 
To verify  Condition {\bf C2} \eqref{cond-bddn}, notice that  for any $M>2$,  if we condition on the event  $\Omega'$ on which $|\xi_i|\le M\left (1+\delta_n/2\right )^{i}$ for all $i$,  then  for all $z\in D_n + B(0, 2)$, 
\begin{equation}
|F_n(z)| = O(M)\sum_{i=0}^{n} \left (1+\delta_n/2\right )^{i}(1-\delta_n+2/n)^{i} = O(M\delta_n^{-1}).\label{interm1}
\end{equation}

Thus, for every $M>2$, we have
\begin{equation}
\P\left (|F_n(z)| = O(M\delta_n^{-1})\right )= 1-O\left (\sum _{i=0}^{n} \frac{1}{M\left (1+\delta_n/2\right )^{i}}\right )= 1- O\left (\frac{1}{M\delta_n}\right ).\label{bound_kac}
\end{equation}

Setting $M = \delta_n^{-A-1}$, we obtain Condition {\bf C2} \eqref{cond-bddn}.

To prove Condition {\bf C2} \eqref{cond-smallball}, we show that for any constants $A$ and $c_1>0$, there exists a constant $B>0$ such that the following holds. For every $z_0$ with $1-2\theta_n\le |z_0|\le 1-\theta_n+1/n$, there exists $z= z_0e^{i\theta}$ where $\theta\in [-\delta_n/100, \delta_n/100]$ such that for every $1\le M\le n\delta_n$, 
\begin{equation}
\P\left (|P_n(z)|\le e^{-\delta_n^{-c_1}}e^{-BM}\right )\le \frac{B\delta_n^{A}}{M^{A}}.\label{smallball_kac}
\end{equation}

Setting $M = 1$, we obtain Condition {\bf C2} \eqref{cond-smallball}.

By writing $z_ 0 = re ^{i\theta_0}$, the bound \eqref{smallball_kac} follows from a more general anti-concentration bound: there exists $\theta\in I := [\theta_0 - \delta_n/100, \theta_0 + \delta_n/100]$ such that 
\begin{equation}
\sup _{Z\in \C}\P\left (|P_n(re^{i\theta})-Z|\le e^{-\delta_n^{-c_1}}e^{-BM}\right )\le \frac{B\delta_n^{A}}{M^{A}}.\nonumber
\end{equation}
 
Since the probability of being confined in a complex ball is bounded from above by the probability of its real part being confined in the corresponding interval on the real line, it suffices to show that
 \begin{equation}
\sup _{Z\in \R}\P\left (\left |\sum_{j=0}^{M\delta_n^{-1}/2} \xi_j r^{j} \cos{j\theta}-Z\right |\le e^{-\delta_n^{-c_1}}e^{-BM}\right )\le \frac{B\delta_n^{A}}{M^{A}}.\nonumber
\end{equation}
 
This is, in turn, a direct application of Lemma \ref{lmanti_concentration} with $N := M\delta_n^{-1}/2$ and $\bar e := e^{-2M}\le r^{j}$ for all $0\le j\le M\delta_n^{-1}/2$.

Finally, to prove Condition {\bf C2} \eqref{cond-poly}, from \eqref{bound_kac}, \eqref{smallball_kac}, and Jensen's inequality, we get for every $1\le M\le n\delta_n$
$$\P(N\ge \delta_n^{-c_1} + BM) = O\left (\frac{\delta_n^{A}}{M^{A}}\right )$$
where $N = N_{F_n}B(w, 2)$, $w\in D_n$. 

Let $A = k+l+2$, $c_1=1$ and $M = 1, 2, 2^{2}, \dots, 2^{m}$ where $m$ is the largest number such that $2^{m}\le n\delta_n$. Combining the above inequality with the fact that $N\le n$ a.s., we get
$$\E N^{k+l+2}\textbf{1}_{N\ge \delta_n^{-1}} \le C\sum_{i=1}^{m} \left (\delta_n^{-1} + B2^{i+1}\right )^{k+l+2}\frac{\delta_n^{A}}{2^{iA}} + C n^{k+l+2} \frac{\delta_n^{A}}{2^{mA}} \le C\delta_n^{A-k-l-2} = O(1).$$
 This proves Condition {\bf C2} \eqref{cond-poly} and completes the proof for $\theta_n\le 1/K$. For $\theta_n\ge 1/K$, note that Jensen's inequality implies that
$$N_{P_n}B(0, 1-1/K) = O_K(1)\log\frac{\max_{w\in B(0, 1-1/2K) }|P_n(w)|}{\max_{w\in B(1-1/K, 1/3K)} |P_n(w)|}.$$

Thus, using the bounds \eqref{interm1}, \eqref{bound_kac}, \eqref{smallball_kac} for $\theta_n = 1-1/K$, we get for every $1\le M\le n/K$,
$$\P(N_{P_n}B(0, 1-1/2K) \ge BM) = O\left (\frac{1}{M^{A}}\right ).$$
And so, $\E N_{P_n}B(0, 1-1/2K) = O(1)$. The same holds for $\tilde P_n$ and therefore the desired result follows.
\end{proof}

\begin{proof}[Proof of Corollary \ref{kacmean}]
Without loss of generality, we can assume that $\tilde \xi_0, \dots, \tilde \xi_n$ are standard Gaussian random variables. As in Remark \ref{Qrmk}, it suffices to restrict to the roots in the interval $[-1, 1]$. Divide this interval into $I_0 = \{x: |x|\le 1-1/C\}$ and $I_1 = [-1, 1]\setminus I_0$ and denote by $N(0)$ and $N(1)$ the number of real roots of $P_n$ in these sets, respectively. We have seen in the proof of Theorem \ref{kacreal} that $\E N(0) = O(1)$, and so is $\tilde N(0)$ which is the corresponding term for $\tilde P_n$.

To get $\E N(1) - \E \tilde N(1) = O(1)$, we decompose the interval $I_1$ into dyadic intervals $\pm [1-1/C, 1-1/2C), \pm[1-1/2C, 1-1/4C), \dots, \pm [1-2/n, 1-1/n)$, and finally $\pm [1-1/n, 1]$. In each of these intervals, say $[x, y)$, we show that $\E N_{P_n}[x, y) - \E N_{\tilde P_n} [x, y) = O((1-y+1/n)^{c})$ for some positive constant $c$. This can be routinely done by approximating the indicator function on the interval $[x, y)$ by a smooth function and applying Theorem \ref{kacreal}. We omit the details as it is similar to the proof of Theorem \ref{comparison}.

\end{proof}

\section{Proof of Theorems \ref{uni_flat} and Corollary \ref{mean_flat}}\label{proof_flat}
\begin{proof}[Proof of Theorem \ref{uni_flat}]

Notice that by  the Borel-Cantelli lemma, with probability 1, there are only a finite number of $i$ such that $|\xi_i|\ge 2^{i}$. Thus with probability 1, the radius of convergence of the series $P$ is infinity and so $P$ is an entire function.

A natural idea is to apply Theorem \ref{gcomplex} with $n=\infty$ to the function $F_{n}(z) := P(z)$, with $\delta_n := |z_0|^{-1}$ and $D_n := \{z_0\}$. (We will skip the redundant subscript $n$ in the rest of the proof.)
However, since $\Var P(z) = e^{|z|^{2}}$,  $|P(z)|$ is likely to be of order $\Theta (e^{|z|^{2}/2})$ in which case  Condition {\bf C2} \eqref{cond-bddn} fails.  The idea here is to find  a proper scaling, which, at the same time, preserves the   analyticity of $F$. We set 
\begin{equation}
F(z) := \frac{P(z)}{e^{|z_0|^{2}/2} e^{(z-z_0)\bar{z_0}}}.\label{res11}
\end{equation}

A routine calculation shows that $\Var F(z) = \Theta(1)$.

 Furthermore,  $F$ is analytic and has 
 the same roots as $P$. For this model, let $\alpha_1=1/2$ and $C_1=2$. The main task is to show that for any positive constants $A, c_1$, there exists a constant $C$ for which Conditions {\bf C2} \eqref{cond-poly}-{\bf C2} \eqref{cond-delocal} hold with parameters $(k, C_1, \alpha_1, A, c_1, C)$. We can, without loss of generality,  assume that $|z_0|$ is sufficiently large because by  Jensen's inequality, one can show that  the expected number of roots of both $P$ and $\tilde P$ in $B(0,K)$, for any constant $K$, is $O_K(1)$.   
 
 Condition {\bf C2} \eqref{cond-bddn} is a direct consequence of the following lemma.
\begin{lemma} For any constant $A>0$, there is a constant $K>0$ such that for any $M \ge 2$, 
\begin{equation}
\P\left (\max _{z\in B(z_0, 2)}|F(z)| \ge K M^{A}\delta^{-A-2}\right )\le \frac{K \delta^{A}}{M^{A}}. \label{bdd1}
\end{equation}

\end{lemma}

\begin{proof}
Let $L = |z_0|+1=\Theta(\delta^{-1})$. Let $\Omega'$ be the event that  $|\xi_i|\le M^{A}L^{A}\left (1+\frac{1}{(L+2M)^{2}}\right )^{i}$ for all $i\ge 0$. Consider its complement $\Omega'^{c}$, 
\begin{equation} \label{probbound} \P\left (\Omega'^{c}\right )= O \left ( \sum_{i=0}^{\infty} \frac{1}{M^{2A}L^{2A}\left (1+(L+2M)^{-2}\right )^{2i}}\right ) = O\left (\frac{\delta^{A}}{M^{A}}\right ).
\end{equation}

On the other hand, once $\Omega'$ holds, then  for every $z\in B(z_0, 2)$, 
\begin{equation}
|P(z)|\le \sum_{i =0}^{\infty}\frac{|\xi_i||z|^{i}}{\sqrt{i!}}\le M^{A}L^{A}\sum_{i =0}^{\infty}\frac{(|z|+|z|^{-1})^{i}}{\sqrt{i!}} = M^{A}L^{A}S(w).\nonumber
\end{equation}
where $w = |z|+|z|^{-1}$ and $S(w) := \sum_{i =0}^{\infty}\frac{w^{i}}{\sqrt{i!}}$. Let $x := x(w) = \lfloor w^{2}-1\rfloor$. We split into the sum of  $S_1 := \sum_{i =0}^{5x-1}\frac{w^{i}}{\sqrt{i!}}$ and $S_2 := \sum_{i=5x}^{\infty}\frac{w^{i}}{\sqrt{i!}}$. Since the terms $\frac{w^{i}}{\sqrt{i!}}$ are increasing with $i$ running  from 0 to $x$ and then decreasing with $i$ running from $x$ to $\infty$, we have $S_1\le 5x\frac{w^{x}}{\sqrt{x!}}$. Moreover, 
\begin{equation}
|S_2|\le \frac{w^{5x}}{\sqrt{(5x)!}}\sum_{i =0}^{\infty}\frac{w^{i}\sqrt{(5x)!}}{\sqrt{(i+5x)!}}\le \frac{w^{5x}}{\sqrt{(5x)!}}S.\nonumber
\end{equation}

By Stirling's formula (and the fact that $x$ is sufficiently large)
\begin{equation}
\frac{w^{5x}}{\sqrt{(5x)!}}\le \sqrt{\frac{(x+2)^{5x}e^{5x}}{(5x)^{5x+1/2}}}\le \frac{1}{2}.\nonumber
\end{equation}

Hence, $S_2\le \frac{1}{2}S$, which implies 
 $$S\le 2S_1\le 10x\frac{w^{x}}{\sqrt{x!}}\le 100w^{2}e^{w^{2}/2}=O(L^{2} e^{|z|^{2}/2}).$$ 
  Thus, on $\Omega'$,
$$|P(z)|=O(M^{A}L^{A+2}e^{|z|^{2}/2}).$$

By the definition of $F$, 
$$|F(z)|= O \left (\frac{M^{A}L^{A+2}e^{|z|^{2}/2}}{e^{|z_0|^{2}/2}e^{\Re((z-z_0)\bar{z_0})}}\right ) = O  (M^{A} L^{A+2} )$$ which, together with \eqref{probbound}, yield the desired claim. 
 \end{proof}

Write $z_0 = re^{i\theta_0}$. To verify Condition {\bf C2} \eqref{cond-smallball}, the idea is to apply Lemma \ref{lmanti_concentration} to the entire function 
$$P(z_0e^{i\theta}) = \sum_{j=0}^{\infty} \frac{r^{j}}{\sqrt{j!}} \xi_je^{ij(\theta+\theta_0)}.$$
  Note that when $|\theta|\le  .01 r^{-1}$, $z_0 e^{i\theta}\in B(z_0, 1/100)$. 
Let $x_0=\lfloor|z_0|^{2}-1\rfloor$. For any $M\ge r$, we apply Lemma \ref{lmanti_concentration} to the set $\CE = \{x_0, x_0+1, \dots, x_0+M\}$, the random variables $(\xi_j)_{j\in \CE}$, the coefficients $e_j = \frac{r^{j}}{\sqrt{j!}} $ and obtain that for any positive constant $A\ge 3$, for the interval $I=[-M^{-A}, M^{-A}]\subset [-.01 r^{-1}, .01 r^{-1}]$, there exists $\theta\in I$ such that
$$\sup_{Z\in \C}\P\left (\left |\sum_{j\in \CE} e_j \xi_j\cos(j\theta+j\theta_0)-Z\right |\le e_{x_0+M}M^{-16A^{2}}\right )=O\left (M^{-A/2}\right )$$
where we use the fact that $e_{x_0}\ge e_{x_0+1}\ge\dots \ge e_{x_0+M}$.

This together with the assumption that $\Re(\xi_0), \Im(\xi_0), \Re(\xi_1), \Im(\xi_1), \dots $ are independent imply that
$$\sup_{Z\in \C}\P\left (\left |\sum_{j\in \CE} e_j \xi_j\exp(ij(\theta+\theta_0))-Z\right |\le e_{x_0+M}M^{-16A^{2}}\right )=O\left (M^{-A/2}\right )$$
because the distance between two complex numbers is at least the distance between their real components.

Conditioning on the random variables outside $\CE$, we obtain some $\theta\in I$ such that with probability at least $1-O\left (M^{-A/2}\right )$,
$$|P(z_0e^{i\theta})|\ge e_{x_0+M}M^{-16A^{2}},$$
which implies
$$|F(z_0e^{i\theta})|\ge \frac{e_{x_0+M}M^{-16A^{2}}}{\exp(r^{2}/2)\left |\exp(r^{2}(e^{i\theta}-1))\right |}=\frac{r^{x_0+M}M^{-16A^{2}}}{\sqrt{(x_0+M)!}\exp(r^{2}/2)\left |\exp(r^{2}(e^{i\theta}-1)) \right |}.$$

For $\theta\in I$, $|r^{2}(e^{i\theta}-1)|=O(r^{2}M^{-A}) = O(1)$. 
Thus, by Stirling's formula,
$$|F(z_0e^{i\theta})|=\Omega\left ( \frac{1}{r}\frac{r^{M}M^{-16A^{2}}}{\sqrt{(x_0+1)\dots(x_0+M)}}\right )=\Omega\left ( \frac{M^{-16A^{2}}}{r}\left (\frac{r^{2}}{r^2+M}\right )^{M/2}\right )
.$$

In other words, we have proved that for every constant $A\ge 3$, for every $M\ge r=|z_0|$, there exists $z\in B(z_0, 1/100)$ for which
\begin{equation}
\P\left (|F(z)|=O_A\left ( \frac{M^{-16A^{2}}}{r}\left (\frac{r^{2}}{r^2+M}\right )^{M/2}\right ) \right )=O_A\left (M^{-A/2}\right ).\label{smb1}
\end{equation}
Setting $M=\lceil r\rceil$, we obtain Condition {\bf C2} \eqref{cond-smallball} (note that $r  =\delta^{-1}$).

Combining \eqref{bdd1} and \eqref{smb1} and Jensen's inequality, we get that there exists a constant $K$ depending only on $A$ such that for any $M\ge r$,
$$\P\left (N_{F}(B(z_0, 1)) \ge M^{2}\right )\le  \frac{K}{M^{A}}.$$

Thus, 
$$\E N^{k+2}_{F}(B(z_0, 1))\textbf{1}_{N_{F}(B(z_0, 1))\ge r^{2}}\le \sum_{M=r}^{\infty} \E N^{k+2}_{F}(B(z_0, 1))\textbf{1}_{M^{2}\le N_{F}(B(z_0, 1))\le (M+1)^{2}}.$$
As the right-hand side is at most $O(1) \sum _{M=r}^{\infty} \frac{K(M+1)^{2k+4}}{M^{A}}=O(1)$ by setting $A=2k+6$, Condition {\bf C2} \eqref{cond-poly} follows.

Finally for Condition {\bf C2} \eqref{cond-delocal}, note that $|z|^i/\sqrt{i!}$ is maximized  at $i = \lfloor|z|^{2}-1\rfloor$. By Stirling's formula, at this $i$, $|z|^i/\sqrt{i!}=O\left ( \frac{\sqrt{\sum_{j} |z|^{2j}/j!}}{|z|^{1/2}}\right )$.
\end{proof}

\begin{proof}[Proof of Corollary \ref{mean_flat}]
As before, we simply approximate the indicator function $\textbf{1}_{B}$ above and below by smooth test functions $f$ and $g$ whose derivatives up to order $6$ are bounded by $O(r^{6a})$ for a sufficiently small constant $a$ and $\int_{\C} (f-g)dm = O( r^{-a})$. Applying Theorem \ref{uni_flat} to the function $f$, we obtain
$$\E N_{P}(B)\le \E \sum_{\zeta: P(\zeta) = 0} f(\zeta) = \E \sum_{\tilde \zeta: \tilde P(\tilde\zeta) = 0} f(\tilde \zeta) + O(r^{-c+6a}) = \E N_{\tilde P}(B) + O(r^{-a} + r^{-c+6a})$$
where $c$ is the constant in Theorem \ref{uni_flat}. By choosing $a = c/12$, we get $\E N_{P}(B) = \E N_{\tilde P}(B) + O(r^{c/12})$. And similarly, applying Theorem \ref{uni_flat} to the function $g$, we get the corresponding lower bound. This completes the proof.
\end{proof}

\section{Proof of Theorem \ref{uni_elliptic} and Corollary \ref{mean_elliptic}}\label{proof_elliptic}

\begin{proof}[Proof of Theorem \ref{uni_elliptic}]
We have $\Var P_n(z) = (|z|^{2}+1)^{n}$. As in the proof of Theorem \ref{uni_flat}, we will apply the framework in Section \ref{framework} to the function
$$F_{n}(z) = \frac{P_n(z/\sqrt n)}{(|x_0|^{2}+1)^{n/2}\exp\left (\frac{n(z/\sqrt{n}-x_0)\bar{x_0}}{(|x_0|^{2}+1)}\right) },$$
$\delta_n = n^{-1}$ and $D_n = \{\sqrt n x_0\}$. We have $\Var F(z) = \Theta(1)$. Note that the denominator is chosen so that $\Var F(z) = \Theta(1)$, $F$ is analytic, and $F(z) = 0$ if and only if $P(z/\sqrt{n})=0$. We will first show that Theorem \ref{gcomplex} holds, and then we show that the conclusion of Theorem \ref{greal} also holds. For Theorem \ref{gcomplex}, it suffices to show that there exist positive constants $C_1, \alpha_1$ such that for any positive constants $A, c_1$, there exists a constant $C$ for which Conditions {\bf C2} \eqref{cond-poly}-{\bf C2} \eqref{cond-delocal} hold with parameters $C_1, \alpha_1, A, c_1, C$. For this model, one can choose $\alpha_1=\ep/4$ and $C_1=1$. Condition {\bf C2} \eqref{cond-bddn} follows from the following. For any constants $A, c_1>0$, we have
\begin{equation}
\P\left (\max _{z\in B(\sqrt n x_0, 2)}|F(z)| \ge C e^{n^{c_1}}\sqrt{n}\right )\le \frac{Cn}{e^{n^{c_1}}} \label{bdd2}
\end{equation}
for some constant $C$ depending only on $A$ and $c_1$.

Indeed, let $\Omega'$ be the event on which $|\xi_i|\le e^{n^{c_1}}$ for all $i\ge 0$. The probability of its complement is bounded from above by
$$\P\left (\Omega'^{c}\right )\le \frac{Cn}{e^{n^{c_1}}}.$$ 

On $\Omega'$, for every $z\in B(x_0, 2/\sqrt n )$, we have
\begin{equation}
|P(z)|\le \sum_{i =0}^{n} \sqrt{n\choose i}|\xi_i||z|^{i} \le e^{n^{c_1}}\sqrt{n}\sqrt{\sum_{i =0}^{n}{n\choose i} |z|^{2i}} = e^{n^{c_1}}\sqrt{n} \sqrt{\Var P(z)}.\label{ell1}
\end{equation}
Thus, 
$$|F(z)|\le C e^{n^{c_1}}\sqrt{n}.$$
 
For Condition {\bf C2} \eqref{cond-smallball}, note that the sequence $\sqrt{n\choose i} |x_0|^{i}$ increases from $i=1$ to $i_0 = \lfloor 1+\frac{(n-1)x_0^{2}}{1+x_0^{2}}\rfloor$ and then decreases. For $n^{-1/2+\ep}\le |x_0|\le 1$, we have $\frac{n^{2\ep}}{4}\le i_0\le \frac{n+1}{2}$. Condition {\bf C2} \eqref{cond-smallball} follows by showing that for any constants $A, c_1>0$, there exists a constant $C$ and an angle $\theta\in [-1/(100\sqrt{n}), 1/(100\sqrt{n})]$ such that
\begin{equation}
\P\left (|F(\sqrt{n}x_0e^{i\theta})|\le Ce^{-n^{c_1}}\right )\le Cn^{-A}.\label{smallball_elliptic}
\end{equation}
We apply Lemma \ref{lmanti_concentration} to the set $\CE = \{i_0, i_0+1, \dots, i_0+m\}$ where $m=\frac{n^{c_1/2}}{\log n}$, the random variables $(\xi_j)_{j\in \CE}$, the coefficients $e_j = \sqrt{n\choose j}r^{j}$ where $r  =|x_0|$, and the interval $I=[-m^{-A'}, m^{-A'}]$ where $A'=5A/c_1$. We have
\begin{eqnarray}
1\le \frac{e_j}{e_{j+1}}\le \frac{\sqrt{{j+1}}}{r\sqrt {n-j}}\le n^{1/2},\nonumber
\end{eqnarray} 
for all $j\in \CE$, which implies
\begin{eqnarray}
e_{i_0+m}\ge e_{i_0} n^{-m/2}.\nonumber
\end{eqnarray}

Moreover, we have since $e_{i_0}$ is the largest term, $\Var P(x_0)\le n e_{i_0}^{2}$, and so,
$$e_{i_0+m}\ge \frac{\sqrt{\Var P(x_0)}}{\sqrt{n} n^{m/2}} = \frac{\sqrt{\Var P(x_0)}}{\sqrt n e^{n^{c_1/2}}}.$$
Hence, there exists $\theta\in I$ such that for all $Z\in \C$, 
$$\P\left (\left|\sum_{j\in \CE} e_j \xi_j \cos(j\theta)-Z\right | \le \sqrt{\Var P(x_0)}e^{-n^{c_1/2}}m^{-16A'^{2}}/\sqrt n\right ) =O\left (m^{-A'/2}\right )=O\left (n^{-A}\right ).$$

By conditioning on the random variables not in $\CE$, we obtain
\begin{equation}
\P\left (|P_n(x_0e^{i\theta})|\le \sqrt{\Var P(x_0)}e^{-n^{c_1/2}}m^{-16A'^{2}}/\sqrt n\right )  =O\left (n^{-A}\right ).\nonumber
\end{equation}
Since $e^{-n^{c_1/2}}m^{-16A'^{2}}/\sqrt n=\Omega\left (e^{-n^{c_1}}\right )$, we obtain
\begin{equation}
\P\left (|P_n(x_0e^{i\theta})|\le \sqrt{\Var P(x_0)}e^{-n^{c_1}}\right )= O\left (n^{-A}\right ).\label{ell2}
\end{equation}
That implies \eqref{smallball_elliptic} and therefore, Condition {\bf C2} \eqref{cond-smallball} follows.

Combining \eqref{bdd2} and \eqref{smallball_elliptic} and Jensen's inequality, we get that
$$\P\left (N_{F}(B(\sqrt n x_0, 1)) \ge n^{c_1}\right )\le Cn^{-A}.$$
From this and the fact that $N_{F}(B(\sqrt n x_0, 1))$ is always at most $n$, Condition {\bf C2} \eqref{cond-poly} follows.

For Condition {\bf C2} \eqref{cond-delocal}, as we have seen above, $E_i:=\sqrt{n\choose i}|x_0|^i$ is largest when $i_0= \lfloor 1+\frac{(n-1)x_0^{2}}{1+x_0^{2}}\rfloor\in [\frac{n^{2\ep}}{4}, \frac{n+1}{2}]$. It suffices to show that the $E_{i_0}=O( n^{-\ep/4})\sqrt{\sum_{i} E_{i}^{2}}$ which can be deduced from showing that the consecutive terms $(E_{i})_{i=i_0-n^{\ep/2}}^{i_0+n^{\ep/2}}$ are of the same order, i.e. $E_{i}/E_{j} = \Theta(1)$. We have for $i$ in the above window, 
$$\frac{E_{i+1}^{2}}{E_{i}^{2}} = \frac{|x_0|(n-i+1)}{i+1} = \Theta\left (\frac{n-i+1}{n-i_0+1}\frac{i_0+1}{i+1}\right )=\Theta\left (1+\frac{1}{n^{\ep}}\right ).$$
Thus for all $i, j$ in the above window, 
$$\frac{E_{i}}{E_{j}} = \Theta\left (1+\frac{1}{n^{\ep}}\right )^{n^{\ep/2}} = \Theta(1)$$
as needed. So Theorem \ref{gcomplex} holds for $F_n$. It's left to show that the conclusion of Theorem \ref{greal} also holds. 

Unfortunately, Condition {\bf C2} \eqref{cond-repulsion} doesn't hold for $F_n$. Note that this condition is used in the proof of Theorem \ref{greal} only to show that \eqref{repulsion_gau} which says that for any $x\in [n^{-1/2+\ep}, 1+n^{-1/2}]$, we have for a sufficiently small constant $c$, 
\begin{equation}
\P(N_{\tilde{F_n}}{B(\sqrt n x,2 n^{-c})}\ge 2)\le C n^{-16c/10}\label{repulsion_gau2}
\end{equation}
where $\tilde F_n$ is the corresponding function with standard Gaussian coefficients.

 To prove \eqref{repulsion_gau2}, we can instead use the fact that
\begin{eqnarray}
&&\P(N_{\tilde{F}}{B(\sqrt n x,2 n^{-c})}\ge 2)\le \P(N_{\tilde{F}}{B(\sqrt n x,2 n^{-c})\cap \C_{+}}\ge 1) +\nonumber\\
&& \quad\quad\quad\quad\quad\quad\quad\quad\quad\quad\quad\quad\quad \P(N_{\tilde{F}}{[\sqrt n x - 2 n^{-c}, \sqrt n x - 2 n^{-c}]}\ge 2)\nonumber\\
&\le& \iint_{B(x, 2 n^{-c-1/2})\cap \C_{+}} \rho^{(0, 1)}(z)dz + \int_{x - 2 n^{-c-1/2}}^{x + 2 n^{-c-1/2}}\int_{x - 2 n^{-c-1/2}}^{x + 2 n^{-c-1/2}} \rho^{(2, 0)}(s, t)dsdt\nonumber
 \end{eqnarray}
 where $\rho^{(0, 1)}$ and $\rho^{(2, 0)}$ are the $(0, 1)$- and $(2, 0)$-correlation functions of $\tilde P_n$ respectively. By \cite[Proposition 13.3]{TVpoly}, these functions are bounded for all $z\in B(x, 2 n^{-c-1/2})\cap \C_{+}$ and $s, t \in [x - 2 n^{-c-1/2}, x + 2 n^{-c-1/2}]$ as follows 
 $$\rho^{(0, 1)}(x, y)= O(n^{3/2})(x-y) = O(n^{1-c})$$
 and 
 $$\rho^{(2, 0)}(z) = O(n).$$
 Thus, 
\begin{eqnarray}
 \P(N_{\tilde{F}}{B(\sqrt n x,2 n^{-c})}\ge 2)=O(n^{-2c})\nonumber
 \end{eqnarray}
giving the desired estimate.
\end{proof}

\begin{proof}[Proof of Corollary \ref{mean_elliptic}]
As mentioned in remark \ref{remark_elliptic}, it suffices to show that 
$$\E N_{P_n}[0, 1]=\frac{1}{4}\sqrt{n} + O(n^{1/2-c}).$$ 

We partition the interval $[0, 1]$ into 2 intervals $I_1:=[0, n^{-1/2+\ep}]$ and $I_2:=[n^{-1/2+\ep}, 1]$. On the interval $I_2$ where Theorem \ref{uni_elliptic} applies, we further partition it into equal intervals $J_i$ of length $n^{-1/2}$. On each of these small intervals $J_{i}$, we routinely approximate its indicator function above and below by smooth test functions and apply Theorem \ref{uni_elliptic} to these functions to obtain
$$\E N_{P_n}(J_i) - \E N_{\tilde P_n} (J_{i}) = O(n^{-c}).$$

Thus, 
$$\E N_{P_n}(I_2) - \E N_{\tilde P_n} (I_2) = O(n^{1/2-c}).$$
It remains to show that the interval $I_1$ is insignificant. Note that $N_{P_n}(I_1)\le N_{P_n}B(x, 3x)$ where $x = n^{-1/2+\ep}$. By Jensen's inequality, 
$$N_{P_n}B(x, 3x) \le C\log \frac{M}{|P_n(x)|}$$
where $M = \max_{|z|\le 4 x} |P_n(z)|$. By \eqref{ell1}, on the event $\Omega'$,
$$M\le e^{n^{c_1}}\sqrt{n}\sqrt{\sum_{i =0}^{n}{n\choose i} |4x|^{2i}} = e^{n^{\ep}}\sqrt{n} (16x^{2}+1)^{n/2}\le \sqrt{n}e^{n^{3\ep}}.$$
Thus, $\P\left (\log M\ge n^{3\ep}\right )\le \frac{n}{e^{n^{\ep}}}$. Moreover, by \eqref{ell2}, we have $ \P\left (|P_n(x)|\le e^{-n^{\ep}}\right )\le n^{-A}$. Combining these bounds, we get 
$$\P\left (N _{P_n}B(x, 3x)\ge Cn^{3\ep}\right )\le Cn^{-2}.$$
Hence, 
$$\E N_{P_n}B(x, 3x) \le Cn^{3\ep} + n. n^{-2} \le (C+1)n^{3\ep}.$$
This completes the proof.
\end{proof}

\section{Proof of Theorem \ref{kacseries_uni} and Corollary \ref{kacseries_cor}}\label{kacseries_proof}
\begin{proof}[Proof of Theorem \ref{kacseries_uni}]
The reader may notice that this proof is quite similar to the proof of Theorem \ref{kacreal}. We nonetheless present it here for the reader's convenience. 

Let us first consider the case $0<\delta<\frac{1}{K}$ for some sufficiently large constant $K>0$.

We apply Theorem \ref{greal} to the random function 
$F(z) := P(z\delta/10)$
and the domain $D := \{z: 1-2\delta\le |z\delta/10|\le 1-\delta\}$.

For this random series, we set $\alpha_1=\min\{1/4, \gamma/2\}$ and $C_1=1$. The main task is to show that for any positive constants $A, c_1$, there exists a constant $C$ for which Conditions {\bf C2} \eqref{cond-poly}-{\bf C2} \eqref{cond-delocal} hold with parameters $(k+l, C_1, \alpha_1, A, c_1, C)$. 

We use the following crucial property of regularly varying coefficients.
\begin{lemma}\cite[Theorem 5, page 423]{feller1966introduction}\label{lmmregular_varying}
If $c_k^{2} = \frac{k^{\gamma-1}L(k)}{\Gamma(\gamma)}$ where $L(k)$ is a slowly varying function then
$$\lim_{a\downarrow 0} \sum_{k=0}^{\infty} c_k^{2} (1-at)^{2k} (2at)^{\gamma}/L\left (\frac{1}{a}\right )=1$$
uniformly as long as $t$ stays in a compact subset of $(0, \infty)$. 
\end{lemma}

Moreover, for any positive constant $c'>0$, there exists a constant  $C>0$ (depending on the function $L$ such that $\frac{1}{Ct^{c'}}\le L(t)\le Ct^{c'}$ for all $t>0$. This simple observation can be proven using, for example, the Karamata's representation theorem (\cite[Proposition 1.3.8, page 26]{bingham1989regular}. 

To verify Condition {\bf C2} \eqref{cond-delocal}, we use Lemma \ref{lmmregular_varying} to get for every $w\in B(0, 1-\delta/2)$, 
$$\sum_{k=0}^{\infty} c_k^{2} |w|^{2k} = \Omega\left (\delta^{-\gamma} L(\delta^{-1})\right )=\Omega\left (\delta^{-\gamma+c'}\right )$$
while 
$$c_k^{2}|w|^{2k}\le Ck^{\gamma-1+c'} (1-\delta)^{2k}=O\left (\delta^{-\gamma+1-2c'}+1\right ).$$
Letting $c'$ sufficiently small, we obtain Condition {\bf C2} \eqref{cond-delocal}.

Condition {\bf C2} \eqref{cond-repulsion} follows immediately from Lemma \ref{lmmregular_varying}.

To verify  Condition {\bf C2} \eqref{cond-bddn}, notice that  for any $M>2$,  if we condition on the event  $\Omega'$ on which $|\xi_i|\le M\left (1+\delta/2\right )^{i}$ for all $i$,  then  for all $z\in D + B(0, 3)$, by Lemma \ref{lmmregular_varying},
\begin{equation}
|F(z)| = O(M)\sum_{i=0}^{\infty} (1+|c_i|^{2})\left (1+\delta/2\right )^{i}(1-\delta)^{i} = O(M \delta^{-\gamma-1}).\label{interm11}
\end{equation}

Thus, for every $M>2$, we have
\begin{equation}
\P\left (|F(z)| = O(M\delta^{-\gamma-1})\right )= 1-O\left (\sum _{i=0}^{n} \frac{1}{M\left (1+\delta/2\right )^{i}}\right )= 1- O\left (\frac{1}{M\delta}\right ).\label{bound_kacseries}
\end{equation}

Setting $M = \delta^{-A-1}$, we obtain Condition {\bf C2} \eqref{cond-bddn}.

To prove Condition {\bf C2} \eqref{cond-smallball}, we show that for any constants $A$ and $c_1>0$, there exists a constant $B>0$ such that the following holds. For every $z_0$ with $1-2\delta\le |z_0|\le 1-\delta$, there exists $z= z_0e^{i\theta}$ where $\theta\in [-\delta, \delta]$ such that for every $M\ge 1$, 
\begin{equation}
\P\left (|P(z)|\le e^{-\delta^{-c_1}}e^{-BM}\right )\le \frac{B\delta^{A}}{M^{A}}.\label{smallball_kacseries}
\end{equation}

Setting $M = 1$, we obtain Condition {\bf C2} \eqref{cond-smallball}.

By writing $z_ 0 = re ^{i\theta_0}$, the bound \eqref{smallball_kac} follows from a more general anti-concentration bound: there exists $\theta\in I := [\theta_0 - \delta, \theta_0 + \delta]$ such that 
\begin{equation}
\sup _{Z\in \C}\P\left (|P(re^{i\theta})-Z|\le e^{-\delta^{-c_1}}e^{-BM}\right )\le \frac{B\delta^{A}}{M^{A}}.\nonumber
\end{equation}
 
Since the probability of being confined in a complex ball is bounded from above by the probability of its real part being confined in the corresponding interval on the real line, it suffices to show that
 \begin{equation}
\sup _{Z\in \R}\P\left (\left |\sum_{j=0}^{M\delta^{-1}/2} c_j\xi_j r^{j} \cos{j\theta}-Z\right |\le e^{-\delta^{-c_1}}e^{-BM}\right )\le \frac{B\delta^{A}}{M^{A}}.\nonumber
\end{equation}
 
This is a direct application of Lemma \ref{lmanti_concentration}.

Finally, to prove Condition {\bf C2} \eqref{cond-poly}, from \eqref{bound_kac}, \eqref{smallball_kac}, and Jensen's inequality, we get for every $1\le M\le n\delta$
$$\P(N\ge \delta^{-c_1} + BM) = O\left (\frac{\delta^{A}}{M^{A}}\right )$$
where $N = N_{F}B(w, 2)$, $w\in D$. 

Setting $c_1=1$ and $M = 1, 2, 2^{2}, \dots$, we get
$$\E N^{k+2}\textbf{1}_{N\ge \delta^{-1}} \le C\sum_{i=1}^{\infty} \left (\delta^{-1} + B2^{i+1}\right )^{k+2}\frac{\delta^{A}}{2^{iA}} \le C\delta^{A-k-2}.$$
 This proves Condition {\bf C2} \eqref{cond-poly} and completes the proof for $\delta\le 1/K$. For $\delta\ge 1/K$, note that the Jensen's inequality implies that
$$N_{P}B(0, 1-1/K) = O_K(1)\log\frac{\max_{w\in B(0, 1-1/2K) }|P(w)|}{\max_{w\in B(1-1/K, 1/3K)} |P(w)|}.$$

Thus, using the bounds \eqref{interm1}, \eqref{bound_kac}, \eqref{smallball_kac} for $\theta = 1-1/K$, and apply we get for every $1\le M$,
$$\P(N_{P}B(0, 1-1/2K) \ge BM) = O\left (\frac{C'}{M^{A}}\right ).$$
And so, $\E N_{P}B(0, 1-1/2K) = O(1)$. The same holds for $\tilde P$ and therefore desired result follows.
 \end{proof}

\begin{proof}[Proof of Corollary \ref{kacseries_cor}]
To prove the first part of Corollary \ref{kacseries_cor}, we decompose the interval $[0, r]$ into dyadic intervals $[0, 1/2],  [1-1/2, 1-1/4), \dots$, and finally $\pm [1-\delta, r]$. In each of these interval, say $[x, y)$, we show that $\E N_{P}[x, y) - \E N_{\tilde P} [x, y) = O((1-y)^{c})$ for some positive constant $c$. This can be routinely done by approximating the indicator function on the interval $[x, y)$ by a smooth function and apply Theorem \ref{kacreal}. We omit the detail as it is similar to the proof of Theorem \ref{comparison}.

Thanks to the first part, to prove the second part of Corollary \ref{kacseries_cor}, it suffices to prove the corresponding statement for $\tilde P$ whose coefficients are Gaussian. We adapt a strategy in \cite{flasche2020expected}. For any interval $[a, b]\subset \R$, by the Kac-Rice formula (Proposition \ref{KacRice}), we have
$$\E N_{\tilde P}[a, b] = \frac{1}{\pi}\int_{a}^{b}\sqrt{f(x)}dx$$
where 
$$f(x) = \frac{\left (\sum_{k=0}^{\infty} c_k^{2}x^{2k}\right )\left (\sum_{k=0}^{\infty} c_k^{2}k^{2}x^{2k-2}\right )-\left (\sum_{k=0}^{\infty} c_k^{2}kx^{2k-1}\right )^{2}}{\left (\sum_{k=0}^{\infty} c_k^{2}x^{2k}\right )^{2}}.$$

Lemma \ref{lmmregular_varying} suggests that we make the transformation
$$f_n(t) := f(1-2^{-n}t).$$
Applying Lemma \ref{lmmregular_varying} to $a = 2^{-n}$ and $t\in [1, 2]$, we obtain that uniformly on $x=1-at\in [1-2^{1-n}, 1-2^{n}]$, as $n\to \infty$
$$\sum_{k=0}^{\infty} c_k^{2}x^{2k} \sim 2^{-\gamma} (1-x)^{-\gamma}L\left (2^{n}\right ), \sum_{k=0}^{\infty} c_k^{2}kx^{2k-1} \sim x^{-1}2^{-\gamma-1} (1-x)^{-\gamma-1}L\left (2^{n}\right )\frac{\Gamma(\gamma+1)}{\Gamma(\gamma)}$$
and
$$ \sum_{k=0}^{\infty} c_k^{2}k^{2}x^{2k-2} \sim x^{-2}2^{-\gamma-2} (1-x)^{-\gamma-2}L\left (2^{n}\right )\frac{\Gamma(\gamma+2)}{\Gamma(\gamma)}$$
where $p_n\sim q_n$ means $\lim_{n\to\infty} \frac{p_n}{q_n}=1$. 

Since $\Gamma(\gamma+2) = (\gamma+1)\Gamma(\gamma+1) = \gamma(\gamma+1)\Gamma(\gamma)$, we obtain that uniformly on $t\in [1, 2]$, 
$$f_n(t)\sim \gamma (2^{-n}t)^{-2}/4.$$ 
We have
$$\E N_{\tilde P}[1-2^{1-n}, 1-2^{-n}] = \frac{1}{\pi}\int_{1}^{2}2^{-n}\sqrt{f_n(t)}dt.$$
By uniform convergence, we obtain
$$ \E N_{\tilde P}[1-2^{1-n}, 1-2^{-n}] \sim \frac{\sqrt\gamma\ln 2}{2\pi}.$$

Taking the Ces\'aro summation, we obtain
$$\frac{1}{n}\E N_{\tilde P}[0, 1-2^{-n}] = \frac{1}{n}\sum_{k=1}^{n}\E N_{\tilde P}[1-2^{1-k}, 1-2^{-k}] \sim \frac{\sqrt\gamma\ln 2}{2\pi}.$$

For each $r\in (0, 1)$, sandwiching $\E N_{\tilde P}[0, r]$ between $\E N_{\tilde P}[0, 1-2^{1-n}]$ and $\E N_{\tilde P}[0, 1-2^{-n}]$ (i.e., $n-1 = \lfloor -\log_{2}(1-r)\rfloor$), we get
$$\frac{1}{-\log (1-r)}\E N_{\tilde P}[0, r]\sim \frac{\sqrt\gamma}{2\pi}.$$
as desired.
\end{proof}

\emph{Acknowledgements.} The authors would like to thank Asaf Ferber and Yuval Peres for helpful remarks that led to some simplifications of our proofs. We thank the anonymous referees for their helpful suggestions.

\bibliographystyle{plain}
\bibliography{polyref}

\begin{thebibliography}{10}

\bibitem{angstpoly}
J\"{u}rgen Angst and Guillaume Poly.
\newblock A weak {C}ram\'{e}r condition and application to {E}dgeworth
  expansions.
\newblock {\em Electron. J. Probab.}, 22:Paper No. 59, 24, 2017.

\bibitem{azais2015local}
Jean-Marc Aza{\"\i}s, Federico Dalmao, Jos{\'e} Le{\'o}n, Ivan Nourdin, and
  Guillaume Poly.
\newblock Local universality of the number of zeros of random trigonometric
  polynomials with continuous coefficients.
\newblock {\em arXiv preprint arXiv:1512.05583}, 2015.

\bibitem{BS}
Albert~T Bharucha-Reid and Masilamani Sambandham.
\newblock {\em Random Polynomials: Probability and Mathematical Statistics: a
  Series of Monographs and Textbooks}.
\newblock Academic Press, 1986.

\bibitem{bingham1989regular}
Nicholas~H Bingham, Charles~M Goldie, and Jef~L Teugels.
\newblock {\em Regular variation}, volume~27.
\newblock Cambridge university press, 1989.

\bibitem{bleher1997correlations}
Pavel Bleher and Xiaojun Di.
\newblock Correlations between zeros of a random polynomial.
\newblock {\em Journal of Statistical Physics}, 88(1):269--305, 1997.

\bibitem{BD}
Pavel Bleher and Xiaojun Di.
\newblock Correlations between zeros of non-{Gaussian} random polynomials.
\newblock {\em International Mathematics Research Notices},
  2004(46):2443--2484, 2004.

\bibitem{BP}
Andr\'e Bloch and Gy{\"o}rgy P{\'o}lya.
\newblock On the roots of certain algebraic equations.
\newblock {\em Proceedings of the London Mathematical Society}, 2(1):102--114,
  1932.

\bibitem{CNNV}
Mei-Chu Chang, Hoi Nguyen, Oanh Nguyen, and Van Vu.
\newblock Random eigenfunctions on flat tori: universality for the number of
  intersections.
\newblock {\em International Mathematics Research Notices},
  2020(24):9933--9973, 2020.

\bibitem{das1968trig}
Minaketan Das.
\newblock The average number of real zeros of a random trigonometric
  polynomial.
\newblock In {\em Mathematical Proceedings of the Cambridge Philosophical
  Society}, volume~64, pages 721--730. Cambridge Univ Press, 1968.

\bibitem{DHV}
Yen Do, Hoi Nguyen, and Van Vu.
\newblock Real roots of random polynomials: expectation and repulsion.
\newblock {\em Proceedings of the London Mathematical Society},
  111(6):1231--1260, 2015.

\bibitem{DOV}
Yen Do, Oanh Nguyen, and Van Vu.
\newblock Roots of random polynomials with coefficients of polynomial growth.
\newblock {\em Annals of probability}, 46(5):2407--2494, 2018.

\bibitem{do2020random}
Yen Do, Oanh Nguyen, and Van Vu.
\newblock Random orthonormal polynomials: local universality and expected
  number of real roots.
\newblock {\em arXiv preprint arXiv:2012.10850}, 2020.

\bibitem{Dunnage1966number}
J.E.A. Dunnage.
\newblock The number of real zeros of a random trigonometric polynomial.
\newblock {\em Proceedings of the London Mathematical Society}, 3(1):53--84,
  1966.

\bibitem{EK}
Alan Edelman and Eric Kostlan.
\newblock How many zeros of a random polynomial are real?
\newblock {\em Bulletin of the American Mathematical Society}, 32(1):1--37,
  1995.

\bibitem{EO}
Paul Erd{\"o}s and Albert~C. Offord.
\newblock On the number of real roots of a random algebraic equation.
\newblock {\em Proceedings of the London Mathematical Society}, 3(1):139--160,
  1956.

\bibitem{esseen1966kolmogorov}
Carl-Gustav Ess\'een.
\newblock On the {Kolmogorov-Rogozin} inequality for the concentration
  function.
\newblock {\em Probability Theory and Related Fields}, 5(3):210--216, 1966.

\bibitem{Far}
Kambiz Farahmand.
\newblock {\em Topics in random polynomials}, volume 393.
\newblock CRC Press, 1998.

\bibitem{feller1966introduction}
William Feller.
\newblock {\em An introduction to probability theory and its applications.
  {V}ol. {I}}.
\newblock Third edition. John Wiley \& Sons, Inc., New York-London-Sydney,
  1968.

\bibitem{flasche}
Hendrik Flasche.
\newblock Expected number of real roots of random trigonometric polynomials.
\newblock {\em Stochastic Processes and their Applications}, 2017.

\bibitem{flasche2020expected}
Hendrik Flasche and Zakhar Kabluchko.
\newblock Expected number of real zeroes of random {Taylor} series.
\newblock {\em Communications in Contemporary Mathematics}, 22(07):1950059,
  2020.

\bibitem{flasche2018real}
Hendrik Flasche and Zakhar Kabluchko.
\newblock Real zeroes of random analytic functions associated with geometries
  of constant curvature.
\newblock {\em Journal of Theoretical Probability}, 33(1):103--133, 2020.

\bibitem{folland}
Gerald~B Folland.
\newblock {\em Real analysis: modern techniques and their applications}.
\newblock John Wiley \& Sons, 2013.

\bibitem{forrester1999exact}
Peter~J Forrester and G~Honner.
\newblock Exact statistical properties of the zeros of complex random
  polynomials.
\newblock {\em Journal of Physics A: Mathematical and General}, 32(16):2961,
  1999.

\bibitem{GKZ}
Friedrich G\"{o}tze, Dzianis Kaliada, and Dmitry Zaporozhets.
\newblock Correlation functions of real zeros of random polynomials.
\newblock {\em Zap. Nauchn. Sem. S.-Peterburg. Otdel. Mat. Inst. Steklov.
  (POMI)}, 454(Veroyatnost i Statistika. 24):102--111, 2016.

\bibitem{halasz1977estimates}
G\'abor Hal{\'a}sz.
\newblock Estimates for the concentration function of combinatorial number
  theory and probability.
\newblock {\em Periodica Mathematica Hungarica}, 8(3-4):197--211, 1977.

\bibitem{HKPV}
John~Ben Hough, Manjunath Krishnapur, Yuval Peres, and B{\'a}lint Vir{\'a}g.
\newblock {\em Zeros of Gaussian analytic functions and determinantal point
  processes}, volume~51.
\newblock American Mathematical Society Providence, RI, 2009.

\bibitem{Ibragimov1968average}
Ildar~A. Ibragimov and Nina~B. Maslova.
\newblock The average number of zeros of random polynomials.
\newblock {\em Vestnik Leningrad. Univ}, 23:171--172, 1968.

\bibitem{Ibragimov1971expected1}
Ildar~A. Ibragimov and Nina~B. Maslova.
\newblock On the expected number of real zeros of random polynomials i.
  coefficients with zero means.
\newblock {\em Theory of Probability \& Its Applications}, 16(2):228--248,
  1971.

\bibitem{iksanov2016local}
Alexander Iksanov, Zakhar Kabluchko, and Alexander Marynych.
\newblock Local universality for real roots of random trigonometric
  polynomials.
\newblock {\em Electronic Journal of Probability}, 21, 2016.

\bibitem{KZ3}
Zakhar Kabluchko and Dmitry Zaporozhets.
\newblock Universality for zeros of random analytic functions.
\newblock {\em arXiv preprint arXiv:1205.5355}, 2012.

\bibitem{kabluchko2014asymptotic}
Zakhar Kabluchko and Dmitry Zaporozhets.
\newblock Asymptotic distribution of complex zeros of random analytic
  functions.
\newblock {\em The Annals of Probability}, 42(4):1374--1395, 2014.

\bibitem{Kac1943average}
Mark Kac.
\newblock On the average number of real roots of a random algebraic equation.
\newblock {\em Bulletin of the American Mathematical Society}, 49(1):314--320,
  1943.

\bibitem{Kac2}
Mark Kac.
\newblock On the average number of real roots of a random algebraic equation
  (ii).
\newblock {\em Proceedings of the London Mathematical Society}, 2(1):390--408,
  1948.

\bibitem{kahane1985}
Jean-Pierre Kahane.
\newblock {\em Some random series of functions, volume 5 of Cambridge Studies
  in Advanced Mathematics}, volume 259.
\newblock Cambridge University Press, Cambridge, 1985.

\bibitem{ledoan2012universality}
Andrew Ledoan, Marco Merkli, and Shannon Starr.
\newblock A universality property of {Gaussian} analytic functions.
\newblock {\em Journal of Theoretical Probability}, 25(2):496--504, 2012.

\bibitem{LO1}
John~E. Littlewood and Albert~C. Offord.
\newblock On the number of real roots of a random algebraic equation (iii).
\newblock {\em Rec. Math. [Mat. Sbornik]}, 12(3):277--286, 1943.

\bibitem{LO2}
John~E. Littlewood and Albert~C. Offord.
\newblock On the distribution of the zeros and $\alpha$-values of a random
  integral function (i).
\newblock {\em Journal of the London Mathematical Society}, 1(3):130--136,
  1945.

\bibitem{LO3}
John~E. Littlewood and Albert~C. Offord.
\newblock On the distribution of zeros and a-values of a random integral
  function (ii).
\newblock {\em Annals of Mathematics}, pages 885--952, 1948.

\bibitem{Nazarov94}
Fedor Nazarov.
\newblock Local estimates of exponential polynomials and their applications to
  inequalities of uncertainty priciple type.
\newblock {\em St Petersburg Mathematical Journal}, 5(4):663--718, 1994.

\bibitem{NNS}
Fedor Nazarov, Alon Nishry, and Mikhail Sodin.
\newblock Log-integrability of {Rademacher Fourier} series, with applications
  to random analytic functions.
\newblock {\em St. Petersburg Mathematical Journal}, 25(3):467--494, 2014.

\bibitem{NNV}
Hoi Nguyen, Oanh Nguyen, and Van Vu.
\newblock On the number of real roots of random polynomials.
\newblock {\em Communications in Contemporary Mathematics}, page 1550052, 2015.

\bibitem{nguyen2011optimal}
Hoi Nguyen and Van Vu.
\newblock Optimal inverse {Littlewood--Offord} theorems.
\newblock {\em Advances in Mathematics}, 226(6):5298--5319, 2011.

\bibitem{nguyenvusurvey}
Hoi~H Nguyen and Van Vu.
\newblock Small ball probability, inverse theorems, and applications.
\newblock In {\em Erd{\H{o}}s Centennial}, pages 409--463. Springer, 2013.

\bibitem{pritsker1}
Igor~E Pritsker.
\newblock Zero distribution of random polynomials.
\newblock {\em Journal d'Analyse Math{\'e}matique}, 134(2):719--745, 2018.

\bibitem{pritsker2}
Igor~E Pritsker and Aaron~M Yeager.
\newblock Zeros of polynomials with random coefficients.
\newblock {\em Journal of Approximation Theory}, 189:88--100, 2015.

\bibitem{prosen1996}
Tomaz Prosen.
\newblock Parametric statistics of zeros of husimi representations of quantum
  chaotic eigenstates and random polynomials.
\newblock {\em Journal of Physics A: Mathematical and General}, 29(17):5429,
  1996.

\bibitem{Ru}
Walter Rudin.
\newblock {\em Real and complex analysis (3rd)}.
\newblock New York: McGraw-Hill Inc, 1986.

\bibitem{sambandham1978number}
Masilamani Sambandham.
\newblock On the number of real zeros of a random trigonometric polynomial.
\newblock {\em Transactions of the American Mathematical Society}, pages
  57--70, 1978.

\bibitem{MP}
Gr{\'e}gory Schehr and Satya~N Majumdar.
\newblock Real roots of random polynomials and zero crossing properties of
  diffusion equation.
\newblock {\em Journal of Statistical Physics}, 132(2):235, 2008.

\bibitem{shalom2010finitary}
Yehuda Shalom and Terence Tao.
\newblock A finitary version of {Gromov}’s polynomial growth theorem.
\newblock {\em Geometric and Functional Analysis}, 20(6):1502--1547, 2010.

\bibitem{sodin2005zeroes}
Mikhail Sodin.
\newblock Zeroes of gaussian analytic functions.
\newblock In {\em European congress of mathematics}, pages 445--458, 2005.

\bibitem{sodin2004random}
Mikhail Sodin and Boris Tsirelson.
\newblock Random complex zeroes, {I. Asymptotic} normality.
\newblock {\em Israel Journal of Mathematics}, 144(1):125--149, 2004.

\bibitem{soze1}
Ken S{\"o}ze.
\newblock Real zeroes of random polynomials, {I.} {F}lip-invariance,
  {T}ur{\'a}n’s lemma, and the {Newton-Hadamard} polygon.
\newblock {\em Israel Journal of Mathematics}, 220(2):817--836, 2017.

\bibitem{soze2}
Ken S{\"o}ze.
\newblock Real zeroes of random polynomials, {II. D}escartes’ rule of signs
  and anti-concentration on the symmetric group.
\newblock {\em Israel Journal of Mathematics}, 220(2):837--872, 2017.

\bibitem{starr2011universality}
Shannon Starr.
\newblock Universality of correlations for random analytic functions.
\newblock {\em Entropy and the quantum II}, 552:135--144, 2011.

\bibitem{Stev}
Donald~C. Stevens.
\newblock The average number of real zeros of a random polynomial.
\newblock {\em Communications on Pure and Applied Mathematics}, 22(4):457--477,
  1969.

\bibitem{taovubook}
Terence Tao and Van Vu.
\newblock {\em Additive combinatorics}, volume 105.
\newblock Cambridge University Press, 2006.

\bibitem{TVpoly}
Terence Tao and Van Vu.
\newblock Local universality of zeroes of random polynomials.
\newblock {\em International Mathematics Research Notices}, page rnu084, 2014.

\bibitem{turan1953}
Paul Tur{\'a}n.
\newblock {\em Eine neue Methode in der Analysis und deren Anwendungen}.
\newblock Akad{\'e}miai Kiad{\'o} Budapest, 1953.

\bibitem{wilkins1991trig}
J.~Ernest Wilkins.
\newblock Mean number of real zeros of a random trigonometric polynomial.
\newblock {\em Proceedings of the American Mathematical Society},
  111(3):851--863, 1991.

\bibitem{zelditch2001random}
Steve Zelditch.
\newblock From random polynomials to symplectic geometry.
\newblock In {\em XIIIth International Congress on Mathematical Physics
  (London, 2000), 367--376, Int}. Citeseer, 2001.

\end{thebibliography}
\section{Appendix}
\subsection{Proof of Lemma \ref{fourier}}\label{fourier_proof}
This proof is taken from \cite{TVpoly}. We will only prove the first part of the Lemma relating to Theorem \ref{gcomplex} as the second part is similar. By translation, we can assume without loss of generality that $z_1 = \dots = z_k=0$. 
Suppose that we have (\ref{gcomplexb}) for $G$ in the form \eqref{h2}. Let $r_0=1/100$. Then, for every function $G$ supported in $\prod _{j=1}^{k}B(0, r_0)$ with $\norm{\triangledown^aG}_\infty\le 1$ for all $0\le a\le 2k+4$, we view it as a smooth function on the torus $(\mathbb{R}/(2.2 r_0)\mathbf{Z})^{2k}$. Expanding $G$ by Fourier series yields
\begin{equation}
G(w) = \sum_{b, c\in \mathbb{Z}^{k}}g_{b, c}e^{2\pi\sqrt{-1}(b\text{Re}(w) + c\text{Im}(w)/(2.2r_0)},\label{tue1}
\end{equation}
for $w\in (\mathbb{R}/(2.2 r_0)\mathbf{Z})^{2k}$, where 
\[g_{b, c} = \frac{1}{(2.2r_0)^{2k}}\int_{B(0, r_0)^{k}}e^{-2\pi\sqrt{-1}(b\text{Re}(w) + c\text{Im}(w))/(2.2r_0)}G(w)\text{d}w,
\]
and the convergence is point-wise (by, for example, \cite[Theorem 8.32]{folland}).

By integration by parts (or \cite[Theorem 8.22e]{folland}), we have 
\[|g_{b, c}|\le C(1+|b|+|c|)^{-2k-4},
\]
where $C = C_k$.

Let $\eta:\mathbb{R}\to\mathbb{R}$ be a smooth function supported on $[-1.1 r_0, 1.1r_0]$ that equals 1 on $[-r_0, r_0]$ and $||\eta||_\infty \le 1$, and let
\[\psi_{b,c, i} = e^{2\pi\sqrt{-1}(b_{i}\text{Re}(w_{i}) + c_{i}\text{Im}(w_{i})/(2.2r_0)}\eta(\text{Re}(w_{i}))\eta(\text{Im}(w_{i})), 
\]
and
\[G_{b, c}(w) = g_{b, c}\prod_{i=1}^{k}\psi_{b, c, i}(w_i).
\]
Since $G$ is supported on $[-r_0, r_0]^{2k}$, multiplying both sides of \eqref{tue1} by $\prod_{i=1}^{k}\eta(\text{Re}(w_{i}))\eta(\text{Im}(w_{i}))$, we have
\[G(w) = \sum_{b, c\in \mathbf{Z}^{k}}G_{b, c}(w),
\]
pointwise.
We have that $\psi_{b, c,i}$ is supported on $B(0, 2.2r_0)$ and $|{\triangledown^aG_{b, c}}|\le C(1+|b|+|c|)^{3}|g_{b, c}|, \forall 0\le a\le 3$. We thus have for all $m\ge 1$
\begin{eqnarray}
&&\left |\E \sum_{i_1, \dots, i_k} G_m(\zeta_{i_1}, \dots, \zeta_{i_k}) -\E \sum_{i_1, \dots, i_k} G_m(\tilde \zeta_{i_1}, \dots, \tilde \zeta_{i_k})\right | 
\le C\delta_n^c \sum_{b, c\in \mathbf{Z}^{k}}(1+|b|+|c|)^{3}|g_{b, c}|\nonumber\\
&&\quad \le C\delta_n^c \sum_{b, c\in \mathbf{Z}^{k}}(1+|b|+|c|)^{-2k-1} = C\delta_n^c \sum_{m=0}^{\infty} \sum_{b, c\in \mathbf{Z}^{k}, |b|+|c| = m }(1+m)^{-2k-1}\nonumber\\
& &\quad \le C\delta_n^{c} \sum_{m=0}^{\infty} (1+m)^{-2k-1} m^{2k-1}\le C \delta_n^{c} \sum_{m=1}^{\infty} m^{-2} \le C\delta_n^{c}\nonumber
\end{eqnarray}
where $G_m = \sum_{|b|+|c|\le m}G_{b, c}$ supported in $B(0, 2r_0)^{k}$ and we recall that the constant $C$ may change from one equation to another. Using Condition {\bf C2} \eqref{cond-poly} and the fact that $G_m\to G$ point-wise and $|G_m|=O(1)$, by dominated convergence theorem, we get 
$$\lim _{m\to \infty}\E \sum_{i_1, \dots, i_k} G_m(\zeta_{i_1}, \dots, \zeta_{i_k}) = \E \sum_{i_1, \dots, i_k} G(\zeta_{i_1}, \dots, \zeta_{i_k}).$$ 
And hence the above inequalities hold for $G$ in place of $G_m$, completing the proof.
\subsection{Proof of Lemma \ref{2norm}}\label{2norm_proof} We follow ideas from \cite{DOV}; the constant $6$ in the conclusion is adhoc but we make no attempt to optimize it.

From  Jensen's inequality
for the number of roots (see the beginning of Section  \ref{app1_proof_1}), we have 
\[N_{F_n}({B(w, 1)})\le \log \frac{5}{2} (\log M - \log|F_n(w)|) < 2   (\log M - \log|F_n(w)|)
\]
where $N_{F_n}({B(w, 1)})$ is the number of zeros of $F_n$ in $B(w, 1)$ and $M = \max_{|w - z| = 2} |F_n(z)|$.

From this and the assumption of Lemma \ref{2norm}, we conclude that
\begin{equation}
N_{F_n}({B(w, 1)})\le 2 \delta_n^{-c_2}.\label{noncluster}
\end{equation}

By the pigeonhole principle, there exists a radius $1\ge r\ge 1/2$ for which $F_n$ has no zeros in the annulus $B(w, r +\eta)\setminus B(w, r-\eta)$ where $\eta =.1\delta_n ^{c_2}$. We can also assume, without loss of generality, that there is no root on the boundary of each disk.

 Let $\zeta_1, \dots, \zeta_m$ be the zeros of $F_n$ in the disk $B(w, r - \eta)$. By \eqref{noncluster},   $m \le 2 \delta_n^{-c_2}$. Define  
 $$f(z) := \frac{F_n(z)}{(z-\zeta_1)\dots(z-\zeta_m)}. $$  Since $f$ is an entire function which does not have zeros in the (closed)  disk $B(w, r +\eta)$, $\log|f|$ is harmonic on this disk. For every $z$ with $|z-w| = r + \eta$, the distance from 
 $z$ to any $\zeta_i$ is at least $\eta$, so  
 $$|f(z)|\le |F_n(z)|\eta^{-m}\le \exp(\delta_n^{-c_2})\eta^{-m}.$$

It follows that for any $z$ where  $|z-w|= r + \eta$
 \begin{equation}\label{upperf} \log|f(z)|\le \delta_n^{-c_2} + m \log \eta^{-1 } \le 21 \delta_n^{-2c_2}, \end{equation}   since 
 $$\delta_n^{-c_2} \le \delta_n^{-2c_2}, m \le 2 \delta_n^{-c_2}, \eta^{-1} = 10 \delta_n^{-c_2} \le 
 e ^{10 \delta_n^{-c_2}}. $$
 Because of the harmonicity of $\log|f|$, its maximum 
 is achieved on the boundary, and so the same bound holds for all $z\in B(w, r+\eta)$.

On the other hand, from the lower bound on $|F(w)|$ in the lemma 
and the fact that  $|\zeta_i-w|\le 1$, 
 \begin{equation} \label{lowerf}  \log |f(w)|\ge \log|F_n(w)|\ge -\delta_n^{-c_2}. \end{equation}

Now, we make a critical use of Harnack's inequality \cite[Chapter 11]{Ru}, which asserts that if a function  $G$ is harmonic on the open disk $B(w, R)$  and is  nonnegative continuous  on its closure, for some $w\in \C$ and $R>0$, then  for every $z\in B(w, r)$ with $r<R$, 
$$G(z) \le \frac{R+r}{R-r} G(w). $$

We apply Harnak's inequality  to $G(z):= 21 \delta_n^{-2c_2} - \log|f|$ which is 
 nonnegative harmonic on $B(w, R)$ with 
$R:= r+ \eta$.  By this inequality, we conclude that for all  $z\in B(w, r)$
 \begin{equation} \label{Harnak1}  21 \delta_n^{-2c_2} - \log|f(z)|\le \frac{2r +\eta }{\eta} (21 \delta_n^{-2c_2} - \log|f(w)|). \end{equation}

As $\eta =.1 \delta_n^{c_2}$ and $r < 1$, $\frac{2r+ \eta}{\eta} \le 3 \eta^{-1} = 30 \delta _n^{-c_2}$. 
By \eqref{lowerf}, the right-hand side is at most 
$$30 \delta_n^{-c_2} \times 22 \delta_n^{-2c_2} = 660 \delta _n^{-3c_2}. $$ It follows that 
$$\log |f(z)|\ge 21 \delta_n^{-2c_2} -660\delta_n^{-3c_2} \ge -660 \delta_n^{-3c_2}.$$

Together with \eqref{upperf}, we have 
\begin{equation} \label{absolutef}  |\log |f(z)|| \le 660 \delta_n^{-3c_2}\quad\forall z\in B(w, r). \end{equation}

By the triangle inequality and the definition of $f$, 
\begin{equation} \label{Harnak2} 
\norm{\log |F_n(z)|}_{L^2(B( w, r))} \le \norm{\log |f(z)|}_{L^2(B( w, r))} + \sum_{i = 1}^{m}\norm{\log |z-\zeta_i|}_{L^2(B( w, r))}. 
\end{equation}

Notice that  each of the $m$ terms in the sum above is at most $\int_{B(0, 2r-\eta)} |\log |z||^{2}dz$, as 
$|\zeta_i| \le r -\eta$ for all $i$. As $r < 1$, we can further upper bound it by 
$\int_{B(0, 2)} |\log |z||^{2}dz$, which is $O(1)$ (in fact, one can easily show 
$\int_{B(0, 2)} |\log |z||^{2}dz < 30$, with room to spare). Since $m \le 2 \delta_n^{-c_2} $, the right-hand side of 
\eqref{Harnak2} is at most 
$$ 660 \delta_n^{-3 c_2} + 60 \delta_n^{-c_2} \le 720 \delta_n^{-3c_2} . $$ Thus, we have 
$$\norm{\log |F_n(z)|}_{L^2(B( w, r))} \le  720 \delta_n^{-3c_2} $$ which implies  the claim of the lemma as
 $r \ge 1/2$.

\subsection{Proof of Lemma \ref{logcomp}} \label{logcomp_proof} To prove Lemma \ref{logcomp}, we will follow the proofs in \cite{DOV} and \cite{TVpoly}. We first prove the following.
\begin{lemma}\label{log-com33}
Under the assumptions of Lemma \ref{logcomp}, there exist constants $\alpha_2>0$ and $C'>0$ such that for any ${z_1}, \dots, {z_k}\in D_n + B(0, 1/10)$ and for any function $L:\mathbb{C}^k\to \mathbb{C}$ with continuous derivatives up to order $3$ and $\norm{\triangledown^a L}_\infty \le \delta_n^{-\alpha_2}$ for all $0\le a\le 3$, we have 
\[\left|\E L\left(\frac{F_n(z_1)}{\sqrt{V(z_1)}}, \dots, \frac{F_n(z_k)}{\sqrt{V(z_k)}}\right)-\E L\left(\frac{\tilde F_n(z_1)}{\sqrt{V(z_1)}}, \dots, \frac{\tilde F_n(z_k)}{\sqrt{V(z_k)}}\right) \right|\le C'\delta_n^{\alpha_2},
 \]
 where $V(z_j)  := \sum _{i=N_0}^{n}|\phi_i(z_j)|^{2}$ and  $N_0$ is the constant in Condition {\bf C1}.
\end{lemma}

\begin{remark}
Following the proof, one can set $\alpha_2 = \frac{\alpha_1 \ep }{4}$.
\end{remark}
\begin{proof} [Proof of Lemma \ref{log-com33}] To prove this Lemma, we first observe that by replacing  $L$ by 
$$L'(z_1, \dots, z_k):= L \left (z_1 + \frac{\E F_n(z_1)}{\sqrt{V(z_1)}}, \dots, z_k + \frac{\E F_n(z_k)}{\sqrt{V(z_k)}}\right ),$$
if necessary, we can assume that $\E \tilde \xi_i = 0$ for all $i$ and $\E \xi_i =0$ for all 
$i > N_0$. (See Condition {\bf C1}.)

We use the Lindeberg swapping argument. Let $G_{i_0} = \sum_{i=1}^{i_0}\tilde{\xi}_i \phi_i(z) + \sum_{i=i_0+1}^n \xi_i \phi_i(z)$. The purpose is to swap the random variables one by one. Under these notations, $G_0 = F_n$ and $G_{n} = \tilde{F}_n$. Put
 \[I_{i_0} := \left|\E L\left(\frac{G_{i_0}(z_1)}{\sqrt{V(z_1)}},\dots, \frac{G_{i_0}(z_k)}{\sqrt{V(z_k)}}\right)-\E L\left(\frac{G_{i_0+1}(z_1)}{\sqrt{V(z_1)}},\dots, \frac{G_{i_0+1}(z_k)}{\sqrt{V(z_k)}}\right)\right|.
 \]

 Then
\[I: = \left|\E L\left (\frac{F_n(z_1)}{\sqrt{V(z_1)}}, \dots, \frac{F_n(z_k)}{\sqrt{V(z_k)}}\right )-\E L\left (\frac{\tilde F_n(z_1)}{\sqrt{V(z_1)}}, \dots, \frac{\tilde F_n(z_k)}{\sqrt{V(z_k)}}\right ) \right |\le \sum_{i_0=0}^{n} I_{i_0}. \]

 Fix $i_0\in [N_0, n]$ and let $Y_j := \frac{G_{i_0}(z_j)}{\sqrt{V(z_j)}} - \frac{\xi_{i_0}\phi_{i_0}(z_j)}{\sqrt{V(z_j)}}$ for $1\le j\le n$. 
 Then, $\frac{G_{i_0+1}(z_j)}{\sqrt{V(z_j)}} = Y_j + \frac{\tilde{\xi}_{i_0}\phi_{i_0}(z_j)}{\sqrt{V(z_j)}}$.
 Condition on $\xi_i$ for $i<i_0$ and $\tilde{\xi}_i$ for $i>i_0$.  The  $Y_j$'s become constants; the only randomness left comes from $\xi_{i_0}, \tilde \xi_{i_0}$. Define  $$\hat L = \hat L_{i_0}(w_1,\dots, w_k) := L(Y_1+w_1,\dots, Y_k+w_k). $$ By the definition of $\hat L$ and the assumption of the lemma, 
$\norm{\triangledown^a \hat L}_\infty \le C\delta_n^{-\alpha_2}$ for all $0\le a\le 3$.

 We are going to  estimate
$$d_{i_0} := \bigg| \E _{\xi_{i_0}, \tilde{\xi}_{i_0}}\hat L\left(\frac{\xi_{i_0}\phi_{i_0}(z_1)}{\sqrt{V(z_1)}},\dots, \frac{\xi_{i_0}\phi_{i_0}(z_k)}{\sqrt{V(z_k)}}\right)- \E _{\xi_{i_0}, \tilde{\xi}_{i_0}}\hat L\left(\frac{\tilde{\xi}_{i_0}\phi_{i_0}(z_1)}{\sqrt{V(z_1)}},\dots, \frac{\tilde{\xi}_{i_0}\phi_{i_0}(z_k)}{\sqrt{V(z_k)}}\right)\bigg|.$$

Let $a_{i, i_0} := \frac{\phi_{i_0}(z_i)}{\sqrt{V(z_i)}}$ and $a_{i_0} := (\sum_{i=1}^k |a_{i, i_0}|^2)^{1/2}$. Taylor expanding $\hat L$ around $(0,\dots, 0)$, we obtain
\begin{equation}
\hat L\left(a_{1, i_0}\xi_{i_0},\dots, a_{k, i_0}\xi_{i_0}\right) = \hat L(0) + \hat L_1 + \text{err}_1,\label{jul1}
\end{equation}
where \[\hat L_1 = \frac{\text{d}\hat L\left(a_{1, i_0}\xi_{i_0}t,\dots, a_{k, i_0}\xi_{i_0}t\right)}{\text{d}t}\bigg|_{t=0} = \sum_{i=1}^k \frac{\partial \hat L(0)}{\partial \Re(w_i)}\Re(a_{i, i_0}\xi_{i_0})+\sum_{i=1}^k \frac{\partial \hat L(0)}{\partial \Im(w_i)}\Im(a_{i, i_0}\xi_{i_0}).\]

(To avoid confusion, we use  $\partial $ to denote a partial derivative of functions of multi-variables and 
${\text {d}}$ to denote  a derivative of function of a single variable.) 

From the bounds on the derivatives of $\hat L$, we have 
\begin{eqnarray}
|\text{err}_1| &\le& \sup_{t\in [0, 1]} \left| \frac{1}{2}\frac{\text{d}^2 \hat L \left(a_{1, i_0}\xi_{i_0}t,\dots, a_{k, i_0}\xi_{i_0}t\right)}{\text{d}t^2}\right| \nonumber\\
&=& O\left (\delta_n^{-\alpha_2}|\xi_{i_0}|^2 k \sum_{i=1}^k|a_{i, i_0}|^2\right)= O\left (\delta_n^{-\alpha_2}|\xi_{i_0}|^2a_{i_0}^2\right )\label{err1}.
\end{eqnarray}
Similarly,
\begin{equation}
\hat L\left(a_{1, i_0}\xi_{i_0},\dots, a_{k, i_0}\xi_{i_0}\right) = \hat L(0) + \hat L_1 +\frac{1}{2}\hat L_2+ \text{err}_2,\label{jul2}
\end{equation}
where $\hat L_2 = \frac{\text{d}^2 \hat L(a_{1, i_0\xi_{i_0}}t,\dots, a_{k, i_0}\xi_{i_0}t)}{\text{d}t^2}\bigg|_{t=0}$ and 
\begin{eqnarray}
|\text{err}_2| &\le& \sup_{t\in [0, 1]} \left| \frac{1}{6}\frac{\text{d}^3 \hat L \left(a_{1, i_0}\xi_{i_0}t,\dots, a_{k, i_0}\xi_{i_0}t\right)}{\text{d}t^3}\right| \nonumber\\
&=& O\left (\delta_n^{-\alpha_2}|\xi_{i_0}|^3 \left (\sum_{i=1}^k|a_{i, i_0}|\right )^3\right)=  O\left (\delta_n^{-\alpha_2}|\xi_{i_0}|^3a_{i_0}^3\right ).\label{mm}
\end{eqnarray}

Note that as in \eqref{err1}, $\hat L_2 = O\left (\delta_n^{-\alpha_2}|\xi_{i_0}|^2a_{i_0}^2\right )$. Thus, 
\begin{eqnarray}
 \text{err}_2  &=& \text{err}_1-\frac{\hat L_2}{2} = O\left (\delta_n^{-\alpha_2}|\xi_{i_0}|^2a_{i_0}^2\right ).\label{err2}
\end{eqnarray}

Using \eqref{mm} and \eqref{err2},  we obtain 
\begin{equation} \label{err2-1} |\text{err}_2| =O\left (\delta_n^{- \alpha_2}\right )\min \{|\xi_{i_0}|^{2}a_{i_0}^{2},  |\xi_{i_0}|^{3}a_{i_0}^{3}\} =O\left (\delta_n^{- \alpha_2}|\xi_{i_0}|^{2+\ep}a_{i_0}^{2+\ep}\right). \end{equation}

The expression \eqref{jul2} also holds for $\tilde{\xi}$ in place of $\xi$; we denote the error term here by  $ \widetilde { \text{err}_2 }$. By the same reasoning, we can show that $\widetilde { \text{err}_2 }$ satisfies 
\eqref{err2-1}.

Take the expectation (with respect to $\xi_{i_0}$) of the right-hand side of \eqref{jul2} and subtract from it the expectation of  the corresponding formula (with respect to $ \tilde \xi_{i_0}$). By Condition {\bf C1},  $\xi_{i_0}$ and $\tilde \xi_{i_0} $ have matching first and second moments, and so the expectations of $\hat L_j$ ($j=1,2$) from the two formulae cancel each other out. Furthermore, $\hat L(0)$ is the same in both formulae. Thus, the only thing remaining 
 after the subtraction  are the error terms. Therefore, 
\[d_{i_0} \le  \left|\E_{\xi_{i_0}} \text{err}_2\right|  +   \left|\E _{ \tilde{\xi}_{i_0}} \widetilde { \text{err}_2 } \right|  =O(1) \tilde C\delta_n^{- \alpha_2}a_{i_0}^{2+\ep}\left(\E |\xi_{i_0}|^{2+\ep}+\E |\tilde\xi_{i_0}|^{2+\ep}\right)=O\left (\delta_n^{- \alpha_2}a_{i_0}^{2+\ep}\right ).
\]

Taking expectation with respect to the the other  variables (which we have conditioned on so far), we obtain $I_{i_0} =O\left (\delta_n^{- \alpha_2}a_{i_0}^{2+\ep}\right )$ for all $N_0\le i_0\le n$.

Now we treat the first few indices  $0\le i_0<N_0$, where $\xi_{i_0}$ may have non-zero mean.  Instead of using  \eqref{jul1} and \eqref{jul2}, we use the mean value theorem to get the rough bound
\begin{equation}
\hat L\left(a_{1, i_0}\xi_{i_0},\dots, a_{k, i_0}\xi_{i_0}\right) = \hat L(0) + O(k \norm{\triangledown \hat L}_{\infty}|\xi_{i_0}|\sum_{i = 1}^{k} |a_{i, i_0}|),\label{tempp}
\end{equation}
which by the same arguments as above gives $I_{i_0}=O\left (\delta_n^{- \alpha_2}a_{i_0}\right )$.

Since we assume $\E \tilde \xi_i=0$ for all $1 \le i \le n$, Condition {\bf C1} implies that
$|\E \xi_{i_0} | = O(1)$. But as $\Var \xi_{i_0} =1$, it follows that 
 $\E|\xi_{i_0}|= O(1)$. 
 
 As $k$ is constant and $\norm{\triangledown \hat L}_{\infty}\le \delta_{n}^{-\alpha_2}$, we have,   from \eqref{tempp}, that 
\begin{eqnarray}
d_{i_0} &=& O (k ||\triangledown \hat L||_{\infty}\sum_{i = 1}^{k} |a_{i, i_0}|)\left (\E|\xi_{i_0}|+ \E|\tilde \xi_{i_0}|\right ) \nonumber\\
&=& O(\delta_{n}^{-\alpha_2}\sum_{i = 1}^{k} |a_{i, i_0}|)=O(\delta_{n}^{-\alpha_2}\sum_{i = 1}^{k} |a_{i, i_0}|^{2})^{1/2} = O(\delta_{n}^{-\alpha_2}a_{i_0})\nonumber.
\end{eqnarray}

Notice that  by Condition {\bf C2} \eqref{cond-delocal}, $a_{i_0}=O(\sqrt k \delta_n^{\alpha_1}) = O(\delta_n^{\alpha_1})$ for all $i$. Furthermore, by the definition $\sum_{i=N_0}^{n} a_{i_0}^{2} = k=O(1)$. 
Thus, we have 
 $$I=O\left (\delta_n^{- \alpha_2}\sum_{i_0 = 0}^{n} a_{i_0}^{2+\ep} + \delta_n^{- \alpha_2}\sum_{i_0 = 0}^{N_0} a_{i_0} \right ) = O(\delta_n^{\alpha_1\ep -\alpha_2})=O(\delta_n^{\alpha_2}),$$
where in the last step we used the fact that $\alpha_2$ was set much smaller than $\alpha_1$. 
\end{proof}

\begin{proof} [Proof of Lemma \ref{logcomp}]
Let $\alpha_2$ be the constant in Lemma \ref{log-com33} and set $\alpha_0 := \frac{\alpha_2}{10}$. Let 

 $$\bar K(w_1,\dots, w_k) := K(w_1 + \frac{1}{2}\log |V(z_1)|,\dots, w_k+\frac{1}{2}\log |V(z_k)|)$$
 where we recall that 
$ V(z_j)  := \sum _{i=N_0}^{n}|\phi_i(z_j)|^{2}$.
\noindent We  have $\norm{\triangledown ^a \bar K}_\infty \le \delta_n^{-\alpha_0}$ for all $0\le \alpha\le 3$;  we aim to show

\begin{equation} \label{L1} \left |\E \bar K \left (\log\frac{|F_n(z_1)|}{\sqrt{V(z_1)}}, \dots, \log\frac{|F_n(z_k)|}{\sqrt{V(z_k)}}\right )-\E \bar K \left (\log\frac{|\tilde F_n(z_1)|}{\sqrt{V(z_1)}}, \dots, \log\frac{|\tilde F_n(z_k)|}{\sqrt{V(z_k)}}\right ) \right |= O\left (\delta_n^{\alpha_0}\right ).
 \end{equation} 

For  $M  := \log \left(\delta_n^{-3\alpha_0}\right)$, define   $$ \Omega_1 := \{(w_1,\dots, w_k)\in \mathbb{R}^k: \min_{i=1,\dots, k} w_i<-M\} $$ and $$ \Omega_2  := \{(w_1,\dots, w_k)\in \mathbb{R}^k: \min_{i=1,\dots, k} w_i>-M-1\}.$$  

By considering the real and imaginary parts of $\bar K$ separately, we can assume that $\bar K:\mathbb{R}^k\to \mathbb{R}$ .

Let $\psi:\mathbb{R}^k\to [0, 1]$ be a smooth function  supported in $\Omega_2$ such that 
 $\psi=1$ on the complement of $\Omega_1$ and $\norm{\triangledown^a\psi}_\infty=O(1)$ for all $0\le a\le 3$.
 As $M \ge 1$, it is easy to see that such a function exists. In particular, one can define 
  $\psi(x_1,\dots, x_k) = \rho(x_1)\dots\rho(x_k)$ where $\rho$ is a smooth function satisfying the corresponding properties on $\R$.

Let $\phi := 1 - \psi$, $K_1 := \bar K \phi$, and $K_2 := \bar K\psi$. Then by the definition  $\bar K = K_1 + K_2$. Furthermore,  both $K_1, K_2$ are smooth functions with $\text{supp } K_1 \subset\bar \Omega_1, \text{supp } K_2\subset\bar \Omega_2$ and $\norm{\triangledown^aK_i}_\infty= O\left (\delta_n^{- \alpha_0}\right )$ for $i=1, 2$ and $0\le a\le 3$.

We now show that the contribution from $K_1$ towards the right-hand side of \eqref{L1}  is negligible. Notice that
$$
\norm{K_1}_{\infty}\le \norm{\bar K}_\infty \le C' \delta_n^{-\alpha_0}. $$ 
This leads to setting $H_1(w_1,\dots,w_k) = C'\delta_n^{-\alpha_0}\phi(\log |w_1|,\dots, \log |w_k|)$. The function $H_1$ is a smooth function on $\R^{k}$ with the following properties 

\begin{itemize}
	\item $\ab{K_1(\log |w_1|,\dots,\log |w_k|)}\le H_1(w_1,\dots, w_k)$,
	
	\item  $\text{supp} (H_1)\subset \{(w_1,\dots, w_k)\in \mathbb{R}^k: \min_{i=1,\dots, k}|w_i|\le e^{-M}\}$,
	
	\item   $\norm{\triangledown^aH_1}_\infty=O\left (\delta_n^{-10 \alpha_0}\right )=O\left ( \delta^{-\alpha_2}\right )$ for all $0\le a\le 3$.
	
	\end{itemize}

\begin{remark} \label{derivative-R}  To verify the last property, notice that the support of $H_1$ is $\{(x, y): |x|\le e^{-M} \mbox{ or }  |y|\le e^{-M}\}$. Moreover, $H_1$ is a constant $C'\delta_n^{-\alpha_0}$ in the set $\{(x, y): |x|\le e^{-M-1} \mbox{ or }  |y|\le e^{-M-1}\}$ (because $\phi=1$ on the complement of $\Omega_2$). So we only need to consider  the derivatives of $H_1$ in the set $\{(x, y): |x|\le b \mbox{ or }  |y|\le e^{-M}\}\cap\{|x|\ge e^{-M-1}, |y|\ge e^{-M-1}\}$. On that set,  $x^{-1}$ and $ y^{-1}$ are bounded from above by $e^{M+1}$, which is significantly smaller 
than the bound. (We define $\alpha_0$ and $M$ with foresight so the claimed bound holds, with room to spare.) 
\end{remark}

Applying Lemma \ref{log-com33}, we obtain 
\begin{eqnarray}
\E \ab{K_1\left(\log \frac{|F_n(z_1)|}{\sqrt{V(z_1)}},\dots, \log \frac{|F_n(z_k)|}{\sqrt{V(z_k)}}\right)}&\le&\E {H_1\left( \frac{|F_n(z_1)|}{\sqrt{V(z_1)}},\dots, \frac{|F_n(z_k)|}{\sqrt{V(z_k)}}\right)}\nonumber\\
&\le& \E {H_1\left( \frac{|\tilde F_n(z_1)|}{\sqrt{V(z_1)}},\dots, \frac{|\tilde F_n(z_k)|}{\sqrt{V(z_k)}}\right)}+ C' \delta_n^{ \alpha_0}.\nonumber
\end{eqnarray}

Since $H_1(w_1, \dots, w_k) = 0$ if $(\log \ab{w_1},\dots, \log \ab{w_k})\notin \Omega_1$ 
and since the variables $\tilde \xi_i$ are Gaussian, we have
\begin{eqnarray}
\E {H_1\left( \frac{|\tilde F_n(z_1)|}{\sqrt{V(z_1)}},\dots, \frac{|\tilde F_n(z_k)|}{\sqrt{V(z_k)}}\right)}&\le& 
C' \delta_n^{-\alpha_0}\P \left(\exists i\in\{1,\dots, k\}:\frac{|\tilde F_n(z_i)|}{\sqrt{V(z_i)}}\le e^{-M} = \delta_n^{3\alpha_0}\right)\nonumber\\
&\le& C' \delta_n^{-\alpha_0}k\delta_n^{3\alpha_0}= O( \delta_n^{\alpha_0}) \nonumber.
\end{eqnarray}
Thus, $\E \ab{K_1\left(\log \frac{|F_n(z_1)|}{\sqrt{V(z_1)}},\dots, \log \frac{|F_n(z_k)|}{\sqrt{V(z_k)}}\right)}\le C' \delta_n^{\alpha_0}$. The same bound holds with $F_n$ replaced by $\tilde F_n$. To conclude the proof, we need to show that 
\[\ab{\E {K_2\left(\log \frac{|F_n(z_1)|}{\sqrt{V(z_1)}},\dots, \log \frac{|F_n(z_k)|}{\sqrt{V(z_k)}}\right)} - \E {K_2\left(\log \frac{|F_n(z_1)|}{\sqrt{V(z_1)}},\dots, \log \frac{|F_n(z_k)|}{\sqrt{V(z_k)}}\right)}}= O( \delta_n^{\alpha_0}).
\]
Define  $H_2(w_1,\dots,w_k) := K_2(\log |w_1|,\dots, \log |w_2|)$. Since $\text{supp }K_2\subset   \bar{\Omega}_2$, $$\text{supp }H_2\subset \{(w_1,\dots, w_k): \log |w_i|\ge -M-1, \forall i\} = \{(w_1,\dots,w_k): |w_i|\ge C'\delta_n^{3\alpha_0}, \forall i\}. $$

 Thus, $H_2$ is well-defined and smooth on $\mathbb{R}^k$. Furthermore, by the definition of $H_2$, it is not hard to check that  $\norm{\triangledown^aH_2}_\infty=O\left (\delta_n^{-10\alpha_0}\right )$ for all $0\le a\le 3$; see 
 Remark \ref{derivative-R}.

  Finally, by applying  Lemma \ref{log-com33}, we obtain
\begin{eqnarray}
&&\ab{\E {K_2\left(\log \frac{|F_n(z_1)|}{\sqrt{V(z_1)}},\dots, \log \frac{|F_n(z_k)|}{\sqrt{V(z_k)}}\right)} - \E {K_2\left(\log \frac{|F_n(z_1)|}{\sqrt{V(z_1)}},\dots, \log \frac{|F_n(z_k)|}{\sqrt{V(z_k)}}\right)}}\nonumber\\
&=&\ab{\E {H_2\left( \frac{|F_n(z_1)|}{\sqrt{V(z_1)}},\dots, \frac{|F_n(z_k)|}{\sqrt{V(z_k)}}\right)} - \E {H_2\left( \frac{|F_n(z_1)|}{\sqrt{V(z_1)}},\dots, \frac{|F_n(z_k)|}{\sqrt{V(z_k)}}\right)}}=O\left (\delta_n^{\alpha_0}\right) .\nonumber
\end{eqnarray}
This completes the proof.
\end{proof}

\subsection{Proof of Lemma \ref{halasz-inequality}}\label{proof_Halasz}
 By rescaling, we can assume that $a = n^{l}$. Thus, we need to estimate 
$\sup_Z  \P\left (\left |\sum_{j=1}^{n}a_j\ep_j- Z\right |\le 1 \right )$.

By Ess\'een's inequality \cite{esseen1966kolmogorov} (see also \cite[Lemma 7.17]{taovubook}), there is an absolute constant $c$ such that for any real number $Z$,
\begin{equation} \label{esseen1} 
\P\left (\left |\sum_{j=1}^{n}a_j\ep_j- Z\right |\le 1 \right ) \le c\int_{-1/2}^{1/2} |\phi(t)|dt 
\end{equation}
where 
$$\phi(t) = \E\exp\left (i2\pi t\sum_{j=1}^{n}a_j\ep_j\right ) = \prod _{j = 1}^{n}\E\exp\left (i2\pi t a_j\ep_j\right ) = \prod_{j=1}^{n} \cos(2\pi a_jt).$$

For every $x\in \R$, let $\norm{x}_{\R/\Z}:= \min\{|x-N|: N\in\Z\}$ be the distance from $x$ to the set of integers.
In the following lemma, we gather a few simple (and well known) facts concerning $\sin $ and $\cos$, whose proof is 
left as an exercise. 

\begin{lemma} \label{integernorm} We have 
	
	\begin{itemize}
		\item  $\sin \theta\ge 2 \theta/\pi$ for all $\theta\in [0, \pi/2]$;
		
\item $| \cos  x| \le 1-2\norm{x/\pi}_{\R/\Z}^{2} \le  \exp\left (-2\norm{x/\pi}^{2}_{\R/\Z}\right )$ for all $x\in \R$;

\item  $\cos (2 x)\ge 1 - 2\pi^{2}\norm{x/\pi}_{\R/\Z}^{2}$ for all $x\in \R$;

\item There is a constant $c>0$ such that for all
$T \ge 1$,  $\max \{ |\int_0^1 \sin Tx dx |,  |\int_0^1 \cos Tx dx | \} \le c/T $, \end{itemize} 
	
	\end{lemma}

By \eqref{esseen1} and  Fubini's Theorem,
\begin{equation} \label{E2} 
\P\left (\left |\sum_{j=1}^{n}a_j\ep_j- Z\right |\le 1 \right ) \le c\int_{-1/2}^{1/2} \exp\left (-2\sum_{j=1}^{n}\norm{2 a_j t}^{2}_{\R/\Z}\right )dt = 2c\int_{0}^{\infty} |A_x| e^{-2x} dx ,
\end{equation} 
where $A_x := \{t\in [-1/2, 1/2]: \sum_{j=1}^{n} \norm{2a_j t}^{2}_{\R/\Z}\le x\}$ and $|A_x|$ denotes the Lebesgue measure of $A_x$. 
We break the last integral in \eqref{E2} into two parts, $\int_{0}^{n/4\pi^{2}} |A_x| e^{-2x} dx $ and 
$\int_{n/4\pi^{2}}^{\infty} |A_x| e^{-2x} dx $. Since $|A_x| \le 1$ for all $x$, 

$$ \int_{n/4\pi^{2}}^{\infty} |A_x| e^{-2x} dx = e^{-\Omega (n)} = o( n^{-l}) $$ for any fixed $l$. Thus, this part is negligible and  it remains to show 
\begin{equation} \label{E3}  \int_{0}^{n/4\pi^{2}} |A_x| e^{-2x} dx = O(n^{-l}). \end{equation}

Let us now bound the measure of the set $A_{n/4\pi^{2}}$. 
By Lemma \ref{integernorm},  
$$A_{n/4\pi^{2}}\subset A:=\{t\in [-1/2, 1/2]: \sum_{j=1}^{n}\cos(4\pi a_j t)\ge n/2\}. $$ 
To bound $|A|$,  we first notice that 
\begin{eqnarray}
\int_{-1/2}^{1/2} \left (\sum_{j=1}^{n}\cos(4\pi a_j t)\right )^{2l} dt &\le& \int_{-1/2}^{1/2} \left (\sum_{j=1}^{n}\left (e^{i4\pi a_j t} + e^{-i4\pi a_j t}\right )\right )^{2l} dt\nonumber\\
&=& \sum_{s_1, \dots, s_{2l} = \pm 1} \sum_{j_1, \dots, j_{2l}\le n} \int_{-1/2}^{1/2} e^{i4\pi t\sum_{h=1}^{2l}s_h a_{j_h}}dt \nonumber.
\end{eqnarray}
Recall the hypothesis of the lemma that for any 
two different multi-sets $\{i_1, \dots, i_{l'}\}$ and $\{j_1, \dots, j_{l''}\}$ where $l'+ l''\le 2l$, it holds that $|a_{i_1}+\dots + a_{i_{l'}} - a_{j_1}-\dots - a_{j_{l''}}|\ge a=n^{l}$. Thus, for each $s_1, \dots, s_{2l}=\pm 1$ and $j_1, \dots, j_{2l}\le n$, consider the multi-sets $S_1=\{j_h: s_h=1\}$ and $S_2 = \{j_h: s_h = -1\}$. If $S_1 \neq S_2$ then $|\sum_{h=1}^{2l}s_h a_{j_h}|\ge n^{l}$. In this case, the corresponding term in the above double sum is of the form $\int_{-1/2}^{1/2} e^{itT}dt$ for some $|T|\ge 2 n^{l}$. By Lemma \ref{integernorm}, we have
$$\int_{-1/2}^{1/2} e^{i4\pi t\sum_{h=1}^{2l}s_h a_{j_h}}dt= O(n^{-l}), \quad \text{if $S_1\neq S_2$}.$$
If $S_1 = S_2$, then $|a_{i_1}+\dots + a_{i_{l'}} - a_{j_1}-\dots - a_{j_{l''}}| = 0$ and the corresponding integral is $1$.  The number of terms in the double sum with $S_1 = S_2$ is at most $2^{2l} n^{l}=O(n^{l})$ while the total number of terms is at most $2^{2l}n^{2l}=O(n^{2l})$. Putting these cases together, we obtain  
$$\int_{-1/2}^{1/2} \left (\sum_{j=1}^{n}\cos(4\pi a_j t)\right )^{2l} dt = O\left (n^{l}+ n^{2l}n^{-l}\right )=O(n^{l}).$$

Hence, $ |A|=O\left(n^{-l}\right )$ by Markov's inequality. This implies $|A_{n/4\pi^{2}}|= O(n^{-l})$, which, in turn, yields \eqref{E3}, completing the proof.

\subsection{Proof of the second Jensen's inequality \eqref{jensenbound}}\label{proof_jensen}
By setting $g(w) = f\left (R(w+z)\right )$ and prove the corresponding inequality for $g$, it suffices to assume that $z = 0$ and $R=1$. Let $a_1, \dots, a_N$ be the zeros of $f$ in $\bar B(0, r)$. For each $a$ inside the unit disk $D$, consider the map 
$$T_a(w) = \frac{w-a}{\bar a w - 1}.$$

For $|a|\le r$ and $|w|\le r$, one can show by algebraic manipulation that 
$$|T_a(w)|\le \frac{2r}{1+r^{2}}<1.$$

Moreover, for all $|a|<1$ and $|w|=1$, we have
$$|T_a(w)| = |\bar w|\left |\frac{w-a}{\bar a w - 1}\right | = \left |\frac{1-a \bar w}{\bar a w - 1}\right |=1.$$

Let $h(w) = \frac{f(w)}{\prod_{k=1}^{N} T_{a_k}(w)}$. Then $h$ is an analytic function on $D$. By maximum principle, we have for every $w_0\in rD$,
\begin{eqnarray}
\frac{|f(w_0)|(1+r^{2})^{N}}{(2r)^{N}}\le \max_{w\in rD} |h(w)| \le \max_{w\in D} |h(w)| = \max_{w\in \partial D} |h(w)|=\max_{w\in \partial D} |f(w)|= M\nonumber.
\end{eqnarray}

Thus, $N\le \frac{\log \frac{M}{|f(w_0)|}}{\log\frac{1+r^{2}}{2r}}$ for all $w_0\in rD$, completing the proof.

\end{document}